\documentclass[10pt]{amsart}
\usepackage{euscript}
\usepackage{amssymb}
\usepackage{amsmath}
\usepackage{enumitem}
\usepackage{epic}
\usepackage{graphics}
\usepackage{epsfig}
\usepackage{color}
\usepackage{quiver}

\usepackage{amscd,euscript}
\usepackage[frame,cmtip,curve,arrow,matrix,line,graph]{xy}

\usepackage{setspace}

\usepackage{tikz}
\usepackage{tikz-cd}
\tikzset{
	cross/.pic = {
		\draw[rotate = 45] (-#1,0) -- (#1,0);
		\draw[rotate = 45] (0,-#1) -- (0, #1);
	}
}

\usepackage[breaklinks,colorlinks,citecolor=teal,linkcolor=teal,urlcolor=teal,pagebackref,hyperindex]{hyperref}

\usepackage[margin=1.2in]{geometry}

\newtheoremstyle{my}{1.5em}{0.5em}{\em}{}{\sc}{.}{0.5em}{}

\newtheorem{theorem}{Theorem}[section]
\newtheorem*{theorem*}{Theorem}
\newtheorem{thm}{Theorem}[section]

\newtheorem*{Theorem*}{Theorem}

\newtheorem{corollary}[thm]{Corollary}
\newtheorem*{corollary*}{Corollary}

\newtheorem{prop}[thm]{Proposition}
\newtheorem*{prop*}{Proposition}

\newtheorem*{conjecture*}{Conjecture}

\newtheorem*{question*}{Question}
\newtheorem{defn}[thm]{Definition}

\newtheorem*{definitions*}{Definitions}

\newtheorem*{rem*}{Remark}

\newtheorem*{remark*}{Remark}

\newtheorem*{remarks*}{Remarks}
\newtheorem*{example*}{Example}
\newtheorem{Example}[thm]{Example}
\newtheorem*{examples*}{Examples}

\newtheorem*{convention*}{Convention}
\newtheorem*{conventions*}{Conventions}

\newtheorem*{exercise*}{Exercise}
\newtheorem*{bibliographical-note*}{Bibliographical note}

\newtheorem{lemma}[thm]{Lemma}

\newtheorem{remark}[thm]{Remark}

\newcommand{\scrH}{\EuScript{H}}

\newcommand{\supp}{\textnormal{supp}}
\newcommand{\val}{\textnormal{val}}
\newcommand{\ival}{\textnormal{ival}}

\newcommand{\even}{\textnormal{even}}

\newcommand{\reg}{\textnormal{reg}}

\newcommand{\scrC}{\EuScript{C}}

\newcommand{\calO}{\mathcal{O}}

\newcommand{\frI}{\mathfrak{I}}

\newcommand{\Vol}{\textnormal{-vol}}

\newcommand{\bR}{\mathbb{R}}
\newcommand{\bZ}{\mathbb{Z}}
\newcommand{\bQ}{\mathbb{Q}}
\newcommand{\bC}{\mathbb{C}}
\newcommand{\bN}{\mathbb{N}}
\newcommand{\bP}{\mathbb{P}}

\newcommand{\fprod}[1]{\sideset{}{_{#1}}\prod}

\newcommand{\Sym}{\mathrm{Sym}}

 \newcommand{\edge}{\mathrm{edg}}
 \newcommand{\leaf}{\mathrm{lf}}
\newcommand{\id}{\mathrm{id}}

\renewcommand{\ker}{\mathrm{ker}}

\newcommand{\Hom}{\mathrm{Hom}}
\newcommand{\RHom}{\mathrm{R}{\mathcal H}{\mathit om}}

\def\lra#1{\overset{#1}{\longrightarrow}}
\def\mod#1{#1\text{-mod}}

\renewcommand{\hbar}{\overline{\frak{h}}}
\newcommand{\Aut}{\mathrm{Aut}}

\newcommand{\scrM}{\EuScript{M}}

\newcommand{\scrF}{\EuScript{F}}

\newcommand{\scrI}{\EuScript{I}}

\newcommand{\scrG}{\EuScript{G}}

 \numberwithin{equation}{section}

\renewcommand{\leq}{\leqslant}
\renewcommand{\geq}{\geqslant}

\newcommand{\C}{\mathbb C}

\newcommand{\bT}{\mathbb{T}}

\def\frM{\mathfrak{M}}

\newcommand{\ev}{\operatorname{ev}}

\newcommand{\Coh}{\operatorname{Coh}}

\newcommand{\AbC}{\operatorname{AbC}}

\newcommand{\Spec}{\operatorname{Spec}}
\newcommand{\Proj}{\operatorname{Proj}}

\newcommand{\ver}{\mathrm{ver}}

\newcommand{\sing}{\mathrm{sing}}
\newcommand{\BL}{\mathrm{Bl}}

\newcommand{\virt}{\operatorname{vir}}

\newcommand{\CVect}{\bC\textnormal{-Vect}}

\renewcommand{\sc}{\operatorname{sc}}

\title[Bounds for Genus Zero GW Invariants.]{Bounds for Genus Zero Gromov-Witten Invariants.}
\author{Mark McLean}
\address{Mark McLean, Stony Brook}
\email{mark.mclean@stonybrook.edu}

\begin{document}

\begin{abstract}
	We give a bound for the norm of each primary genus zero Gromov-Witten invariant in any smooth projective variety.
	This bound depends only on the maximum degree of its defining polynomials, the norms of the differential forms representing each insertion, the
	number of marked points and the degree of the curves.
	These bounds grow factorially in the number of marked points and a multiple of the degree.
	This implies that the Borel transform of the genus zero primary generating function converges absolutely in a particular translation of the ample cone.

	To prove these bounds, we start with Siebert's formula, expressing genus zero Gromov-Witten invariants in terms of the Segre class of the normal cone of the underlying moduli space inside the space of curves mapping to projective space and the Chern classes of the bundle whose fiber over a curve is the space of sections of the pulled back normal bundle.
	We then use this description to bound each primary Gromov-Witten invariant in terms of an intermediate count of curves called the $D$-volume.
	A deformation to the normal cone argument is then used to compare such an intermediate count with one in projective space.
\end{abstract}

\maketitle

{ \setcounter{tocdepth}{4}
	\setcounter{secnumdepth}{4}
	\hypersetup{linkcolor=black}
	\tableofcontents
}

\section{Introduction}

We establish upper bounds for primary genus-zero Gromov-Witten invariants of smooth projective varieties.
We will prove the following main theorem:

\begin{theorem} \label{maintheorem}
	Let $X \subset \bP^N$ be a smooth projective subvariety cut out by polynomials of degree at most $\delta_X \in \bN$.
	Let $d \in H_2(X;\bZ)$ be a curve class and let $\deg(d) \in \bN$ be its image in $H_2(\bP^N;\bZ) \cong \bZ$.
	Then for each tuple $\alpha_1,\cdots,\alpha_m$ of homogenous closed differential forms on $X$ we have the following bound of primary genus $0$ Gromov-Witten invariants:
	\begin{equation} \label{eqnmaintheoremineqluaty}
		\left|\langle [\alpha_1],\cdots,[\alpha_m] \rangle^X_{m,d} \right| \leq
		m!\left(4952{N+(4\delta_X)^{2^{N-1}}+1 \choose N}^6(\deg(d)+1)\right)!
		\prod_{j=1}^m |\alpha_j|_{C^0,\omega}
	\end{equation}
	where $|\alpha_j|_{C^0,\omega}$ is the $C^0$ norm of $\alpha_j$ with respect to the metric on $X$ coming from the Fubini-Study form $\omega$\footnote{We use the convention that the Fubini-Study form is Poincar\'{e} dual to the hyperplane.} for each $j$.
\end{theorem}

This is the first result of its kind providing a bound for all primary genus zero Gromov-Witten invariants in terms of the geometry of the target projective variety.
These bounds are very, very far from being optimal even in elementary examples \footnote{These inequalities say that there are at most $10$ to the power of $600$ million conics passing through five general points in the plane.}. For instance, we do not prove that the primary generating function converges, which is conjectured in \cite[Section 4.3]{kontsevich1994gromov}, \cite[Conjecture 8.1.2]{coxkatz1999mirror} and \cite[Conjecture 1]{zinger2014genus} and known in many examples (See \cite{iritani2007convergence}, \cite{zinger2014genus}, \cite{coatesiritani} and \cite{koto2022convergence}).
It is possible using the methods of this paper to improve the inequality \eqref{eqnmaintheoremineqluaty}, but we have chosen brevity instead.
A very different possible approach to giving bounds in the symplectic setting is contained in \cite{swaminathan2021quantitative}.

We have the following consequence of the theorem above. Let
\[
	H^{>2,\even}(X) := \bigoplus_{j > 1} H^{2j}(X; \bR).
\]
Let $\textnormal{Amp}(X) \subset H^2(X; \bR)$ be the ample cone of $X$.
We also let $H_2^+(X) \subset H_2(X;\bZ)$ be the subset of classes represented by proper curves mapping to $X$.
We write $\Gamma$ for the $\Gamma$-function.

\begin{corollary}
	Let $\Omega_X := 19808{N + (4\delta_X)^{2^{N-1}+1} \choose N}^6[\omega|_X] \in \textnormal{Amp}(X)$.
	Then the Borel transform of the primary generating function:
	\[
		B : (\Omega_X + \textnormal{Amp}(X)) \times H^{>2,\even}(X) \times \bR_{>0} \to \bR,
	\]
	\begin{equation} \label{eqnborelsumconvergeega}
		B(\eta, t^+,s) := \sum_{m \in \bN, \ d \in H_2^+(X)}\frac{\exp(d \cdot \eta)}{m! \, \Gamma(d \cdot \eta+1)} \langle t^+, \cdots, t^+ \rangle^X_{m,d} s^{d \cdot \eta}
	\end{equation}
	converges absolutely and hence is analytic.
\end{corollary}
\begin{proof}
	We have
	\begin{equation} \label{eqncdotedt}
		\frac{1}{2}(d \cdot \eta) \geq 4952{N+(4\delta_X)^{2^{N-1}}+1 \choose N}^6(\deg(d)+3), \quad d \cdot \eta \geq (N+1)\deg(d) + N-3
	\end{equation}
	for each  $d \in H_2^+(X) - 0$ and each $\eta \in \Omega_X + \textnormal{Amp}(X)$.
	Since $t^+$ is a sum of classes of degree higher than $2$ and by Lemma \ref{lemmadimensionofX}, we have
	\begin{equation} \label{eqneafet49t4}
		m \leq (N+1)\deg(d) + N-3 \overset{\textnormal{\eqref{eqncdotedt}}}{\leq} d \cdot \eta
	\end{equation}
	if $\langle t^+, \cdots, t^+ \rangle^X_{m,d} \neq 0$.
	Hence, by Theorem \ref{maintheorem} combined with the fact that the sum of the $|\cdot|_{C^0,\omega}$-norms of the homogenous components of a form on $X$ is at most $2^{2\dim(X)}$ times its $|\cdot|_{C^0,\omega}$-norm, we get
	\begin{equation} \label{eqnboundgwinvariant}
		\begin{aligned}
			\sum_{m \in \bN} \left| \frac{\exp(d \cdot \eta)}{m! \, \Gamma(d \cdot \eta+1)} \langle t^+, \cdots, t^+ \rangle^X_{m,d} s^{d \cdot \eta}  \right|
			 & \leq \sum_{m=0}^{\lfloor d \cdot \eta \rfloor} \frac{m!\Gamma(\frac{1}{2}d \cdot \eta+1)(e s(2^{2\dim(X)}|t^+|_{C^0,\omega}+1))^{d \cdot \eta}}{m!\Gamma((d \cdot \eta+1))} \\
			 & \leq (d \cdot \eta+1) \frac{((e s(2^{2\dim(X)}|t^+|_{C^0,\omega}+1))^2)^{\frac{1}{2}d \cdot \eta}}{\Gamma(\frac{1}{2}(d \cdot \eta))}
		\end{aligned}
	\end{equation}
	for each $(\eta,t^+,s) \in (\Omega_X + \textnormal{Amp}(X)) \times H_2^{>2,\even}(X) \times \bR$, $m \in \bN$ and $d \in H_2^+(X) - 0$.
	So the sum \eqref{eqnborelsumconvergeega} converges absolutely.
\end{proof}
It would be nice if the Legendre transform in the $s$-variable for $B$ existed since this would produce an analytic function whose asymptotic power series at $0$ would be equal to the primary generating function. Such an analytic function would satisfy the WDVV equation for instance.

We will now give a brief sketch of the proof of Theorem \ref{maintheorem}.
The key idea is to try to compare Gromov-Witten invariants of $X$ with those of projective space.
For each variety $Y$, we let $H_2^+(Y) \subset H_2(Y;\bZ)$ be the subset of homology classes represented by morphisms of curves to $Y$.
We let $Y_{m,d}$ be the moduli space
of stable genus zero nodal curves with $m$ marked points representing $d$ for any $m \in \bN$ and $d \in H_2^+(Y)$.
For any smooth hypersurface $E \subset Y$, we let $Y_{m+E!,d} \subset Y_{m+E\cdot d,d}$ be the closure of the space of curves $u$ transverse to $E$ with $u^{-1}(E)$ the union of the last $E \cdot d$ marked points.
For any vector bundle $\pi_W : W \to Y$, we define $W_{m,d} := W_{m,\iota_*d}$ where $\iota : Y \hookrightarrow W$ is the zero section inclusion map.
We let $W_{m+E!,d} := W_{m+(\pi_W^{-1}E)!,\iota_*d} \subset W_{m+E \cdot d,d}$.
Note that these moduli spaces have some similarities with those in \cite{cieliebak5symplectic}.

In order to get bounds for genus zero Gromov-Witten invariants we will use a description of the virtual fundamental class by Siebert (\cite{siebert2004virtual}) as a homogenous component of the cap product of the Segre class of the normal cone of $X_{m,d}$ inside $\bP^N_{m,{\iota_X}_*d}$
with the total Chern class of $NX_{m,d}$ where $NX \to X$ is the normal bundle of $X$ in $\bP^N$
and $\iota_X : X \hookrightarrow \bP^N$ is the inclusion map.
As a result, we first find bounds for integrals of the form:
\begin{equation} \label{eqnintegralcone}
	\langle [\alpha_1],\cdots,[\alpha_m] \rangle^{s(C_\bullet)}_{m,d} := \int_{s(C_{m,d})} \prod_{i=1}^m \ev_i^* \alpha_j
\end{equation}
where $\alpha_1,\cdots,\alpha_m$ are homogenous de Rham forms on $X$, $C_{m,d} \to X_{m,d}$ is a subcone of $V_{m,d}$ for some vector bundle $V$ which is a pullback of a cone in $V_{0,d}$ under the forgetful map and $s(C_{m,d})$ is its Segre class.
This is done in terms of the \emph{$D$-volume} $D\Vol(C_{m,d})$ (Definition \ref{defnvolumeformultiplecones}).
To define the $D$-volume for a smooth divisor $D$ we first define $C_{m+D!,d}$ to be the pullback of the cone $C_{m,d}$ to $X_{m+D!,d}$ under the forgetful map forgetting the last $D \cdot d$ marked points.
Let $\ev : X_{m+D!,d} \to X^{m'}$ and
$\ev : C_{m+D!,d} \to V^{\boxtimes m'}$ be the evaluation maps where $V^{\boxtimes m'} \to X^{m'}$ is the product vector bundle and $m' = m+D \cdot d$.
Let $P_{V^{\boxtimes m'}} := \bP(V^{\boxtimes m'} \oplus \bC)$ be the projective compactification of $V^{\boxtimes m'}$.
Suppose $\pi_{P_{V^{\boxtimes m'}}} : P_{V^{\boxtimes m'}} \to X^{m'}$ is the natural projection map and $p_i : X^{m'} \to X$ the $i$th projection map for each $i$.
Let $\widetilde{\ev} : C_{m+D!,d} \to \ev^*V^{\boxtimes m'}$ be the natural evaluation map and let $P\ev : \ev^*P_{V^{\boxtimes m'}} \to P_{V^{\boxtimes m'}}$ be the natural projection map.
We let $\overline{\widetilde{\ev}(C_{m+D!,d})}$ be the closure of the image $\widetilde{\ev}(C_{m+D!,d})$ in $\ev^*P_{V^{\boxtimes m'}}$.
We let $D_{\sum m', V} := \pi_{P_{V^{\boxtimes m'}}}^*\sum_{i=1}^{m'} p_i^*D + \infty_{V^{\boxtimes m}}$ be a divisor in $P_{V^{\boxtimes m'}}$
where $\infty_{V^{\boxtimes m'}}$ is the divisor at infinity.
The $D$-volume is then defined to be:
\begin{equation} \label{exnqdet}
	D\Vol(C_{m,d}) := \frac{1}{m!(D\cdot d)!} \int_{P\ev_*[\overline{\widetilde{\ev}(C_{m+D!,d})}]} \exp(D_{\sum m',V}).
\end{equation}
More generally we define $D\Vol(C^i_{m,d})_{i \in I}$ for multiple cones $C_{m,d}$, $i \in I$ using a fiber product of bundles of the form $P_{V^{\boxtimes m'}}$ over $X^{m'}$
(See Definition \ref{defnvolumeformultiplecones}).

Next, we give give bounds for the norm of the integral \eqref{eqnintegralcone} in terms of $D\Vol(C_{m,d})$
(Proposition \ref{propboundsforsegreclass})
and conversely bounds for $D\Vol(C_{m,d})$ in terms of the norm of integral \eqref{eqnintegralcone} for certain $D$ (Proposition \ref{propsegreboundsforDvolume}).
Let us explain the key idea behind the proof of Proposition \ref{propboundsforsegreclass} in the special case where $m=1$, $X=Y \times \bP^1$, $d=[\textnormal{pt} \times \bP^1]$ and we have a single cone $C$ inside a vector bundle $V_{h,d}$ over $X$ where $V$ is the pullback of a vector bundle over $Y$. In this case, $X_{m,d}=X$ and $V_{m,d}=V$.
Let $\overline{C}$ be the closure of $C$ inside $\bP(V\oplus \bC)$. Let $\infty_V = \bP(V \oplus 0)$ be the divisor at infinity in $\bP(V\oplus \C)$ and let $\pi : \bP(V \oplus\C) \to X$ be the natural projection map.
Suppose $D$ is far inside the ample cone of $X$ and linearly equivalent to $4D'$ for some $D'$. Let $\omega'$ and $\omega_C$ be K\"{a}hler forms Poincar\'{e} dual to $D'$ and $2\pi^*D' + \infty_V$ respectively.
The norm of \eqref{eqnintegralcone} in our special case is
less than or equal to
\begin{equation}
	\sum_r \left|\int_{\overline{C}} \pi^*\alpha_1 \cup \infty_V^r \right| = \sum_{r=1}^{\dim(C)} \left|\int_{\overline{C}} \pi^*\alpha \cup ((2\pi^*D' + \infty_V) - (2\pi^*D'))^r\right|
\end{equation}
\begin{equation}
	\leq \sum_{r=1}^{\dim(C)} \sum_{k=1}^r {r \choose k} \left|\int_{\overline{C}} \pi^*\alpha_1 \cup (\omega_C)^{r-k} \cup (2\pi^*\omega')^k\right|
\end{equation}
\begin{equation}
	\overset{\textnormal{Remark \ref{remarkcomparisonofnorms}}}{\leq}
	\sum_{r=1}^{\dim(C)} \sum_{k=1}^r {r \choose k}
	\left| \pi^*\alpha_1 \cup \omega_C^{r-k} \cup (\pi^*\omega_{D'})^k \right|_{C^0,\omega_C+2\pi^*\omega'} \int_{\overline{C}} \exp(\omega_C + 2\pi^*\omega')
\end{equation}
\begin{equation}
	\overset{\textnormal{Lemmas \ref{lemmaproductnorms} and \ref{lemmaaboutpositivedefiniteforms}}}{\leq} C\left| \alpha_1 \right|_{C^0,\omega}
	\int_{\overline{C}} \exp(\pi^* D + \infty_V)
	\leq C \left| \alpha_1 \right|_{C^0,\omega} D\Vol(C_{h,d})
\end{equation}
where
\begin{equation}
	C = \sum_{r=1}^{\dim(C)}\sum_{k=1}^r {r \choose k} 2^{2\dim(X)} {2\dim(X) \choose 2}^k
	(r-k)! {\dim(C) \choose r-k} \dim(X)^k.
\end{equation}
In order for the proof of Proposition \ref{propboundsforsegreclass} to work in general, we need to embed our cone $C_{m+D!,d}$ in $\ev^* V^{\boxtimes m'}$ where $\ev : X_{m+D!,d} \to X^{m'}$ is the evaluation map.
Lemma \ref{lemmainjectivefromvlarge} does this showing that the natural map $V_{m+D!,d} \to \ev^*V^{\boxtimes m'}$ is injective for appropriate $D$. This lemma only works for moduli spaces of the form $X_{m+D!,d}$, rather than $X_{m,d}$, since there are many marked points on each irreducible component of each curve.


In the rest of Section \ref{sectionsomeboundsfordvolumes} we show that the $D$-volume satisfies some other useful technical properties.
One of these technical theorems allows us to bound the $D$-volumes of cones of the form $V_{m,d}$ for a convex vector bundle $V$ using curves in $P_V$ (Proposition \ref{propcurvesincompatificationsofconvexvecctorbundles}). This is important because the total Chern class of $NX_{m,d}$ appears in Siebert's description of the virtual fundamental class.

Now, one important cone is the image $C^{X/\bP^N}_{m,d}$ in $NX_{m,d}$ of the normal cone of $X_{m,d}$ inside $\bP^N_{m,(\iota_X)_*d}$.
In Section \ref{boundsfornormalcones}
we bound $D|_X\Vol(C^{X/\bP^N}_{m,d},(C^i_{m,d}|_{X_{m,d}})_{i \in I})$
by $D\Vol(C^i_{m,d})_{i \in I}$ where $C^i_{m,d}$, $i \in I$ are certain cones over $\bP^N_{m,(\iota_X)_*d}$ and $D$ is a generic divisor in $\bP^N$ of sufficiently high degree (See Proposition \ref{propositionboundsfornormalcones}).
Let us illustrate the proof of Proposition \ref{propositionboundsfornormalcones} in the case where $X = \bP^1$ is a line inside $\bP^2$, $m = 0$, $d = [X]$, $D$ is a generic degree $2$ hypersurface and $I$ is a singleton set with the only cone being $0$.
In this case, $X_{0,h+D|_X!,d}$ consists of two points. Its normal cone inside $\bP^2_{0,h+D!,(\iota_X)_*d}$ is $X_{0,h+D|_X!,d} \times \bC^2$. Hence, $D|_X\Vol(C^{X/\bP^2}_{h,d}) = \frac{1}{2}$.
Now, $\bP^2_{0,h+D!,(\iota_X)_*d} = D \times D \cong \bP^1 \times \bP^1$ since each curve is determined by its intersection with $D \cong \bP^1$ and these intersection points are labelled.
The image of this space under the evaluation map is the product embedding $D \times D \subset \bP^2 \times \bP^2$. Hence, $D\Vol(\bP^2_{h,(\iota_X)_*d}(0))$ is one-quarter of the self intersection of this image which is $4$.
Let us describe the geometric mechanism demonstrating $D|_X\Vol(C^{X/\bP^2}_{h,d}) \leq D\Vol(\bP^2_{h,(\iota_X)_*d})(0)$ in our special case.
Let $\chi$ be the blowup of $0 \times X$ inside $\bP^1 \times \bP^2$, let $E$ be its exceptional divisor and $\widetilde{D}$ the proper transform of $\bP^1 \times D$ in $\chi$.
Let $\chi_2 := \chi \times_{\bP^1} \chi$
and let $\Pi : \chi_2 \to \bP^1$ be the corresponding projection map.
Define $\pi_k : \chi_2 \to \chi$, $k=1,2$ to be the projection maps to each factor.
Let $D_2 := \widetilde{D} \times_{\bP^1} \widetilde{D} \subset \chi_2$.
In the proof Proposition \ref{propositionboundsfornormalcones}, $D_2$ is the image under the evaluation map of a moduli space of curves.
Define $D_E$ to be $D_2 \cap E^2$ minus the preimage of the diagonal in $X \times X$ under the blowdown map.
Then since $\widetilde{D}-E$ is nef, we have
\begin{equation}
	D|_X\Vol(C^{X/\bP^2}_{0,d}) = \frac{1}{2} \int_{\bP^2 \sqcup \bP^2} \exp(\bP^1 \sqcup \bP^1)
\end{equation}
\begin{equation}
	\leq \frac{1}{2} \int_{(\bP^1)^2 \sqcup (\bP^1)^2} \exp((\infty \times \bP^1 + \bP^1 \times \infty) \sqcup (\infty \times \bP^1 + \bP^1 \times \infty))
\end{equation}
\begin{equation}
	= \frac{1}{2} \int_{D_E} \exp(\pi_1^* (\widetilde{D}-E) + \pi_2^* (\widetilde{D}-E))
\end{equation}
\begin{equation}
	\leq \frac{1}{2} \int_{\Pi^{-1}(0) \cap D_2} \exp(\pi_1^* (\widetilde{D}-E) + \pi_2^* (\widetilde{D}-E))
\end{equation}
\begin{equation}
	= \frac{1}{2} \int_{\Pi^{-1}(1) \cap D_2} \exp(\pi_1^* (\widetilde{D}-E) + \pi_2^* (\widetilde{D}-E))
	=  D\Vol(\bP^2_{0,(\iota_X)_*d}(0)).
\end{equation}


Using Siebert's description of the virtual fundamental class we next describe the virtual fundamental class in terms of Segre classes of the cones $NX_{m,d}$
and $C^{X/\bP^N}_{m,d}$. Using this we then give upper bounds for norms Gromov-Witten invariants in terms of $D$-volumes of these cones (Section \ref{vfcconnection}).
In fact, we use cones coming from vector bundles $T$ and $W$ where there exists a short exact sequence of convex vector bundles \begin{equation} \label{eqnsaintroses}
	0 \to NX \to T \to W \to 0
\end{equation}
where $T$, later on, is the restriction of a sum of line bundles on $\bP^N$. This is because the tools above compute Segre classes rather than Chern classes and
the total Chern class of $NX_{m,d}$ is $c(T_{m,d}) \cap s(W_{m,d})$ and the total Chern class $c(T_{m,d})$ can be computed more easily since it extends to a sum of line bundles over $\bP^N$.

In Section \ref{sectionboundsforproductsofprojhectve} we give explicit upper bounds for Gromov-Witten invariants twisted by sums of line bundles over products of projective space.
This in turn allows us to give upper bounds for the $D$-volumes involving the bundle $T_{m,d}$. Since these bounds only need to be very coarse we just use the localization formula \cite{kontsevich1995enumeration} directly without needing any more sophisticated arguments (E.g. we don't need to work as hard as \cite{zinger2014genus}).

Finally, in Section \ref{sectionmainargument} we assemble all of these ingredients together.
Here is also where we use the constant $\delta_X$. Such a constant can be used to bound the Castelnuovo-Mumford regularity of the defining ideal of $X$ as well as its square.
The bound is needed to embed certain projectivizations of vector bundles over $X$ into products of projective space of bounded dimension and degree.
It also allows us to construct the short exact sequence \eqref{eqnsaintroses} above.
Once we have this data we can use the theorems involving $D$-volumes above to bound the norm of our Gromov-Witten invariants (See Theorem \ref{theoremproperestimates}).

Conjecturally, bounds for genus zero Gromov-Witten invariants should involve an exponential of a multiple of the degree, rather than a multiple of the degree factorial
(\cite[Section 4.3]{kontsevich1994gromov}, \cite[Conjecture 8.1.2]{coxkatz1999mirror} and \cite[Conjecture 1]{zinger2014genus}).
It is the techniques used to prove Propositions \ref{propboundsforsegreclass} and \ref{propositionboundsfornormalcones} which are largely responsible for this factorial, rather than an exponential bound, as seen in the calculations above.
At this time the author has no idea how to improve them to get the conjectured exponential bound.

\bigskip

{\bf Acknowledgements:} The author is grateful to Hiroshi Iritani, Eric Riedl, Bernd Siebert, Ivan Smith, Jason Starr, Mohan Swaminathan and Aleksey Zinger for helpful discussions.
This paper is supported by NSF grant number DMS-2203308.

\section{Notation and Conventions} \label{sectionconventions}

The natural numbers $\bN$ start from $0$. We write $\bN_{>0}$ if we wish to start from $1$.
By \emph{stack} we will mean a separated Deligne Mumford stack of finite type over $\bC$ with quasi-projective coarse moduli space and with the fppf topology unless stated otherwise.
The only exception to this rule is when we refer to the moduli stack of genus zero pre-stable curves $\frM_{m}$ (Section \ref{sectioncomprisinvfc}).
If $(Y_i)_{i \in I}$ is a collection of stacks over a stack $X$ indexed by a set $I$, we define $$\fprod{X}_{i \in I}Y_i$$ to be their fiber product over $X$.

\begin{defn} \label{defnmodulieg}
	Let $Y$ be a smooth quasi-projective variety.
	We define $H_2^+(Y) \subset H_2(Y;\bZ)$ to be the submonoid of classes represented by morphisms from curves to $Y$.
	For each $m \in \bN$ and $d \in H_2^+(Y)$,
	we let $Y_{m,d}$ be the moduli stack
	of stable genus zero curves with $m$ marked points
	representing $d$.
	We let $\ev_i : Y_{m,d} \to Y$
	be the evaluation map at the $i$th marked point for each $i = 1,\cdots,m$
	and $\ev := \prod_{i=1}^m \ev_i : Y_{m,d} \to Y^m$.
	We let $\pi_j : Y_{m,d} \to Y_{m-1,d}$ be the $j$th forgetful map for each $j=1,\cdots,m$.
	Let $[Y_{m,d}]^{\virt}$ be the virtual fundamental class
	(See \cite{BehrendFantechinormalcone} and \cite{BehrendGW}).
\end{defn}

For any stack $M$, $CH_*(M)$ will be the Chow group
of $M$ with \emph{rational coefficient}s as defined in
\cite{gillet1984intersection} and \cite{vistoli1989intersection}.
So Chow cycles will be rational linear combinations of integral closed substacks.
Now the stack $M$ has an underlying coarse moduli space,
which is a variety $\underline{M}$.
Such a variety has an analytic topology since it is of finite type over $\bC$.

All singular (co)homology groups have rational coefficients unless stated otherwise.
In other words, $H^*(A) = H^*(A;\bQ)$ and $H_*(A) = H_*(A;\bQ)$ for any topological space $A$.
We define $H^{\even}(A)$ to be the even dimensional cohomology of $A$ with coefficients in $\bQ$.
If we talk about singular (Borel Moore) homology or singular cohomology $H^*(M)$
of our stack $M$,
then we mean with respect to the analytic topology on $\underline{M}$.

If we have a singular cohomology cycle $\alpha$ on $\underline{M}$ with respect to the analytic topology above
and a class $C \in CH_*(M)$,
then we can pair them, giving a rational number.
This pairing is done first by pushing $C$ forward to $\underline{M}$
and then pairing with $\alpha$.
We will call this rational number the \emph{integral of $\alpha$ over $C$}
and write it as $\int_C \alpha$.
More generally, if we are given a power series expression with Chern class parameters $c$ then we define $\int_C \alpha \cup c = \int_{C \cap c} \alpha$.

If $V \to X$ is a vector bundle then we define $\bP(V)$ to be equal to $\Proj(\Sym(V^*))$ (I.e. the Grassman bundle of $1$-planes in $V$ rather than dimension $1$ quotients).

We write $\Coh(M)$ to be the category of coherent sheaves on $M$.
For any sheaf or complex of sheaves $\scrF$ over a stack $Z$, we let $\scrF^\vee = Hom(\scrF,\calO_Z)$
be its dual.



\section{Systems of Classes} \label{subsectionaxiomsforclasses}

In this section we define Gromov-Witten counts for systems of fundamental classes satisfying a forgetful axiom.
Let $X$ be a smooth projective variety.

\begin{defn} \label{defnalternativevfc}
	%
	A \emph{system of classes} for $X$ is a collection
	\begin{equation}
		[]' := ([X_{m,d}]')_{m \in \bN, \ d \in H_2^+(X)},
	\end{equation}
	\begin{equation}
		[X_{m,d}]'  \in CH_*(X_{m,d}), \ \ m \in \bN, \ d \in H_2^+(X),
	\end{equation}
	of classes
	so that the following property holds:
	\begin{enumerate}
		\item[(Forget.)] \label{ittemsese} For each $m,j \in \bN$, $d \in H_2^+(X)$ with $1 \leq j \leq m$,
		      we have $\pi_j^*([X_{m-1,d}]') = [X_{m,d}]'$.
	\end{enumerate}
\end{defn}

\begin{Example} \label{examplevirtual}
	$[]^{\virt} = ([X_{m,d}]^{\virt})_{m \in \bN, \ d \in H_2^+(X)}$
	is a system of fundamental classes for $X$.
	The (Forget.) axiom is proven in \cite[Axiom IV]{BehrendGW}.
\end{Example}

\begin{defn} \label{defnsefsef}
	Define
	\begin{equation} \label{eqn40primecor0}
		\langle a_1, \cdots, a_m\rangle^{[]'}_{m,d} = \int_{[X_{m,d}]'} \prod_{k=1}^m ev_k^*a_k
	\end{equation}
	for each $m \in \bN$, $a_1,\cdots,a_m \in H^*(X)$ and each $d \in H_2^+(X)$.
	We also define
	\begin{equation} \label{eqn40primecor}
		\langle a_1, \cdots, a_m \rangle^X_{m,d} := \langle a_1, \cdots, a_m \rangle^{[]^{\virt}}_{m,d}
	\end{equation}
	for each such $m,a_1,\cdots,a_m, d$.
\end{defn}


\section{Systems of Cones} \label{sectionssystemsofcones}

\subsection{Definition of Systems of Cones} \label{subsectiondefinitionofsystemsofcones}
Later on in Section \ref{sectioncomprisinvfc}, we describe the virtual fundamental class in terms of Segre and Chern classes of various cones. As a result it is natural to define systems of cones.
Let us fix a smooth projective variety $X$ and a vector bundle $\pi : V \to X$.

\begin{defn} \label{defntransfverse}
	Let $D \subset X$ be a divisor, let $\supp(D)$ be its support and let $\supp(D)^{\sing}$
	be the singular locus of its support.
	Suppose $u : \Sigma \to X$ is a morphism from a nodal curve $\Sigma$.
	Define $\Sigma^{\sing} \subset \Sigma$ to be the set of nodal points.
	We say that $u$ is \emph{transverse} to $D$
	if
	\begin{enumerate}
		\item the map $u|_{\Sigma - \Sigma^{\sing}}$ is transverse to $\supp(D) - \supp(D)^{\sing}$,
		\item $\supp(D)^{\sing} \cap u(\Sigma) = \emptyset$ and
		\item $u(\Sigma^{\sing}) \cap D = \emptyset$.
	\end{enumerate}
\end{defn}

\begin{defn} \label{defndfactorialmodulispaces}
	Let $m \in \bN$ and $d \in H_2^+(X)$.
	Let $D \subset X$ be a reduced divisor.
	Define
	\begin{equation}
		\mathring{X}_{m+D!,d} \subset X_{m+D\cdot d,d}
	\end{equation}
	to be the substack parameterizing flat families of curves where the image in $X_{0,d}$ of each curve $u$ corresponding to a $\bC$-point in each such family is transverse to $D$ and where $u^{-1}(D)$ is the union of the last $D \cdot d$ marked points.

	We define
	\begin{equation}
		X_{m+D!,d} \subset X_{m+D\cdot d,d}
	\end{equation}
	to be the closure of $\mathring{X}_{m+D!,d}$.
	We have the following evaluation map
	\begin{equation}
		\ev : X_{m+D!,d} \to X^{m+D \cdot d}, \quad \ev := \prod_{j=1}^{m+D \cdot d} \ev_j.
	\end{equation}
\end{defn}

Let $\pi : V \to X$ be a vector bundle over $X$.



\begin{defn} \label{evaluationmap}
	Let $m \in \bN$ and $d \in H_2^+(X)$.
	We define
	\begin{equation} \label{eqnidentificationofspecsymwithmodulispace}
		V_{m,d} = V_{m,\iota_*d}
	\end{equation}
	where $\iota : X \hookrightarrow V$ is zero section map.
	This has a natural projection map:
	\begin{equation}
		\pi_{V_{m,d}} : V_{m,d} \to X_{m,d}
	\end{equation}
	sending $u$ to $\pi \circ u$.
	Define $V^{\boxtimes m} := \bigoplus_{i=1}^m p_i^*V$
	where $p_i : X^m \to X$ is the $i$th projection map for each $i = 1,\cdots,m$.
	We define
	\begin{equation} \label{eqnevaluationV}
		\ev : V_{m,d} \to V^{\boxtimes m}
	\end{equation}
	to be the evaluation map given by $\ev_1 \times \cdots \times \ev_m$.
	We let
	\begin{equation}
		\ev^* V^{\boxtimes m} \to X_{m,d}
	\end{equation}
	be the pullback of the vector bundle $V^{\boxtimes m} \to X^m$
	via the map $\ev : X_{m,d} \to X^m$.
	Define
	\begin{equation}
		\widetilde{\ev} : V_{m,d} \to \ev^*V^{\boxtimes m}
	\end{equation}
	to be the pullback of the map \eqref{eqnevaluationV}.
	Now, let $D$ be a reduced divisor in $X$.
	Define
	\begin{equation}
		V_{m+D!,d} := V_{m+\pi^{-1}(D)!,\iota_*d}.
	\end{equation}
	%
	Define
	\begin{eqnarray} \label{evvdevinition}
		\widetilde{\ev} : V_{m+D!,d} \to \ev^* V^{\boxtimes m + D \cdot d}
	\end{eqnarray}
	to be the restriction of
	$\widetilde{\ev} : V_{m+D \cdot d, d} \to \ev^*V^{\boxtimes m + D \cdot d}$ to $V_{m+D!,d}$.
\end{defn}

Since $V$ is a vector bundle, we get that $V_{m,d}$ is a cone over $X_{m,d}$ where the $\bC^*$-action is given by post composing each curve by the scaling action on $V$.

\begin{defn} \label{defnsystemofcones}
	A \emph{system of cones for $V$} is a collection
	\begin{equation}
		C_\bullet = (C_{m,d})_{m \in \bN, \ d \in H_2^+(X)}
	\end{equation}
	where $C_{m,d}$ is a (not necessarily closed) subcone of $V_{m,d}$ so that the pullback of $C_{m,d}$ to $X_{m+1,d}$ via the forgetful map $\pi_j$ is $C_{m+1,d}$
	under the identification $V_{m+1,d} \cong \pi_j^*V_{m,d}$ for each $m,j \in \bN$ with $1 \leq j \leq m$ and $d \in H_2^+(X)$.


	We define $V_\bullet$ to be the following system of cones for $V$:
	\begin{equation}
		V_\bullet = (V_{m,d})_{m \in \bN, \ d \in H_2^+(X)}.
	\end{equation}
	Now let $D$ be a smooth hypersurface.
	For each $m \in \bN$ and $d \in H_2^+(X)$,
	we define $C_{m+D!,d}$ to be the pullback of the cone $C_{m,d}$ to $V_{m+D!,d}$
	under the map $V_{m+D!,d} \to V_{m,d}$ forgetting the last $D \cdot d$ marked points.
\end{defn}

\begin{defn} \label{defngenericallytransverse}
	Let $C_\bullet$ be a system of cones for $V$ as in Definition \ref{defnsystemofcones}.
	We say that a divisor $D$ is \emph{generically transverse to $C_\bullet$} if for each $m \in \bN$ and each $d \in H_2^+(X)$ there exists a dense open subspace $U \subset C_{m,d}$ so that $D$ is transverse to $\pi \circ u$ for each curve $u : \Sigma \to V$ corresponding to a $\bC$-point in $U$.

\end{defn}

\subsection{\texorpdfstring{$D$}--Volumes of Systems of Cones} \label{subsectionvolumesofsystemsofcones}

In order to estimate Gromov-Witten counts, we need to define the ``volume'' of a collection of cones with respect to an ample divisor. Such estimates will be stated and proven in Section \ref{sectionsomeboundsfordvolumes} later on.
Let $X$ be a smooth projective variety and let $\pi : V \to X$ be a vector bundle.

\begin{lemma} \label{lemagloballygenerated}
	The dual of $V$ is globally generated if and only if
	$L := \calO_{\bP(V)}(1)$ is globally generated.
\end{lemma}
\begin{proof}
	Let $\pi_\bP : \bP(V) \to X$ be the natural projection map. Then $(\pi_\bP)_*L = V^*$.
	Hence, $H^0(V^*) \to H^0(V^*|_x)$ is surjective if and only if  $H^0(L) \to H^0(L|_{\bP(V)|_x})$ is surjective for each $x \in X$.
	Since $H^0(L|_{\bP(V)|_x})$ is globally generated and any proper linear subsystem of $H^0(L|_{\bP(V)|_x})$ is not globally generated, we then get that $V^*$ is globally generated if and only if $L$ is globally generated.
\end{proof}

\begin{defn} \label{defnprojectivecompactification}
	Define  $P_V := \bP(V \oplus \bC)$, $L_V := \calO_{P_V}(1)$ and $\infty_V := \bP(V) \subset P_V$.
	We define $\pi_{P_V} : P_V \to X$ to be the natural projection map.

	Let $(V_i)_{i \in I}$ be a finite collection of vector bundles over $X$.
	Define the fiber product
	\begin{equation}
		P_{(V_i)_{i \in I}} := \fprod{X}_{i \in I} P_{V_i}.
	\end{equation}
	Define
	\begin{equation}
		\pi_{P_{(V_i)_{i \in I}}} : P_{(V_i)_{i \in I}} \to X
	\end{equation}
	to be the corresponding projection map.
	We define
	\begin{equation}
		\infty_{(V_i)_{i \in I}} := \sum_{i \in I} P_i^* \infty_{V_i}
	\end{equation}
	where $P_i : P_{(V_i)_{i \in I}} \to P_{V_i}$ is the natural projection map for each $i \in I$.

	If $I = \sqcup_{j = 0}^k I_j$ then we sometimes write:
	\begin{equation}
		P_{(V_j)_{j \in I_0},\cdots,(V_j)_{j \in I_k}} = P_{(V_i)_{i \in I}},
	\end{equation}
	\begin{equation}
		\pi_{P_{(V_j)_{j \in I_0},\cdots,(V_j)_{j \in I_k}}} = \pi_{P_{(V_i)_{i \in I}}},
	\end{equation}
	and
	\begin{equation}
		\infty_{(V_j)_{j \in I_0},\cdots,(V_j)_{j \in I_k}} = \infty_{(V_i)_{i \in I}}.
	\end{equation}
\end{defn}

\begin{corollary} \label{corollarytwistedgloballygenerated}
	Suppose $L_0$ is a globally generated line bundle over $X$.
	Then $V^* \otimes L_0$ is globally generated if and only if $L_V \otimes \pi_{P_V}^*L_0$ is globally generated.
\end{corollary}
\begin{proof}
	Let $W := V \otimes L_0^* \oplus L_0^*$.
	Then $\bP(W) = P_V$
	and $\calO_{\bP(W)}(1) = L_V \otimes \pi_{P_V}^*L_0$.
	Our result follows from Lemma \ref{lemagloballygenerated}.
\end{proof}

\begin{defn} \label{defnnefproduct}
	Let $D$ be a divisor on $X$.
	Let $p_i : X^m \to X$ be the $i$th projection map for each $i=1,\cdots,m$.
	We define
	\begin{equation}
		D_{\sum m} := \sum_{i=1}^m p_i^*D.
	\end{equation}
	For each finite collection $(V_i)_{i \in I}$ of vector
	bundles over $X$, we define
	\begin{equation}
		D_{\sum m, (V_i)_{i \in I}} := \pi_{P_{(V_i^{\boxtimes m})_{i \in I}}}^*D_{\sum m} + \infty_{(V_i^{\boxtimes m})_{i \in I}}.
	\end{equation}
	If $I = \sqcup_{j = 0}^k I_j$ then we sometimes write
	\begin{equation}
		D_{\sum m,(V_j)_{j \in I_0},\cdots,(V_j)_{j \in I_k}} := D_{\sum m, (V_i)_{i \in I}}.
	\end{equation}
\end{defn}

\begin{corollary} \label{corollarynefsum}
	Let $D$ be a divisor in $X$ so that $\calO(D)$ and $V^* \otimes \calO(D)$ are globally generated.
	Then for each $m \in \bN$ we have that $D_{\sum m, V}$
	is globally generated and hence nef.
\end{corollary}
\begin{proof}
	The line bundle on $P_{V^{\boxtimes m}}$ corresponding to $D_{\sum m, V}$ is equal to $L_{V^{\boxtimes m}} \otimes \pi_{P_{V^{\boxtimes m}}}^* \calO(D_{\sum m})$.
	This is globally generated by Corollary \ref{corollarytwistedgloballygenerated} since \begin{equation}
		\calO(D_{\sum m})=\sum_{i=1}^m p_i^*\calO(D), \quad (V^{\boxtimes m})^* \otimes \calO(D_{\sum m}) = \sum_{i=1}^m p_i^*(V^* \otimes \calO(D))
	\end{equation}
	are globally generated where  $p_i : X^m \to X$ is the $i$th projection map for each $i=1,\cdots,m$.
\end{proof}

\begin{corollary} \label{corollarynefsum2}
	Now suppose that $(V_i)_{i \in I}$ is a finite collection of vector bundles over $X$ and suppose that $D_i$ is a divisor in $X$ so that $\calO(D_i)$ and $V_i^* \otimes \calO(D_i)$ are globally generated for each $i \in I$. Let $D$ be a divisor in $X$ so that $D - \sum_{i \in I} D_i$ is nef. Then $D_{\sum m, (V_i)_{i \in I}}$ is nef.
\end{corollary}
\begin{proof}
	Let $P_i : P_{(V_i)_{i \in I}} \to P_{V_i}$ be the natural projection map for each $i \in I$.
	Then
	\begin{equation}
		D_{\sum m, (V_i)_{i \in I}} = \sum_{i \in I} P_i^* ((D_i)_{\sum m, V_i}) + \pi_{P_{(V_i)_{i \in I}}}^*(D-\sum_{i \in I} D_i)_{\sum m}
	\end{equation}
	is a sum of nef divisors by Corollary \ref{corollarynefsum}.
\end{proof}

\begin{lemma} \label{lemmmavverylarge}
	Let $D$, $D'$ be divisors in $X$ so that $V^* \otimes \calO(D')$ and $\calO(D')$ are globally generated and $D - D'$ is ample.
	Then for all $m \in \bN$ we have $D_{\sum m,V}$
	is ample.
\end{lemma}
\begin{proof}
	We will use Kleiman's criterion to prove this.
	We have $D'_{\sum m,V}$
	is nef by Corollary \ref{corollarynefsum}.
	Also,
	\begin{equation} \label{eqnsplitdivisor}
		D_{\sum m, V} = D'_{\sum m, V} + \pi_{P_{V^{\boxtimes m}}}^*((D-D')_{\sum m}).
	\end{equation}
	Now let $C$ be in the closure of the cone of curves in $N_1(P_{V^{\boxtimes m}})$.
	Then if $C$ is contained in a fiber of $\pi_{P_{V^{\boxtimes m}}}$ then it pairs positively with $D_{\sum m, V}$ since $D_{\sum m, V}$ restricted to this fiber  is a linear hypersurface which is ample.
	If $C$ is not contained in a fiber then
	\begin{equation} \label{eqndivisordifferencepositive}
		C \cdot (\pi_{P_{V^{\boxtimes m}}}^*(D-D')_{\sum m})
		= \pi_{P_{V^{\boxtimes m}}}(C) \cdot (D-D')_{\sum m} > 0
	\end{equation}
	since $(D-D')_{\sum m}$ is ample.
	So
	\begin{equation}
		C \cdot D_{\sum m, V}
		\overset{\textnormal{\eqref{eqnsplitdivisor}}}{=}
		C \cdot D'_{\sum m, V} + C \cdot \pi_{P_{V^{\boxtimes m}}}^*((D-D')_{\sum m})
		\geq
	\end{equation}
	\begin{equation}
		C \cdot \pi_{P_{V^{\boxtimes m}}}^*((D-D')_{\sum m}) \overset{\textnormal{\eqref{eqndivisordifferencepositive}}}{>}0.
	\end{equation}
\end{proof}

\begin{defn} \label{defnvolumeformultiplecones}
	Let $(V_i)_{i \in I}$ be a finite collection of vector bundles over $X$.
	Let
	\begin{equation}
		C^i_\bullet = (C^i_{m,d})_{m \in \bN, \ d \in H_2^+(X)}
	\end{equation}
	be a system of cones for $V_i$ as in Definition \ref{defnsystemofcones} for each $i \in I$.
	Suppose $D$ is a reduced divisor.
	Let $m \in \bN$ and $d \in H_2^+(X)$ and let $m' := m + D \cdot d$.
	Define
	\begin{equation} \label{eqnfiberproductofcones}
		\widetilde{\ev}(C^i_{m+D!,d})_{i \in I} := \widetilde{\ev}\left(\fprod{X_{m+D!,d}}_{i \in I} C^i_{m+D!,d}\right) \subset \ev^*\oplus_{i \in I} V_i^{\boxtimes m'}.
	\end{equation}
	We define
	\begin{equation}
		\ev^* P_{(V_i^{\boxtimes m'})_{i \in I}}
	\end{equation}
	to be the pullback of $P_{(V_i^{\boxtimes m'})_{i \in I}} \to X^{m'}$
	via the map $\ev : X_{m',d} \to X^{m'}$.
	Define
	\begin{equation} \label{eqnPev}
		P\ev : \ev^*P_{(V_i)_{i \in I}} \to P_{(V_i)_{i \in I}}
	\end{equation}
	to be the natural projection map.
	We define
	\begin{equation}
		\overline{\widetilde{\ev}((C^i_{m+D!,d})_{i \in I})}
	\end{equation}
	to be the closure of $\widetilde{\ev}(C^i_{m+D!,d})_{i \in I}$ inside $\ev^*P_{(V_i^{\boxtimes m'})_{i \in I}}$
	(Definition \ref{evaluationmap}).
	We let
	\begin{equation}
		P\ev_*\left[\overline{\widetilde{\ev}((C^i_{m+D!,d})_{i \in I})}\right]
	\end{equation}
	be the image of its Chow class under the map \eqref{eqnPev}.


	Define the \emph{$D$-volume} of $(C^i_{m,d})_{i \in I}$ to be
	\begin{equation}
		D\Vol(C^i_{m,d})_{i \in I} :=
		\frac{1}{m!(D \cdot d)!} \int_{P\ev_*\left[\overline{\widetilde{\ev}((C^i_{m+D!,d})_{i \in I})}\right]} \exp(D_{\sum m',(V_i)_{i \in I}}).
	\end{equation}
	If $I = \sqcup_{j = 0}^k I_j$ then we sometimes write
	\begin{equation}
		\ev^*P_{(V_j^{\boxtimes m'})_{j \in I_0},\cdots,(V_j^{\boxtimes m'})_{j \in I_k}} = \ev^*P_{(V_i^{\boxtimes m'})_{i \in I}},
	\end{equation}
	\begin{equation}
		\widetilde{\ev}((C^i_{m+D!,d})_{i \in I_0},\cdots,(C^i_{m+D!,d})_{i \in I_k}) =  \widetilde{\ev}((C^i_{m+D!,d})_{i \in I}),
	\end{equation}
	\begin{equation}
		\overline{\widetilde{\ev}((C^i_{m+D!,d})_{i \in I_0},\cdots,(C^i_{m+D!,d})_{i \in I_k})} =  \overline{\widetilde{\ev}((C^i_{m+D!,d})_{i \in I})}
	\end{equation}
	and
	\begin{equation}
		D\Vol((C^i_{m+D!,d})_{i \in I_0},\cdots,(C^i_{m+D!,d})_{i \in I_k})
		= D\Vol(C^i_{m,d})_{i \in I}.
	\end{equation}

\end{defn}

\subsection{\texorpdfstring{V}--Large Divisors} \label{subsectionVlargedivisors}

In order to compare $D$-volumes with Gromov-Witten counts, we need the divisor $D$ to satisfy special properties. These properties ensure that each irreducible component of each curve intersects the divisor in many points so that we can evaluate along enough of these points (Lemma \ref{lemmainjectivefromvlarge} is the crucial property that is needed here).
We fix a smooth projective variety $X$ and a vector bundle $\pi : V \to X$ over it.

\begin{defn} \label{defncompactieab}
	A divisor $D \subset X$ is \emph{simple $V$-large} if it is reduced and
	for each non-constant $u : \bP^1 \to X$,
	\begin{enumerate}
		\item the image of $u$ is not contained in $D$
		\item \label{itemprop5}
		      and the number of points in $u^{-1}(D)$ is greater than or equal to $1 + \int_{\bP^1} u^*(c_1(V))$.
	\end{enumerate}
	We say that $D \subset X$ is \emph{$V$-large} if there exists vector bundles $(V_i)_{i \in I}$ over $X$ so that
	\begin{enumerate}
		\item $D$ is simple $V_i$-large for each $i \in I$ and
		\item $V \cong \bigoplus_{i \in I} V_i$.
	\end{enumerate}
\end{defn}

\begin{lemma} \label{lemmavlargeexist}
	Let $\iota_X : X \hookrightarrow \bP^N$ be a closed immersion and let $(V_i)_{i \in I}$ be vector bundles over $X$ satisfying $V \cong \oplus_{i \in I} V_i$.
	Let $\widehat{D}$ be a very general smooth divisor in $\bP^N$
	of degree $\delta + 2N$ so that
	$(\delta-1) \deg(C) \geq c_1(V_i)(C)$ for each $i \in I$ and each rational curve $C$ in $X$.
	Then $D = \widehat{D} \cap X$ is $V$-large.
\end{lemma}
\begin{proof}
	Let $u : \bP^1 \to X$ be a non-constant morphism.
	By
	\cite{voisin1996conjecture}, the image of $u$ is not contained in $D$.
	Now write $u$ as a composition:
	\begin{equation}
		\bP^1 \lra{e} \bP^1 \lra{\phi} X
	\end{equation}
	where $\phi$ is a somewhere injective morphism.
	By
	\cite[Theorem 1.7]{chen2004algebraic} or \cite[Corollary 4]{PacienzaRousseau},
	we get that the number of points in $\phi^{-1}(D)$ is greater than $1+\delta \deg(\phi)$, which in turn is $\geq 1 + \int_{\bP^1} \phi^*(c_1(V_i))+\deg(\phi)$ for each $i \in I$.
	Since the sum of the multiplicities of the branch points of $e$ is at most $\deg(u)$, our result follows.
\end{proof}

\begin{defn} \label{defnconvex}
	A vector bundle $W \to X$ is \emph{convex} if for each $m \in \bN$, $d \in H_2^+(X)$ and each curve $u$ corresponding to a point $\Spec(\bC) \to X_{m,d}$ we have $H^1(u^*W) = 0$.
	We say that $W$ is \emph{strongly convex} if it is convex and if $u^*W$ is globally generated for each such $u$.
\end{defn}

If $W$ is convex then $W_{m,d}$ is a vector bundle over $X_{m,d}$ by Grothendieck's cohomology and base change theorem.
%
The following lemma will be useful later on in Section \ref{sectionmainargument}.
\begin{lemma} \label{lemmaconvexquotient}
	Let $r : V \twoheadrightarrow W$ be a surjection of vector bundles.
	If $V$ is (strongly) convex then $W$ is (strongly) convex.
\end{lemma}
\begin{proof}
	Let $u$ be a curve corresponding to a point $\Spec(\bC) \to X_{m,d}$.
	Then we have an exact sequence:
	\begin{equation}
		H^1(u^*V) \to H^1(u^*W) \to 0.
	\end{equation}
	Since $H^1(u^*V) = 0$, we get
	$H^1(u^*W) = 0$ which implies that $W$ is convex.
	Also, $u^*W$ is globally generated if $u^*V$ is.
\end{proof}

\begin{corollary} \label{corolarysumconvex}
	Let $(V_i)_{i \in I}$ be a finite collection of vector bundles over $X$.
	If $\bigoplus_{i \in I} V_i$ is strongly convex, then $V_i$ is strongly convex for each $i \in I$.
\end{corollary}

\begin{lemma} \label{lemmainjectivefromvlarge}
	Let $D$ be a $V$-large divisor.
	Suppose also that $V|_U$ is strongly convex for some open subset $U \subset X$.
	For each $m \in \bN$, $d \in H_2^+(X)-0$ and for each curve $u$ corresponding to a point $\Spec(\bC) \to X_{m+D!,d}$ with image in $U$, the map
	\begin{equation} \label{evaevaeev}
		\widetilde{\ev}|_u : V_{m+D!,d}|_u \to \ev^*(V^{\boxplus m+D\cdot d})|_u
	\end{equation}
	from \eqref{evvdevinition} is injective.
\end{lemma}
\begin{proof}
	By Corollary \ref{corolarysumconvex}, it is sufficient for us to prove this when $D$ is simple $V$-large. Therefore, let us assume that $D$ is simple $V$-large.
	Let $u : \Sigma \to U$ be a curve in $X_{m+D!,d}$ as above and let $v \in V_{m+D!,d}|_u = H^0(u^* V)$ be an element of the corresponding fiber.
	Suppose that the image $\widetilde{\ev}(v)$ of $v$ in $\ev^*(V^{\boxtimes m + D\cdot d})|_u$ is zero.
	Let  $\Sigma' \subset \Sigma$ be an irreducible component so that $u|_{\Sigma'}$ is non-constant.
	Let $P \subset \Sigma'$ be the subset of points $p$ satisfying $u(p) \in D$.
	For each $p \in P$, let $\Sigma_p \subset \Sigma$ be the connected component of $u^{-1}(u(p))$ containing $p$.
	Then $\Sigma_p$ contains a marked point for each $p \in P$ since $u$ doesn't have image in $D$ and $u$ is a limit of curves $v$ transverse to $D$ with $v^{-1}(D)$ a union of marked points.
	Since $u|_{\Sigma_p}$ is constant with image in $D$, we get that $v(p) = 0$.
	Hence, $v|_{\Sigma'}$ vanishes on each point $p \in \Sigma'$.

	Also, since $D$ is simple $V$-large, we have
	\begin{equation} \label{eqnfe94g}
		|P| \geq 1+\int_{\Sigma'} u^*(c_1(V)).
	\end{equation}
	Since $V$ is strongly convex, we get that $u^*V$ is a sum of line bundles of non-negative degree and hence
	\eqref{eqnfe94g} implies $v|_{\Sigma'} = 0$ for degree reasons.
	Hence, $v$ vanishes on each non-constant irreducible component of $\Sigma$.
	Since $\Sigma$ is connected and has at least one non-constant irreducible component due to the fact that $d \neq 0$, we get $v = 0$.
\end{proof}

\begin{defn} \label{defnrestrictionofcones}
	Let $U \subset X$ be subvariety, let $m \in \bN$ and $d \in H_2^+(X)$.
	We define  $X_{m,d}|_U \subset X_{m,d}$ to be the subspace parameterizing families of curves with image in $U$.
	Now let $V \to X$ be a vector bundle.
	If $C \subset V_{m,d}$ is a subcone, then we define $C|_U \subset V_{m,d}$ to be the subspace parameterizing families of curves mapping to $V$ that project to $X_{m,d}|_U$.
	We think of this as a cone over $X_{m,d}$ (rather than $X_{m,d}|_U$).
	If $C_\bullet = (C_{m,d})_{m \in \bN, \ d \in H_2^+(X)}$ is a system of cones for $V$ (Definition \ref{defnsystemofcones}) then we define the following system of cones for $V$:
	\begin{equation}
		C_\bullet|_U := (C_{m,d}|_U)_{m \in \bN, \ d \in H_2^+(X)}.
	\end{equation}

	Now let $D \subset X$ be a reduced divisor. We define $X_{m+D!,d}|_U := X_{m+D!,d} \cap X_{m+D \cdot d,d}|_U$.
	If $C' \subset V_{m+D!,d}$ is a subcone then we define $C'|_U := C' \cap (V_{m+D \cdot d,d}|_U)$, thought of as a cone over $X_{m+D!,d}$.
\end{defn}

\begin{defn} \label{defnprojectiveevaluatoin}
	Let $D$ be a $V$-large divisor and suppose $V$ is strongly convex.
	We define
	\begin{equation} \label{eqnaefaefPev}
		P\widetilde{\ev} : P_{V_{m+D!,d}} \to \ev^*P_{(V)^{\boxtimes m + D \cdot d}}
	\end{equation}
	to be the extension
	of the map $\widetilde{\ev}$ to the projective compactification of these vector bundles (which exists by Lemma \ref{lemmainjectivefromvlarge}).
\end{defn}


\section{Some Bounds Involving \texorpdfstring{$D$}--Volumes} \label{sectionsomeboundsfordvolumes}

\subsection{Segre Classes Coming from Systems of Cones} \label{subsectionsegreclassesfromsystemsofcones}

In Section \ref{vfcconnection} we compute the virtual fundamental class of the moduli space of stable maps from Segre classes of cones. Therefore, we need to bound Segre class counts in terms of $D$-volumes both from above and below. We will do this in this section.
We let $X$ be a smooth projective variety and $\pi : V \to X$ a vector bundle.

\begin{defn} \label{defnsegrefromsystemsofcones}
	Let $C_\bullet = (C_{m,d})_{m \in \bN, \ d \in H_2^+(X)}$
	be a system of cones for $V$ as in Definition \ref{defnsystemofcones}.
	For each $m \in \bN$ and $d \in H_2^+(X)$, we define $s(C_{m,d})$ to be the Segre class of the closure of $C_{m,d}$ inside $V_{m,d}$.
	We define the system of classes
	(Definition \ref{defnalternativevfc})
	\begin{equation}
		s(C_\bullet) := (s(C_{m,d}))_{m \in \bN, \ d \in H_2^+(X)}.
	\end{equation}
	We define $\overline{C_{m,d}}$ to be the closure of $C_{m,d}$ inside $P_{V_{m,d}}$
	and $\overline{C_{m+D!,d}}$ the closure of $C_{m+D!,d}$ inside $P_{V_{m+D!,d}}$
	for each $m \in \bN$, $d \in H_2^+(X)$ and any smooth divisor $D$ in $X$.
\end{defn}

Now, we wish to bound Gromov-Witten counts coming from Segre classes as in Definition \ref{defnsefsef}
in terms of $D$-volume under certain conditions.

\begin{lemma} \label{lemmadimensionofX}
	Suppose that $X$ admits a closed immersion into $\bP^N$ for some $N \in \bN$.
	Then
	\begin{equation} \label{eqndimensio}
		\dim(X_{m,d}) \leq (N+1)\deg(d) + N-3+m.
	\end{equation}
\end{lemma}
\begin{proof}
	Let $\iota_X : X \hookrightarrow \bP^N$ be our closed immersion.
	Then $X_{m,d} \subset \bP^N_{m,(\iota_X)_*d}$ and the dimension of $\bP^N_{m,(\iota_X)_*d}$ is the right-hand side of \eqref{eqndimensio} since $T\bP^N$ is convex and $c_1(T\bP^N|_X)(d) = (N+1)\deg(d)$.
	Hence, \eqref{eqndimensio} is our upper bound.
\end{proof}

To compare $D$-volumes with Gromov-Witten counts, we need the following computation of the pushforward of the Segre class of a cone. Here the $V$-large property is essential.

\begin{lemma} \label{lemmasegre}
	Suppose that $V$ is strongly convex and that $X$ admits a closed immersion into $\bP^N$ for some $N \in \bN$.
	Let $D$ be a $V$-large divisor which is generically transverse to some system of cones $C_\bullet = (C_{m,d})_{m \in \bN, \ d \in H_2^+(X)}$.
	Then
	\begin{equation} \label{eqnformulaforsegre}
		\ev_*(s(C_{m,d})) = \frac{1}{(D \cdot d)!}\sum_{r = 0}^{f_d} F_*(\pi_{P_{V^{\boxtimes m + D \cdot d}}})_*\left(
		P\ev_*[\overline{\widetilde{\ev}(C_{m+D!,d})}] \cap (\infty_{V^{\boxtimes m + D \cdot d}})^r
		\right)
	\end{equation}
	(Definition \ref{defnvolumeformultiplecones})
	where
	\begin{equation} \label{eqnforF}
		F : X^{m+D \cdot d} \to X^m
	\end{equation}
	is the projection map to the first $m$ factors and $f_d \in \bN$ satisfies
	\begin{equation} \label{eqnfhd}
		f_d \geq  (N+1)\deg(d) + N-3 + c_1(V)(d) + \dim(V)
	\end{equation}
	for each $m \in \bN$ and $d \in H_2^+(X)-0$.
\end{lemma}
\begin{proof}
	Let $\underline{\rho} : X_{m+D!,d} \to X_{0,d}$ be the map forgetting all the marked points.
	Since $V_{m+D!,d} \cong \underline{\rho}^*V_{0,d}$, we get an induced map
	$\rho : P_{V_{m+D!,d}} \to P_{V_{0,d}}$.
	We have that
	\begin{equation}
		\infty_{V_{m+D!,d}}^r = \rho^*(\infty_{V_{0,d}})^r = 0
	\end{equation}
	if $r > f_d$ since $\dim(P_{V_{0,d}}) \leq f_d$ by Lemma \ref{lemmadimensionofX} and the fact that $V$ is convex.
	Similarly, $\infty_{V_{m,d}}^r = 0$ for $r > f_d$.
	Hence,
	\begin{equation} \label{eqnsegreformula}
		\begin{aligned}
			s(C_{m,d})    & = \sum_{r = 0}^{f_d} (\pi_{P_{V_{m,d}}})_*([\overline{C_{m,d}} \cap \infty_{V_{m,d}}^r]),          \\
			s(C_{m+D!,d}) & = \sum_{r = 0}^{f_d} (\pi_{P_{V_{m+D!,d}}})_*([\overline{C_{m+D!,d}} \cap \infty_{V_{m+D!,d}}^r]).
		\end{aligned}
	\end{equation}

	Let $m' := m + D \cdot d$.
	By Lemma \ref{lemmainjectivefromvlarge}, the map $\widetilde{\ev}$ is fiberwise injective.
	In particular,
	\begin{equation} \label{eqnpulbackofo1r}
		P\widetilde{\ev}^*P\ev^*\calO_{P_{V^{\boxtimes m'}}}(1) = \calO_{P_{V_{m+D!,d}}}(1)
	\end{equation}
	and
	\begin{equation} \label{eqnpushforwardclosurerelation}
		P{\widetilde{\ev}}_*[\overline{C_{m+D!,d}}] = [\overline{\widetilde{\ev}(C_{m+D!,d})}]
	\end{equation}
	(Definition \ref{defnprojectiveevaluatoin}).
	Hence,
	\begin{equation} \label{eqnevstarpistar}
		\begin{aligned}
			\ev_*(s(C_{m+D!,d}))
			 & \overset{\textnormal{\eqref{eqnsegreformula}}}{=} \sum_{r=0}^{f_d} \ev_*((\pi_{P_{V_{m+D!,d}}})_* ([\overline{C_{m+D!,d}}] \cap (\infty_{V_{m+D!,d}})^r))
			\\
			 & \overset{\textnormal{\eqref{eqnpulbackofo1r}}}{=} \sum_{r=0}^{f_d} \ev_*((\pi_{P_{V_{m+D!,d}}})_* ([\overline{C_{m+D!,d}}] \cap P\widetilde{\ev}^*P\ev^*(\infty_{V^{\boxtimes m'}}^r)))
			\\
			 & = \sum_{r=0}^{f_d} (\pi_{P_{V^{\boxtimes m'}}})_* P\ev_*P{\widetilde{\ev}}_* ([\overline{C_{m+D!,d}}] \cap P\widetilde{\ev}^*P\ev^*(\infty_{V^{\boxtimes m'}}^r))
			\\
			 & = \sum_{r=0}^{f_d} (\pi_{P_{V^{\boxtimes m'}}})_* (P\ev \circ P\widetilde{\ev})_* ([\overline{C_{m+D!,d}}]) \cap (P\ev \circ P\widetilde{\ev})^*(\infty_{V^{\boxtimes m'}}^r)
			\\
			 & = \sum_{r=0}^{f_d} (\pi_{P_{V^{\boxtimes m'}}})_* (P\ev \circ P\widetilde{\ev})_* ([\overline{C_{m+D!,d}}]) \cap \infty_{V^{\boxtimes m'}}^r
			\\
			 & = \sum_{r=0}^{f_d} (\pi_{P_{V^{\boxtimes m'}}})_* P\ev_* P\widetilde{\ev}_* ([\overline{C_{m+D!,d}}]) \cap \infty_{V^{\boxtimes m'}}^r
			\\
			 & \overset{\textnormal{\eqref{eqnpushforwardclosurerelation}}}{=} \sum_{r=0}^{f_d} (\pi_{P_{V^{\boxtimes m'}}})_*\left(
			P\ev_*[\overline{\widetilde{\ev}(C_{m+D!,d})}] \cap \infty_{V^{\boxtimes m'}}^r
			\right).
		\end{aligned}
	\end{equation}
	Let $\pi : X_{m+D!,d} \to X_{m,d}$ be the map forgetting the last $D \cdot d$ marked points.
	Let $P\pi : P_{V_{m+D!,d}} \to P_{V_{m,d}}$
	be the map induced by this forgetful map.
	This map satisfies:
	\begin{equation} \label{eqnpullbackinfinitybypi}
		\infty_{V_{m+D!,d}} = P\pi^*\infty_{V_{m,d}}.
	\end{equation}
	Since $D$ is generically transverse to $C_\bullet$, we have that
	$P\pi|_{C_{0,m+D!,d}} : C_{0,m+D!,d} \to C_{0,m,d}$ is a generically \'{e}tale morphism of order $(D \cdot d)!$.
	Hence,
	\begin{equation} \label{eqnpushfowardidentity}
		P\pi_*([\overline{C_{m+D!,d}}]) = (D \cdot d)![\overline{C_{m,d}}].
	\end{equation}
	We have
	\begin{equation}
		\begin{aligned}
			\pi_*(s(C_{m+D!,d})) & \overset{\textnormal{\eqref{eqnsegreformula}}}{=}
			\sum_{r=0}^{f_d} \pi_*((\pi_{P_{V_{m+D!,d}}})_* ([\overline{C_{m+D!,d}}] \cap (\infty_{V_{m+D!,d}})^r))
			\\
			                     & = \sum_{r=0}^{f_d} (\pi_{P_{V_{m,d}}})_* P\pi_*( ([\overline{C_{m+D!,d}}] \cap (\infty_{V_{m+D!,d}})^r))
			\\
			                     & \overset{\eqref{eqnpullbackinfinitybypi}}{=} \sum_{r=0}^{f_d} (\pi_{P_{V_{m,d}}})_* P\pi_*( ([\overline{C_{m+D!,d}}] \cap P\pi^*(\infty_{V_{m,d}})^r))
			\\
			                     & = \sum_{r=0}^{f_d} (\pi_{P_{V_{m,d}}})_* (P\pi_*([\overline{C_{m+D!,d}}]) \cap (\infty_{V_{m,d}})^r)
			\\
			                     &
			\overset{\eqref{eqnpushfowardidentity}}{=} \sum_{r=0}^{f_d} (\pi_{P_{V_{m,d}}})_* ((D \cdot d)![\overline{C_{m,d}}] \cap (\infty_{V_{m,d}})^r)
			\\
			                     & \overset{\textnormal{\eqref{eqnsegreformula}}}{=} (D \cdot d)! s(C_{m,d}).
		\end{aligned}
	\end{equation}
	Combining this with Equation \eqref{eqnevstarpistar} together with the fact that $\ev \circ \pi = F \circ \ev$ gives Equation \eqref{eqnformulaforsegre}.
\end{proof}

\begin{defn} \label{defnc0norm}
	Let $V$ be a finite dimensional complex vector space with a Hermitian metric.
	Let $\omega_V \in \wedge^2V^*$ be the corresponding imaginary part of the Hermitian form.
	Then we have an induced metric on $\wedge^k V^*$ where an orthonormal basis $\eta_1,\cdots,\eta_h$ of $V^*$
	gives an orthonormal basis $\eta_{i_1} \wedge \cdots \wedge \eta_{i_k}$, $1 \leq i_1 < \cdots < i_k \leq h$
	for $\wedge^k V^*$.
	We write $| \cdot |_{\omega_V}$ for the corresponding norm on $\wedge^k V^*$.

	Let $\omega$ be a K\"{a}hler form on a smooth variety $Y$.
	If $\alpha$ is a $k$-form on $Y$ then the \emph{$C^0$-norm of $\alpha$ with respect to $\omega$}, denoted by $|\alpha|_{C^0,\omega}$, is the supremum of the $| \cdot |_{\omega|_x}$ norm of $\alpha|_x \in \wedge^k T^*_xM$ over all $x \in M$.

\end{defn}

\begin{remark} \label{remarkcomparisonofnorms}
	If $M$ is a smooth complex manifold, $\eta \in \Omega^k(M)$ and $\omega$ is a K\"{a}hler form on $M$ then
	\begin{equation} \label{eqnintcomparisonofnorms}
		\left| \int_M \eta \right| \leq |\eta|_{C^0,\omega} \int_M \exp(\omega).
	\end{equation}
\end{remark}

\begin{lemma} \label{lemmaproductnorms}
	Let $\pi : Y \to X$ be a morphism of smooth varieties and let $\omega$ be a K\"{a}hler form on $Y$.
	Let $n = \dim(X)$.
	Also, let $\alpha$ be a $C^\infty$ differential form on $Y$ and $\beta$ a $C^\infty$ differential form on $X$.
	Then
	\begin{equation} \label{eqnwedgeproductinequality}
		|\alpha \wedge \pi^*\beta|_{C^0,\omega} \leq 2^{2n} |\alpha|_{C^0,\omega} |\pi^*\beta|_{C^0,\omega}.
	\end{equation}
	If $\beta$ is homogenous of degree $j$, then
	\begin{equation} \label{eqnwedgeproductinequalityhomogenous}
		|\alpha \wedge \pi^*\beta|_{C^0,\omega} \leq {2n \choose j} |\alpha|_{C^0,\omega} |\pi^*\beta|_{C^0,\omega}.
	\end{equation}
\end{lemma}
\begin{proof}
	Let us first assume that $\beta$ is homogenous of degree $j$.
	Suppose $y \in Y$ and let $V \subset T_y Y$ be the kernel of $d\pi|_{T_yY} : T_y Y \to T_{\pi(y)}X$.
	Let $V^\perp \subset T_yY$ be the $\omega$-orthogonal subspace of $V$
	and suppose $\Omega := \omega|_{T_yY}$.
	Choose an orthonormal basis $v_1,\cdots,v_h$ of $T_yY$ so that $v_1,\cdots,v_k$ is an orthonormal basis for $V^\perp$ and $v_{k+1},\cdots,v_n$ is an orthonormal basis for $V$. Since $\dim X = n$, we have $k \leq 2n$.
	Let $v^*_1,\cdots,v_n^*$ be the corresponding dual basis.
	For each $I \subset \{1,\cdots,h\}$, write $v^*_I := v^*_{i_1} \wedge \cdots \wedge v^*_{i_k}$ where $I = \{i_1,\cdots,i_k\}$ and $i_1 < \cdots < i_k$.
	Then $\pi^*\beta|_{T_y Y} = \sum_{\underset{I \subset \{1,\cdots,k\}}{|I|=j}} b_I v^*_I$ for some $b_I \in \bR$, $I \subset \{1,\cdots,k\}$, $|I|=j$.
	So,
	\begin{equation}
		\alpha \wedge \pi^*\beta|_{T_yY} = \sum_{\underset{I \subset \{1,\cdots,k\}}{|I|=j}} b_I \alpha|_{T_yY} \wedge v^*_I.
	\end{equation}
	Hence, by the triangle inequality, we have
	\begin{equation}
		\big|\alpha \wedge \pi^*\beta|_{T_yY}\big|_{C^0,\Omega}
		\leq \sum_{\underset{I \subset \{1,\cdots,k\}}{|I|=j}} |b_I| \big|\alpha|_{T_yY}\big|_{C^0,\Omega}
	\end{equation}
	\begin{equation}
		\leq \sum_{\underset{I \subset \{1,\cdots,k\}}{|I|=j}} |\pi^*\beta|_{T_yY}|_{C^0,\Omega} \big|\alpha|_{T_yY}\big|_{C^0,\Omega} = {2k \choose j} |\pi^*\beta|_{T_yY}|_{C^0,\Omega} \big|\alpha|_{T_yY}\big|_{C^0,\Omega}
	\end{equation}
	\begin{equation}
		\leq {2n \choose j} |\pi^*\beta|_{T_yY}|_{C^0,\Omega} \big|\alpha|_{T_yY}\big|_{C^0,\Omega}.
	\end{equation}
	Hence,
	\begin{equation}
		\big| \alpha \wedge \pi^*\beta \big|_{C^0,\omega} = \sup_{y \in Y} \big| \alpha \wedge \pi^*\beta|_{T_yY } \big|_{C^0,\omega|_{T_yY}}
	\end{equation}
	\begin{equation}
		\leq \sup_{y \in Y} {2n \choose j} |\pi^*\beta|_{C^0,\omega|_{T_yY}} \big|\alpha|_{T_yY}\big|_{C^0,\omega|_{T_yY}}
	\end{equation}
	\begin{equation}
		\leq \sup_{y \in Y} {2n \choose j} |\pi^*\beta|_{T_yY}|_{C^0,\omega|_{T_yY}} \sup_{y' \in Y} \big|\alpha|_{T_{y'}Y}\big|_{C^0,\omega|_{T_{y'}Y}}
			= {2n \choose j} |\alpha|_{C^0,\omega} |\pi^*\beta|_{C^0,\omega}.
	\end{equation}

	Now suppose that $\beta$ is not necessarily homogenous. Then we can write $\beta = \sum_{j=0}^{2n} \beta_j$ where $\beta_j$ is homogenous of degree $j$ for each $j=0,\cdots,2n$.
	Then
	\begin{equation}
		|\alpha \wedge \pi^* \beta|_{C^0,\omega} \leq \sum_{j=0}^{2n} |\alpha \wedge \pi^*\beta_j|_{C^0,\omega}
		\leq \sum_{j=0}^{2n} {2n \choose j} |\alpha|_{C^0,\omega} |\pi^*\beta_j|_{C^0,\omega}
	\end{equation}
	\begin{equation}
		\leq \sum_{j=0}^{2n} {2n \choose j} |\alpha|_{C^0,\omega} |\pi^*\beta|_{C^0,\omega}
			= 2^{2n} |\alpha|_{C^0,\omega} |\pi^*\beta|_{C^0,\omega}.
	\end{equation}
\end{proof}

\begin{lemma} \label{lemmaaboutpositivedefiniteforms}
	Let $\Omega_1, \Omega_2 \in \wedge^2 V^*$ be two imaginary parts of Hermitian forms on a finite dimensional vector space $V$ where the first one is positive definite and the second is positive semidefinite.
	Let $\Gamma \in \wedge^k V^*$.
	Then
	\begin{equation} \label{eqncomparisonofnorms}
		|\Gamma|_{\Omega_1+\Omega_2} \leq |\Gamma|_{\Omega_1},
	\end{equation}
	\begin{equation} \label{eqnomegaonenorm}
		|\wedge^j \Omega_1|_{\Omega_1+\Omega_2} \leq j! {\dim(V) \choose j},
	\end{equation}
	\begin{equation} \label{eqnomegaonenorm2}
		|\wedge^j \Omega_2|_{\Omega_1+\Omega_2} \leq j! {\dim(V) \choose j}, \quad \forall \ j = 0,\cdots,\dim(V).
	\end{equation}
\end{lemma}
\begin{proof}
	We can choose a basis so that the metrics associated to $\Omega_1$ and $\Omega_1 + \Omega_2$ are both diagonal with the first one being the identity matrix.
	The diagonal entries of the second matrix are $\geq 1$ since $\Omega_2$ has positive semidefinite associated Hermitian form. So, if we write $\Gamma$ as a linear combination of wedge products of elements of the corresponding dual basis, then its $|\cdot|_{\Omega_1}$-norm squared is the sum of the squares of the corresponding coefficients because the norms of these wedge products are $1$, and they are orthogonal to each other. However, the $|\cdot|_{\Omega_1+\Omega_2}$-norm of each such wedge product of basis elements is $\leq 1$ because the diagonal entries of the second matrix are $\geq 1$
	and so Equation \eqref{eqncomparisonofnorms} holds.

	Since $\Omega_2$ has positive semidefinite associated Hermitian form, we can approximate it by a sequence $\Omega_{2,i}$, $i \in \bN$ whose associated Hermitian forms are positive definite.
	Let $\Omega_{1,i} = \Omega_1 + \Omega_2 - \Omega_{2,i}$ for each $i \in \bN$. Then $\Omega_{1,i}$ has positive definite associated Hermitian form for each sufficiently large $i \in \bN$. Therefore, after passing to a subsequence, we can assume $\Omega_{1,i}$ has positive definite associated Hermitian form for each $i \in \bN$ too.
	Since $\Omega_{k,i}$ has a positive definite associated Hermitian form, it is equal to $\frac{i}{2}\sum_{j=1}^{\dim V} z_j^* \wedge \overline{z}_j^*$ with respect to some orthonormal basis $z_1^*,\cdots,z_{\dim(V)}^*$ of $V^*$ with respect to the Hermitian metric coming from $\Omega_{k,i}$ for each $k=1,2$, $i \in \bN$.
	Hence, by the first part of this lemma with $\Omega_1$ and $\Omega_2$ replaced with $\Omega_{1,i}$ and $\Omega_{2,i}$ respectively, each of the forms $z_j^* \wedge \overline{z}_j^*$ has $|\cdot|_{\Omega_1+\Omega_2}=|\cdot|_{\Omega_{1,i}+\Omega_{2,i}}$ norm at most $1$.
	Also, the wedge product of $z_j^* \wedge \overline{z}_j^*$ with itself is zero.
	Hence:
	\begin{equation}
		|\wedge^j \Omega_{k,i}|_{\Omega_1+\Omega_2} \leq j! {\dim(V) \choose j}, \quad \forall \ j = 0,\cdots,\dim(V), \ k=1,2
	\end{equation}
	for each $i \in \bN$.
	Taking the limit as $i$ goes to infinity gives Equations \eqref{eqnomegaonenorm} and \eqref{eqnomegaonenorm2}.
\end{proof}

\begin{lemma} \label{lemmasizeofintegral}
	Let $\pi : Y \to Z$ be a morphism between smooth projective varieties.
	Let $\omega_Y$ be a K\"{a}hler form on $Y$ and $\omega_Z$ a K\"{a}hler form on $Z$.
	Define $\omega = \omega_Y + \pi^*\omega_Z$.
	Let $T$ be a Chow cycle on $Y$ of dimension at most $h \in \bN$ and which is equal to a positive sum of closed integral subschemes and suppose the scheme theoretic image of each such subscheme in $Z$ has dimension at most $b \in \bN$.
	Let $p_i : Y \to Z_i$, $i=1,\cdots,m$ be morphisms to smooth projective varieties of dimension at most $n$.
	Also, let $\alpha_1,\cdots,\alpha_m$ be differential forms on $Z_1,\cdots,Z_m$ respectively and let $k \in \bN$.

	Then
	\begin{equation} \label{eqnchowcyclemsegs}
		\left| \int_T \left(\cup_{j=1}^m p_j^*\alpha_j\right) \cup (\omega_Y - \pi^*\omega_Z)^k \right| \leq  (4b+2 m n+4)! h^{2b}  h! \prod_{j=1}^m |p_j^*\alpha_j|_{C^0,\omega}  \int_T \exp(\omega).
	\end{equation}
\end{lemma}
\begin{proof}
	By linearity combined with the triangle inequality and the fact that $T$ is a positive sum of integral subschemes it is sufficient for us to prove this when $T \subset X$ is an integral subscheme.
	So, let us assume that this is the case.

	There exists a dense open subspace $T^0 \subset T$ with the property that both $T^0$ and $\pi(T_0)$ are smooth.
	Define $\pi_0 := \pi|_{T^0} : T^0 \to \pi(T^0)$.
	Let $\iota : T^0 \hookrightarrow Y$ and $\underline{\iota} : \pi(T^0) \to Z$ be the natural inclusion maps.

	Since $\underline{\iota}^*\omega_Z^j = 0$ for $j > b$, we have
	\begin{equation}
		\left| \int_T \left(\cup_{j=1}^m p_j^*\alpha_j\right) \cup (\omega_Y - \pi^*\omega_Z)^k \right|
		=\left| \int_{T^0} \left(\cup_{j=1}^m p_j^*\alpha_j\right) \cup (\omega_Y - \pi^*\omega_Z)^k \right|
	\end{equation}
	\begin{equation}
		\overset{\textnormal{Remark \ref{remarkcomparisonofnorms}}}{\leq} \left| \left(\cup_{j=1}^m \iota^*p_j^*\alpha_j\right) \cup (\iota^*\omega_Y - \iota^*\pi^*\omega_Z)^k \right|_{C^0,\iota^*\omega} \int_T \exp(\iota^*\omega)
	\end{equation}
	\begin{equation}
		= \left| \left(\cup_{j=1}^m \iota^*p_j^*\alpha_j\right) \cup (\iota^*\omega_Y - \pi_0^*\underline{\iota}^*\omega_Z)^k \right|_{C^0,\iota^*\omega} \int_T \exp(\iota^*\omega)
	\end{equation}
	\begin{equation}
		= \left| \sum_{j=0}^b (-1)^j {k \choose j} \left(\cup_{j=1}^m \iota^*p_j^*\alpha_j\right) \cup \iota^*\omega_Y^{k-j} \cup \pi_0^*\underline{\iota}^*\omega_Z^j \right|_{C^0,\iota^*\omega} \int_T \exp(\iota^*\omega)
	\end{equation}
	\begin{equation}
		\leq \sum_{j=0}^b {k \choose j} \left| \left(\cup_{j=1}^m \iota^*p_j^*\alpha_j\right) \cup (\iota^*\omega_Y)^{k-j} \cup (\pi_0^*\underline{\iota}^*\omega_Z)^j \right|_{C^0,\iota^*\omega} \int_T \exp(\omega)
	\end{equation}
	\begin{equation}
		\overset{\textnormal{Lemma \ref{lemmaproductnorms}}}{\leq}
		\sum_{j=0}^b {k \choose j} 2^{2 m n} {2b \choose 2}^j \left| (\iota^*\omega_Y)^{k-j} \right|_{C^0,\iota^*\omega} \left| (\pi_0^*\underline{\iota}^*\omega_Z) \right|^j_{C^0,\iota^*\omega}  \prod_{j=1}^m \left|\iota^*p_j^*\alpha_j \right|_{C^0,\iota^*\omega} \int_T \exp(\omega)
	\end{equation}
	\begin{equation}
		\overset{\textnormal{Lemma \ref{lemmaaboutpositivedefiniteforms}}}{\leq}
		\sum_{j=0}^b {k \choose j} 2^{2 m n} {2b \choose 2}^j (k-j)!{h \choose k-j} h^j  \prod_{j=1}^m \left|\iota^*p_j^*\alpha_j \right|_{C^0,\iota^*\omega} \int_T \exp(\omega)
	\end{equation}
	\begin{equation}
		\leq \sum_{j=0}^b k^j2^{2 m n} {2b \choose 2}^j h! h^b  \prod_{j=1}^m \left|\iota^*p_j^*\alpha_j \right|_{C^0,\iota^*\omega} \int_T \exp(\omega)
	\end{equation}
	\begin{equation}
		\leq (b+1) (kh)^b 2^{2 m n} ((2b)^2)^b  h!  \prod_{j=1}^m \left|\iota^*p_j^*\alpha_j \right|_{C^0,\iota^*\omega} \int_T \exp(\omega)
	\end{equation}
	\begin{equation}
		\leq (b+1) h^{2b} 2^{2 m n} (2b)^{2b} h!  \prod_{j=1}^m \left|\iota^*p_j^*\alpha_j \right|_{C^0,\iota^*\omega} \int_T \exp(\omega)
	\end{equation}
	\begin{equation} \label{eqnstuffleq}
		\leq h^{2b} 2^{2 m n} (2b+1)^{2b+1} h!  \prod_{j=1}^m \left|\iota^*p_j^*\alpha_j \right|_{C^0,\iota^*\omega} \int_T \exp(\omega)
	\end{equation}
	since $k \leq h$.
	Combining this with the following inequality:
	\begin{equation} \label{eqnreallreallyterribleinequality}
		\prod_{i=1}^{\ell} g_i! \prod_{j = 1}^{k'} e_j^{f_j} \leq (\sum_i g_i + \sum_j e_j+f_j)!, \quad \forall \ e_1,\cdots,e_{k'},f_1,\cdots,f_{k'}, g_1,\cdots,g_\ell \in \bN,
	\end{equation}
	we get \eqref{eqnstuffleq}
	\begin{equation}
		\leq (2+2 m n+2b+1+2b+1)! h^{2b} h!  \prod_{j=1}^m \left|\iota^*p_j^*\alpha_j \right|_{C^0,\iota^*\omega} \int_T \exp(\omega)
	\end{equation}
	\begin{equation}
		= (4b+2 m n+4)! h^{2b} h!  \prod_{j=1}^m \left|\iota^*p_j^*\alpha_j \right|_{C^0,\iota^*\omega} \int_T \exp(\omega).
	\end{equation}

\end{proof}

\begin{prop} \label{propboundsforsegreclass}
	Suppose $X$ is a subvariety of $\bP^N$.
	Suppose that $V$ is a strongly  convex vector bundle over $X$.
	Let $D, D'$ be divisors on $X$ so that:
	\begin{enumerate}
		\item $V^* \otimes \calO(D')$ and $\calO(D')$ are globally generated,
		\item $D$ is linearly equivalent to $4D'$,
		\item $D'$ is ample,
		\item $D$ is smooth, $V$-large and generically transverse to some system of cones \begin{equation}
			      C_\bullet = (C_{m,d})_{m \in \bN, \ d \in H_2^+(X)}
		      \end{equation} for $V$.
	\end{enumerate}
	Let $\omega$ be a K\"{a}hler form Poincar\'{e} dual to $D'$.
	Also, let $m \in \bN$, $d \in H_2^+(X)$ and let
	$\alpha_1,\cdots,\alpha_m$ be closed de Rham forms on $X$.
	Then
	\begin{equation}
		\left|
		\langle [\alpha_1],\cdots,[\alpha_m]
		\rangle_{m,d}^{s(C_\bullet)}
		\right|
		\leq (9e_d + 2 m N + 6m +9)!((\rho_d+1))^{4(e_d+m+1)}\rho_d! D\Vol(C_{m,d}) \prod_{j=1}^m |\alpha_j|_{C^0,\omega}
	\end{equation}
	where
	\begin{equation} \label{eqnedeqn}
		e_d := (N+1)\deg(d) + N,
	\end{equation}
	\begin{equation} \label{eqnedeqn2}
		\rho_d := c_1(V)(d) + \dim(V).
	\end{equation}
\end{prop}
\begin{proof}
	Let $m' = m + D \cdot d$.
	Let $p_j : X^{m'} \to X$ be the $j$th projection map for each $j=1,\cdots,m'$.
	Also, let $D'' := 2D'$.
	By Lemma \ref{lemmmavverylarge}, $D''_{\sum m', V}$ is ample and so there is a K\"{a}hler form $\widetilde{\omega}$ Poincar\'{e} dual to it.
	Let $f_d := e_d + \rho_d$.
	Define $\omega_{\sum m'} := \sum_{j=1}^{m'} p_j^* \omega$.

	%
	We have
	\begin{equation}
		\left|
		\langle [\alpha_1],\cdots,[\alpha_m]
		\rangle_{m,d}^{s(C_\bullet)}\right|
		= \left| \int_{\ev_* (s(C_{m,d}))} \prod_{i = 1}^m p_i^*\alpha_i
		\right|
	\end{equation}
	\begin{equation}
		\overset{\textnormal{Lemma \ref{lemmasegre}}}{=}
		\frac{1}{(D \cdot d)!}\left|
		\sum_{r=0}^{f_d}
		\int_{(\pi_{P_{V^{\boxtimes m'}}})_*(P\ev_*[\overline{\widetilde{\ev}(C_{m+D!,d})}] \cap \infty_{V^{\boxtimes m'}}^r)} \prod_{i = 1}^m p_i^*\alpha_i
		\right|
	\end{equation}
	\begin{equation}
		=
		\frac{1}{(D \cdot d)!}\left|
		\sum_{r=0}^{f_d}
		\int_{P\ev_*[\overline{\widetilde{\ev}(C_{m+D!,d})}] \cap \infty_{V^{\boxtimes m'}}^r} \prod_{i = 1}^m \pi_{P_{V^{\boxtimes m'}}}^*p_i^*\alpha_i
		\right|
	\end{equation}
	\begin{equation}
		= \frac{1}{(D \cdot d)!}\left| \sum_{r=0}^{f_d}
		\int_{P\ev_*[\overline{\widetilde{\ev}(C_{m+D!,d})}]} \prod_{i = 1}^m \pi_{P_{V^{\boxtimes m'}}}^*p_i^*\alpha_i \cup (D''_{\sum m',V} - \pi_{P_{V^{\boxtimes m'}}}^*D''_{\sum m'})^r
		\right|
	\end{equation}
	\begin{equation} \label{eqnlastieaeataety00}
		=
		\frac{1}{(D \cdot d)!}\left|\sum_{r=0}^{f_d}
		\int_{P\ev_*[\overline{\widetilde{\ev}(C_{m+D!,d})}]} \prod_{i = 1}^m \pi_{P_{V^{\boxtimes m'}}}^*p_i^*\alpha_i \cup (\widetilde{\omega} - 2\pi_{P_{V^{\boxtimes m'}}}^*\omega_{\sum m'}
		)^r
		\right|
	\end{equation}
	\begin{equation} \label{eqnlastieaeataety0}
		\leq
		\frac{1}{(D \cdot d)!}\sum_{r=0}^{f_d}\left|
		\int_{P\ev_*[\overline{\widetilde{\ev}(C_{m+D!,d})}]} \prod_{i = 1}^m \pi_{P_{V^{\boxtimes m'}}}^*p_i^*\alpha_i \cup (\widetilde{\omega} - 2\pi_{P_{V^{\boxtimes m'}}}^*\omega_{\sum m'}
		)^r
		\right|.
	\end{equation}
	Let $\omega' := 2\pi_{P_{V^{\boxtimes m'}}}^*\omega_{\sum m'} + \widetilde{\omega}$
	and
	\begin{equation}
		\omega^i_\epsilon = \pi_{P_{V^{\boxtimes m'}}}^* p_i^* \omega + \epsilon \omega'
	\end{equation}
	for each $i = 1,\cdots, m'$ and $\epsilon >0$.
	Then $\omega^i_\epsilon$ and $\omega' - \omega^i_\epsilon$ are both ample for $\epsilon > 0$ small.
	Hence, by Lemma \ref{lemmaaboutpositivedefiniteforms},
	\begin{equation} \label{eqninequalityforpullback}
		|\pi_{P_{V^{\boxtimes m'}}}^* p_i^*\alpha_i|_{C^0,\omega'} \leq \lim_{\epsilon \to 0} |\pi_{P_{V^{\boxtimes m'}}}^* p_i^*\alpha_i|_{C^0,\omega^i_\epsilon}
		= |\alpha_i|_{C^0,\omega}.
	\end{equation}
	We have that the dimension $X_{m+D!,d}$ is $\leq e_d + m$ by Lemma \ref{lemmadimensionofX}
	and the dimension of $\overline{\widetilde{\ev}(C_{m+D!,d})}$ is $\leq f_d+m$.
	Therefore, by Lemma \ref{lemmasizeofintegral} we have \eqref{eqnlastieaeataety0} is less than or equal to:
	\begin{equation}
		\frac{1}{(D \cdot d)!}\sum_{r=0}^{f_d} (4e_d+2 m N+4)! (f_d+m)^{2(e_d+m)} (f_d+m)!\prod_{j=1}^m |\pi_{P_{V^{\boxtimes m'}}}^* p_j^*\alpha_j|_{C^0,\omega'} \int_{P\ev_*[\overline{\widetilde{\ev}(C_{m+D!,d})}]} \exp(\omega')
	\end{equation}
	\begin{equation}
		\overset{\textnormal{\eqref{eqninequalityforpullback}}}{\leq} \frac{1}{(D \cdot d)!}\sum_{r=0}^{f_d} (4e_d + 2 m N+4)! (f_d+m)^{2(e_d+m)}(f_d+m)!\prod_{j=1}^m |\alpha_j|_{C^0,\omega} \int_{P\ev_*[\overline{\widetilde{\ev}(C_{m+D!,d})}]} \exp(\omega')
	\end{equation}
	\begin{equation}
		\overset{\textnormal{Definition \ref{defnvolumeformultiplecones}}}{=} \frac{1}{(D \cdot d)!}\sum_{r=0}^{f_d} (4e_d + 2 m N+4)! (f_d+m)^{2(e_d+m)}(f_d+m)!\prod_{j=1}^m |\alpha_j|_{C^0,\omega} m!(D\cdot d)!D\Vol(C_{m,d})
	\end{equation}
	\begin{equation}
		\leq (f_d+1) (4e_d + 2 m N+4)! (f_d+m)^{2(e_d+m)}(f_d+m)!m! D\Vol(C_{m,d}) \prod_{j=1}^m |\alpha_j|_{C^0,\omega}
	\end{equation}
	\begin{equation}
		\leq  (4e_d + 2 m N+4)! (f_d+m)^{2(e_d+m)}(e_d+\rho_d+m+1)!m! D\Vol(C_{m,d}) \prod_{j=1}^m |\alpha_j|_{C^0,\omega}
	\end{equation}
	\begin{equation}
		\leq (4e_d + 2 m N+4)!(f_d+m)^{2(e_d+m)} {f_d+m+1 \choose e_d + m + 1} (e_d+m+1)! \rho_d!m! D\Vol(C_{m,d}) \prod_{j=1}^m |\alpha_j|_{C^0,\omega}
	\end{equation}
	\begin{equation}
		\leq (4e_d + 2 m N+4)! (f_d+m)^{2(e_d+m)}(f_d+m+1)^{e_d+m+1} (e_d+m+1)! \rho_d!m! D\Vol(C_{m,d}) \prod_{j=1}^m |\alpha_j|_{C^0,\omega}
	\end{equation}
	\begin{equation}
		\leq (4e_d + 2 m N+4)!(f_d+m)^{2(e_d+m)} (f_d+m+1)^{2(e_d+m+1)} \rho_d!m! D\Vol(C_{m,d}) \prod_{j=1}^m |\alpha_j|_{C^0,\omega}
	\end{equation}
	\begin{equation}
		\leq (4e_d + 2 m N+4)!(f_d+m+1)^{4(e_d+m+1)}\rho_d!m! D\Vol(C_{m,d}) \prod_{j=1}^m |\alpha_j|_{C^0,\omega}
	\end{equation}
	\begin{equation}
		\overset{\textnormal{\eqref{eqnreallreallyterribleinequality}}}{\leq} (4e_d + 2 m N+ m +4)!(f_d+m+1)^{4(e_d+m+1)}\rho_d! D\Vol(C_{m,d}) \prod_{j=1}^m |\alpha_j|_{C^0,\omega}
	\end{equation}
	\begin{equation}
		= (4e_d + 2 m N+ m +4)!(\rho_d+e_d+m+1)^{4(e_d+m+1)}\rho_d! D\Vol(C_{m,d}) \prod_{j=1}^m |\alpha_j|_{C^0,\omega}
	\end{equation}
	\begin{equation}
		\leq (4e_d + 2 m N+ m +4)!((\rho_d+1)(e_d+m+1))^{4(e_d+m+1)}\rho_d! D\Vol(C_{m,d}) \prod_{j=1}^m |\alpha_j|_{C^0,\omega}
	\end{equation}
	\begin{equation}
		\overset{\textnormal{\eqref{eqnreallreallyterribleinequality}}}{\leq}
		(4e_d + 2 m N+ m +4+e_d+m+1+4(e_d+m+1))!(\rho_d+1)^{4(e_d+m+1)}\rho_d! D\Vol(C_{m,d}) \prod_{j=1}^m |\alpha_j|_{C^0,\omega}
	\end{equation}
	\begin{equation}
		=
		(9e_d + 2 m N + 6m +9)!((\rho_d+1))^{4(e_d+m+1)}\rho_d! D\Vol(C_{m,d}) \prod_{j=1}^m |\alpha_j|_{C^0,\omega}.
	\end{equation}

\end{proof}

We now wish to bound $D$-volume in terms of Segre counts.
Before we do this, we need a preliminary definition and some lemmas.

\begin{defn} \label{defndfactorialefactorial}
	Let $m \in \bN$ and $d \in H_2^+(X)$.
	Suppose $D, E$ are reduced divisors in $X$. Let $m' = m + D \cdot d$ and $m'' := m' + E \cdot d$.
	We define $X_{m+D!+E!,d}^\circ \subset X_{m'',d}$ be the subspace parameterizing families of curves $u$ where
	\begin{enumerate}
		\item the image of $u$ under the forgetful map forgetting the last $m''-m'$ marked points lies in $X_{m+D!,d}^\circ$
		\item and the image of $u$ under the forgetful map forgetting the marked points labelled by $m+1,\cdots,m'$ lies in $X_{m+E!,d}^\circ$.
	\end{enumerate}
	Define \begin{equation}
		X_{m+D!+E!,d} \subset X_{m'',d}
	\end{equation} to be its closure.
	Let $V \to X$ be a vector bundle and let $C_\bullet = (C_{\check{m},\check{d}})_{\check{m} \in \bN, \ \check{d} \in H_2^+(X)}$ be a system of cones for $V$.
	Define $C^o_{m+D!+E!,d} \subset V_{m'',d}$ to be the
	restriction of $C_{m'',d}$ to $X^o_{m+D!+E!,d} \subset X_{m'',d}$.
	Similarly, define $C_{m+D!+E!,d} \subset V_{m'',d}$ to be the restriction of $C_{m'',d} \subset V_{m'',d}$ to $X_{m+D!+E!,d}$.

	Let $(V_i)_{i \in I}$ be a finite collection of vector bundles over $X$.
	Let $C^i_\bullet = (C^i_{\check{m},\check{d}})_{\check{m} \in \bN, \ \check{d} \in H_2^+(X)}$ be a system of cones for $V_i$ for each $i \in I$.
	Define
	\begin{equation} \label{eqwnproductofconesDE}
		\widetilde{\ev}(C_{m+D!+E!,d}^i)_{i \in I} := \widetilde{\ev}\left(\fprod{X_{m+D!+E!,d}}_{i \in I} C_{m+D!+E!,d}^i\right).
	\end{equation}
	We let $\overline{\widetilde{\ev}(C_{m+D!+E!,d}^i)_{i \in I}}$ be the closure of \begin{equation}
		\widetilde{\ev}(C_{m+D!+E!,d}^i)_{i \in I}
	\end{equation} in $\ev^*P_{(V^{\boxtimes m''})_{i \in I}}$.
	If $I = \{\star\}$ then we sometimes just write:
	\begin{equation}
		\overline{\widetilde{\ev}(C_{m+D!+E!,d}^\star)} = \overline{\widetilde{\ev}(C_{m+D!+E!,d}^i)_{i \in I}}.
	\end{equation}
\end{defn}

The same definitions above hold when $X$ is replaced with a smooth quasi-projective variety.

\begin{lemma} \label{lemmaperturbingdivisors}
	Let $(V_i)_{i \in I}$ be a finite collection of vector bundles over $X$.
	Let $C^i_\bullet$ be a system of cones for $V_i$ for each $i \in I$.
	Suppose $D, D', E, E'$
	are reduced divisors generically transverse to $C^i_\bullet$ for each $i \in I$
	and so that $D \sim D'$ and $E \sim E'$.
	Let $d \in H^2_+(X)$, $m \in \bN$, $m'' := m + D \cdot d + E \cdot d$
	and let $\alpha$ be a closed differential form in $X^{m''}$.
	Then
	\begin{equation} \label{eqnequalityofintegralsaddingpoints}
		\begin{aligned}
			\int_{P\ev_*[\overline{\widetilde{\ev}(C_{m+D!+E!,d})}]} \alpha = \int_{P\ev_*[\overline{\widetilde{\ev}(C_{m+D'!+E'!,d})}]} \alpha.
		\end{aligned}
	\end{equation}
\end{lemma}
\begin{proof}
	Let $p : \bC \times X \to X$ be the natural projection map.
	Let $(D_t)_{t \in \bC}$ be the one dimensional linear family of divisors satisfying $D_0 = D$ and $D_1 = D'$.
	Similarly, let $(E_t)_{t \in \bC}$ be the one dimensional linear family of divisors satisfying $E_0 = E$ and $E_1 = E'$.
	Define
	\begin{equation}
		\widetilde{D} := \{(t,x) \in \bC \times X : x \in D_t\}, \quad \widetilde{E} := \{(t,x) \in \bC \times X : x \in E_t\}.
	\end{equation}
	Define $\widetilde{D}_t := \widetilde{D} \cap (\{t\} \times X)$ and $\widetilde{E}_t := \widetilde{E} \cap (\{t\} \times X)$ for each $t \in \bC$.
	Let $\widetilde{V}_i := p^* V_i$ and let
	$\widetilde{C}^i_{\check{m},(0,\check{d})}$ be the pullback of $C^i_{\check{m},\check{d}}$ to $(\widetilde{V}_i)_{\check{m},(0,\check{d})}$
	for each $\check{m} \in \bN$, $\check{d} \in H_2^+(X)$ and $i \in I$.
	Also, let $(C^i_t)_{\check{m},(0,\check{d})}$ be the restriction of $\widetilde{C}^i_{\check{m},(0,\check{d})}$ to $(\{t\} \times X)_{\check{m},\check{d}}$ for each $t \in \bC$, $\check{m} \in \bN$, $\check{d} \in H_2^+(X)$ and $i \in I$.
	Define
	\begin{equation}
		C := \fprod{(\bC \times X)_{m+\widetilde{D}!+\widetilde{E}!}}_{i \in I} \widetilde{C}^i_{m+\widetilde{D}!+\widetilde{E}!,(0,d)}, \quad C^o := \fprod{(\bC \times X)^o_{m+\widetilde{D}!+\widetilde{E}!}}_{i \in I}{C^i}^o_{m+\widetilde{D}!+\widetilde{E}!,(0,d)}
	\end{equation}
	and
	\begin{equation}
		C_t := \fprod{(\{t\} \times X)_{m+\widetilde{D}_t!+\widetilde{E}_t!}}_{i \in I} (C^i_t)_{m+\widetilde{D}_t!+\widetilde{E}_t!,(0,d)}, \quad C^o_t := \fprod{(\{t\} \times X)_{m+\widetilde{D}_t!+\widetilde{E}_t!}}_{i \in I} (C^i_t)^o_{m+\widetilde{D}_t!+\widetilde{E}_t!,(0,d)}
	\end{equation}
	for each $t \in \bC$.
	Let $\Delta_t$ be the divisor equal to the preimage of the divisor $\{t\} \subset \bC$ under the natural map:
	\begin{equation}
		\ev^* P_{(\widetilde{V_i}^{\boxtimes m''})_{i \in I}} \to (\bC \times X)_{m+\widetilde{D}!+\widetilde{E}!,(0,d)} \to \bC.
	\end{equation}
	We let $\overline{\widetilde{\ev}(C)}$, $\overline{\widetilde{\ev}(C_t)}$ be the closures of $\widetilde{\ev}(C)$, $\widetilde{\ev}(C_t)$ in $\ev^*P_{(\widetilde{V_i}^{\boxtimes m''})_{i \in I}}$ for each $t \in \bC$.
	Let $\widetilde{\alpha} := p^* \alpha$.

	Since $D$, $D'$, $E$, $E'$ are generically transverse to $C_\bullet$, we have that
	$\widetilde{\ev}(C^o_t) = \widetilde{\ev}(C^o) \cap \Delta_t$ is dense in $\overline{\widetilde{\ev}(C)} \cap \Delta_t$ for $t=0,1$.
	Hence, $\overline{\widetilde{\ev}(C)} \cap \Delta_t = \overline{\widetilde{\ev}(C_t)}$ for all $t=0,1$.
	Hence:
	\begin{equation} \label{eqnperturbingEstuff}
		\begin{aligned}
			 & \int_{P\ev_*[\overline{\widetilde{\ev}(C_{m+D!+E!,d})}]} \alpha =
			\int_{P\ev_*[\overline{\widetilde{\ev}(C_0)}]} \widetilde{\alpha} =
			\int_{P\ev_*([\overline{\widetilde{\ev}(C)}] \cap \Delta_0)} \widetilde{\alpha} = \\ &
			\int_{P\ev_*([\overline{\widetilde{\ev}(C)}] \cap \Delta_1)} \widetilde{\alpha} =
			\int_{P\ev_*([\overline{\widetilde{\ev}(C_1)}])} \widetilde{\alpha} =
			\int_{P\ev_*[\overline{\widetilde{\ev}(C_{m+D!+E!,d})}]} \alpha.
		\end{aligned}
	\end{equation}
\end{proof}

\begin{lemma} \label{lemmawelldefinedprojectionmap}
	Let $(V_i)_{i \in I}$ be a finite collection of vector bundles over $X$.
	Suppose $C^i_\bullet$ is a system of cones for $V_i$ for each $i \in I$.
	Let $D$ and $E$
	be reduced divisors generically transverse to $C^i_\bullet$ with $E \sim D$.
	Let $d \in H^2_+(X)$ and $m \in \bN$, $m' := m + D \cdot d$ and $m'' := m' + E \cdot d$.
	Define $F : X^{m''} \to X^{m'}$ to be the projection map to the first $m'$ factors.
	Let $H$ be a divisor on $X^{m'}$.
	Then
	\begin{equation} \label{eqnequalityofintegralsaddingpoints1}
		\begin{aligned}
			 & \int_{P\ev_*[\overline{\widetilde{\ev}(C_{m+D!,d})}]} \exp(\infty_{(V_i^{\boxtimes m'})_{i \in I}} + \pi_{P_{(V_i^{\boxtimes m'})_{i \in I}}}^* H)                                 \\
			 & = \frac{1}{(E\cdot d)!}\int_{P\ev_*[\overline{\widetilde{\ev}(C_{m+D!+E!,d})}]} \exp(\infty_{(V_i^{\boxtimes m''})_{i \in I}} + \pi_{P_{(V_i^{\boxtimes m''})_{i \in I}}}^* F^*H).
		\end{aligned}
	\end{equation}
	Also, if $E = D$ and $\infty_{(V_i^{\boxtimes m'})_{i \in I}} + \pi_{P_{(V_i^{\boxtimes m'})_{i \in I}}}^* H$ is nef then the restriction of
	\begin{equation}
		\infty_{(V_i^{\boxtimes m''})_{i \in I}} + \pi_{P_{(V_i^{\boxtimes m''})_{i \in I}}}^* F^*H
	\end{equation}
	to $\overline{\widetilde{\ev}(C_{m+D!+E!,d})}$ is also nef.
\end{lemma}
\begin{proof}

	Let
	\begin{equation} \label{eqnprojectionmapcones1}
		\pi : {\widetilde{\ev}(C^i_{m+D!+D!,d})_{i \in I}} \to {\widetilde{\ev}(C^i_{m+D!,d})_{i \in I}}
	\end{equation}
	be the projection map projecting out the last $m''-m'$ factors of the product bundle $\ev^* (\oplus_{i \in I} V_i)^{\boxtimes m''}$ and then forgetting the last $m''-m'$ marked points.
	The map \eqref{eqnprojectionmapcones1} extends to a rational map:
	\begin{equation} \label{eqnprojec2}
		\overline{\pi} : \overline{{\widetilde{\ev}(C^i_{m+D!+D!,d})_{i \in I}}} \dashrightarrow \overline{\widetilde{\ev}(C^i_{m+D!,d})_{i \in I}}.
	\end{equation}
	This rational map is ill-defined in the region:
	\begin{equation} \label{eqnilldefinedregion}
		\overline{{\widetilde{\ev}(C^i_{m+D!+D!,d})_{i \in I}}} \cap (\ev^*P_{(V_i^{\boxtimes m''})_{i \in I}} - \ev^*\bigoplus_{i \in I} V_i^{\boxtimes m''}) \cap \ev^*(0 \times P_{(V_i^{\boxtimes m''-m'})_{i \in I}})
	\end{equation}
	which is empty.
	Hence, the map \eqref{eqnprojectionmapcones1} extends to a generically \'etale map:
	\begin{equation} \label{eqnprojec1}
		\overline{\pi} : \overline{{\widetilde{\ev}(C^i_{m+D!+D!,d})_{i \in I}}} \to \overline{\widetilde{\ev}(C^i_{m+D!,d})_{i \in I}}
	\end{equation}
	of degree $(D \cdot d)!$
	with the additional property that the restriction of $P\ev^*\infty_{(V_i^{\boxtimes m'})_{i \in I}}$ to the codomain of $\overline{\pi}$ pulls back to the restriction of $P\ev^*\infty_{(V_i^{\boxtimes m''})_{i \in I}}$ to the domain of $\overline{\pi}$.
	Hence:
	\begin{equation}
		\begin{aligned}
			\int_{P\ev_*[\overline{\widetilde{\ev}(C_{m+D!,d})}]} \exp(\infty_{(V_i^{\boxtimes m'})_{i \in I}} + \pi_{P_{(V_i^{\boxtimes m'})_{i \in I}}}^* H)                                                            \\
			= \frac{1}{(D\cdot d)!}\int_{P\ev_*[\overline{\widetilde{\ev}(C_{m+D!+D!,d})}]} \exp(\overline{\pi}^*\infty_{(V_i^{\boxtimes m'})_{i \in I}} + \overline{\pi}^* \pi_{P_{(V_i^{\boxtimes m'})_{i \in I}}}^* H) \\
			= \frac{1}{(D\cdot d)!}\int_{P\ev_*[\overline{\widetilde{\ev}(C_{m+D!+D!,d})}]} \exp(\infty_{(V_i^{\boxtimes m''})_{i \in I}} + \pi_{P_{(V_i^{\boxtimes m''})_{i \in I}}}^* F^*H)                             \\
			\overset{\textnormal{Lemma \ref{lemmaperturbingdivisors}}}{=}
			\frac{1}{(E\cdot d)!}\int_{P\ev_*[\overline{\widetilde{\ev}(C_{m+D!+E!,d})}]} \exp(\infty_{(V_i^{\boxtimes m''})_{i \in I}} + \pi_{P_{(V_i^{\boxtimes m''})_{i \in I}}}^* F^*H).
		\end{aligned}
	\end{equation}

	Now suppose that $\infty_{(V_i^{\boxtimes m'})_{i \in I}} + \pi_{P_{(V_i^{\boxtimes m'})_{i \in I}}}^* H$ is nef.
	Then $\overline{\pi}^*\infty_{(V_i^{\boxtimes m'})_{i \in I}} + \overline{\pi}^* \pi_{P_{(V_i^{\boxtimes m'})_{i \in I}}}^* H$ is nef
	and hence the restriction of $\infty_{(V_i^{\boxtimes m''})_{i \in I}} + \pi_{P_{(V_i^{\boxtimes m''})_{i \in I}}}^* F^*H$ to $\overline{\widetilde{\ev}(C_{m+D!+D!,d})}$ is also nef.
\end{proof}


\begin{lemma} \label{lemmaDDDfactorialinequality}
	Let $C_\bullet = (C_{m,d})_{m \in \bN, \ d \in H_2^+(X)}$ be a system of cones for $V$
	and suppose that $V$ is strongly convex.
	Let $D$, $E$ be smooth and generically transverse to $C_\bullet$ and so that $\calO(D)$ and $V^* \otimes \calO(D)$ are globally generated and $E \sim D$.
	Suppose $D$ is also $V$-large.
	Let $m \in \bN$ and $d \in H_2^+(X)$ and define $m' = m+ D \cdot d$ and $m'' = m' + E \cdot d$.
	Define $p_j : X^{m''} \to X$ to be the $j$th projection map for each $j=1,\cdots,m''$.
	Let $\widehat{D} = \sum_{j=m+1}^{m'} p_j^* D$.
	Then
	\begin{equation} \label{eqncapproductrhs}
		D\Vol(C_{m,d}) \leq \frac{1}{m! (D \cdot d)!(E \cdot d)!} \int_{P\ev_*[\overline{\widetilde{\ev}(C_{m'+E!,d})}] \cap \pi_{P_{V^{\boxtimes m''}}}^*\widehat{D}} \exp(\infty_{V^{\boxtimes m''}} + \sum_{j=1}^{m'}\pi_{P_{V^{\boxtimes m''}}}^* p_j^* D).
	\end{equation}
\end{lemma}
\begin{proof}
	Since $\calO(D)$ is globally generated and $D \sim E$, we can choose $E$ so that
	\begin{enumerate}[label=\alph*)]
		\item $E$ is still smooth and generically transverse to $C_\bullet$,
		\item $E$ is a small perturbation of $D$,
		\item the right-hand side of \eqref{eqncapproductrhs} is unchanged by Lemma \ref{lemmaperturbingdivisors} and
		\item \label{item:emptyintersection} there is an open dense subset $\nu$ of $C_{m,d}$ so that each curve $u$ in the image of $\nu$ in $X_{m,d}$ has the property that $u^{-1}(D \cap E)$ is empty.
	\end{enumerate}

	Let $D', E'$ be the pullbacks of $D$ and $E$ respectively to $\bC \times X$ via the natural projection map to $X$.
	Suppose $D'|_t := D'|_{t \times X}$ and $E'|_t := E'|_{t \times X}$ for each $t \in \bC$.
	Let $(D_j)_{j=m+1,\cdots,m'}$ be generic among tuples of smooth divisors in $\bC \times X$ linearly equivalent to $D$, each of which are transverse to $0 \times X$, and whose restriction to $0 \times X$ is $D$.
	Suppose $\widetilde{p}_j : (\bC \times X)^{m''} \to \bC \times X$ is the $j$th projection map for each $j=1,\cdots,m''$.
	Let $\widetilde{D} := \sum_{j=m+1}^{m'} \widetilde{p}_j^* D_j$.

	We let $\widetilde{U}$
	be the subspace of curves $u$ in $(\bC \times X)_{m'+E'!,(0,d)}$
	so that
	\begin{enumerate}
		\item the $j$th marked point maps to $D_j$ for each $j=m+1,\cdots,m'$ and
		\item if we remove the last $m'' - m$ marked points of $u$ then the resulting map is still stable (I.e. no irreducible components of the domain curve are contracted under the forgetful map).
	\end{enumerate}
	By \ref{item:emptyintersection} combined with the fact that $D$ and $E$ are generically transverse to $C_\bullet$, we have that there is a neighborhood $N \subset \widetilde{U}$ of the image of an open dense subset of $C_{m+D!+E!,d}$ in $X_{m+D!+E!,d} = X_{m+D'|_0!+E'|_0!,(0,d)}$.

	Let $(\bC \times X)_{m+\widetilde{D}^{\widetilde{D}}+E!,(0,d)}$ be the closure of $\widetilde{U}$
	inside $(\bC \times X)_{m'+E'!,(0,d)}$.
	Let $\widetilde{V}$ be the pullback of $V$ to $\bC \times X$ via the natural projection map
	and $C_{m,(0,d)} \subset \widetilde{V}_{m,(0,d)}$ the pullback of $C_{m,d}$ via the map taking a curve
	in $(\bC \times X)_{m,(0,d)}$
	and post composing it with the projection map to $X$.
	Also, let $C_{m+\widetilde{D}^{\widetilde{D}}+E'!,(0,d)} \subset \widetilde{V}_{m'',(0,d)}$ be the pullback
	of $C_{m,(0,d)}$ to $(\bC \times X)_{m+\widetilde{D}^{\widetilde{D}}+E'!,(0,d)}$ via the natural projection map
	forgetting the last $m'' - m$ marked points.
	Define $\overline{\widetilde{\ev}(C_{m+\widetilde{D}^{\widetilde{D}}+E'!,(0,d)})}$
	to be the closure of
	$\widetilde{\ev}(C_{m+\widetilde{D}^{\widetilde{D}}+E'!,(0,d)})$
	in $\ev^*P_{\widetilde{V}^{\boxtimes m''}}$.
	Let \begin{equation}
		(C_{m+\widetilde{D}^{\widetilde{D}}+E'!,(0,d)})_t \subset C_{m+\widetilde{D}^{\widetilde{D}}+E'!,(0,d)}
	\end{equation}
	be the restriction of $C_{m+\widetilde{D}^{\widetilde{D}}+E'!,(0,d)}$ to
	\begin{equation} \label{eqnttimesXspace}
		(\{t\} \times X)_{m'+E'|_t!,(0,d)} \cap (\bC \times X)_{m + \widetilde{D}^{\widetilde{D}}+E'!,(0,d)}.
	\end{equation}
	Let $\Delta_t$ be the Cartier divisor in $\ev^* P_{\widetilde{V}^{\boxtimes m''}}$ given by the preimage
	of the Cartier divisor $(\{t\} \times X)_{m'',(0,d)}$ inside $(\bC \times X)_{m''!,(0,d)}$ for each $t \in \bC$.
	Also, let $\overline{\widetilde{\ev}(C_{m+\widetilde{D}^{\widetilde{D}}+E'!,(0,d)})}_t$ be the closure of
	$\widetilde{\ev}(C_{m+\widetilde{D}^{\widetilde{D}}+E'!,(0,d)})_t$ in $\ev^* P_{\widetilde{V}^{\boxtimes m''}}$ for each $t \in \bC$.
	By Lemma \ref{lemmainjectivefromvlarge} this is isomorphic to the restriction of $\overline{\widetilde{\ev}(C_{m+\widetilde{D}^{\widetilde{D}}+E'!,(0,d)})}$
	to $\ev^* P_{\widetilde{V}^{\boxtimes m''}}|_{\textnormal{\eqref{eqnttimesXspace}}}$
	for each $t \in \bC$.
	Therefore, we have
	\begin{equation} \label{eqneffectiveintersection}
		[\overline{\widetilde{\ev}(C_{m+\widetilde{D}^{\widetilde{D}}+E'!,(0,d)})}] \cap \Delta_t = [\overline{\widetilde{\ev}(C_{m+\widetilde{D}^{\widetilde{D}}+E'!,(0,d)})_t}]
	\end{equation}
	for each $t \in \bC$ since any defining function for $\{t\} \subset \bC$ pulls back to a regular one on
	\begin{equation}
		\overline{\widetilde{\ev}(C_{m+\widetilde{D}^{\widetilde{D}}+E'!,(0,d)})}.
	\end{equation}

	We will identify $X_{m'+E!,d}$ with $(\{t\} \times X)_{m'+E'|_t!,(0,d)}$ for each $t \in \bC$
	via the natural identification $X = \{t\} \times X$.
	Let \begin{equation}
		C_{m+\widetilde{D}^{\widetilde{D}}+E'!,(0,d)}|_{\widetilde{U}} \subset C_{m+\widetilde{D}^{\widetilde{D}}+E'!,(0,d)}
	\end{equation}
	be the corresponding pullback to $\widetilde{U}$.
	Since $\widetilde{U}$ contains $N$, we have that a dense open subset of $C_{m+D!+E!,d}$ is also open in $(C_{m+\widetilde{D}^{\widetilde{D}}+E'!,(0,d)})_0$.
	Therefore, by \eqref{eqneffectiveintersection},
	we have that
	\begin{equation} \label{eqnpositivesum}
		[\overline{\widetilde{\ev}(C_{m+\widetilde{D}^{\widetilde{D}}+E'!,(0,d)})}] \cap \Delta_0 - [C_{m+D!+E!,d}]
	\end{equation}
	is represented by a positive sum of integral substacks
	in $CH_*(\ev^* P_{\widetilde{V}^{\boxtimes m''}})$.

	Since $V$ is strongly convex and $D$ is $V$-large, we have by Lemma \ref{lemmainjectivefromvlarge} that
	\begin{equation}
		\widetilde{\ev} : \widetilde{V}_{m' + D'!} \to \ev^* \widetilde{V}^{\boxtimes m''}
	\end{equation}
	is injective.
	Since $E$ is a small perturbation of $D$, the map
	\begin{equation}
		\widetilde{\ev} : \widetilde{V}_{m' + E'!} \to \ev^* \widetilde{V}^{\boxtimes m''}
	\end{equation}
	is also injective.
	Hence, the map
	\begin{equation} \label{eqnprojectionmapDD}
		\pi : \widetilde{\ev}(C_{m+\widetilde{D}^{\widetilde{D}} + E'!,d}) \to \ev^*\widetilde{V}^{\boxtimes m'}
	\end{equation}
	given by projecting out the last $m'' - m'$ factors and then forgetting the last $m''-m'$ marked points
	extends to a map
	\begin{equation} \label{eqnprojectionmapPDD}
		\overline{\pi} : \overline{\widetilde{\ev}(C_{m+\widetilde{D}^{\widetilde{D}} + E'!,d})} \to \ev^*P_{\widetilde{V}^{\boxtimes m'}}
	\end{equation}
	with the property that $\overline{\pi}^* \infty_{\widetilde{V}^{\boxtimes m'}}$ is equal to the restriction of $\infty_{\widetilde{V}^{\boxtimes m''}}$ to the domain of $\overline{\pi}$.
	Therefore, since $D_{\sum m',V}$ is nef by Corollary \ref{corollarynefsum},
	\begin{equation}
		P\ev^*(\infty_{\widetilde{V}^{\boxtimes m''}} + \sum_{j=1}^{m'} \pi^*_{\widetilde{V}^{\boxtimes m''}} \widetilde{p}_j^* D_j)|_{\overline{\widetilde{\ev}(C_{m+\widetilde{D}^{\widetilde{D}} + E'!,d})}} = \overline{\pi^*} D'_{\sum m', \widetilde{V}}
	\end{equation}
	is also nef.
	Combining this with the above stated fact that \eqref{eqnpositivesum} is represented by a positive sum of integral substacks in $\ev^* P_{\widetilde{V}^{\boxtimes m''}}$, we have
	\begin{equation} \label{eqndifficultpartDD}
		\begin{aligned}
			 & \int_{P\ev_*[\overline{\widetilde{\ev}(C_{m+D!,d})}]} \exp(D_{\sum m',V})
			\\& \overset{\textnormal{Lemma \ref{lemmawelldefinedprojectionmap}}}{=} \frac{1}{(E \cdot d)!}\int_{P\ev_*[\overline{\widetilde{\ev}(C_{m+D!+E!,d})}]} \exp(\infty_{V^{\boxtimes m''}} + \sum_{j=1}^{m'}\pi_{P_{V^{\boxtimes m''}}}^* p_j^* D)
			\\ & \leq \frac{1}{(E \cdot d)!}\int_{P\ev_*[\overline{\widetilde{\ev}(C_{m+\widetilde{D}^{\widetilde{D}}+
								E'!,(0,d)})}] \cap \Delta_0}
			\exp(\infty_{V^{\boxtimes m''}} + \sum_{j=1}^{m'}\pi_{P_{V^{\boxtimes m''}}}^* \widetilde{p}_j^* D_j).
		\end{aligned}
	\end{equation}

	Hence,
	\begin{equation}
		D\Vol(C_{m,d}) \overset{\textnormal{Definition \ref{defnvolumeformultiplecones}}}{=} \frac{1}{m! (D \cdot d)!} \int_{P\ev_*[\overline{\widetilde{\ev}(C_{m+D!,d})}]} \exp(D_{\sum m',V})
	\end{equation}
	\begin{equation}
		\overset{\textnormal{\eqref{eqndifficultpartDD}}}{\leq}
		\frac{1}{m! (D \cdot d)!(E \cdot d)!} \int_{P\ev_*[\overline{\widetilde{\ev}(C_{m+\widetilde{D}^{\widetilde{D}}+E'!,(0,d)})}] \cap \Delta_0} \exp(\infty_{\widetilde{V}^{\boxtimes m''}} + \sum_{j=1}^{m'}\pi_{P_{V^{\boxtimes m''}}}^* \widetilde{p}_j^* D_j)
	\end{equation}
	\begin{equation}  \label{eqnprecapproduct}
		= \frac{1}{m! (D \cdot d)!(E \cdot d)!} \int_{P\ev_*[\overline{\widetilde{\ev}(C_{m+\widetilde{D}^{\widetilde{D}}+E'!,d})}] \cap \Delta_1} \exp(\infty_{\widetilde{V}^{\boxtimes m''}} + \sum_{j=1}^{m'}\pi_{P_{V^{\boxtimes m''}}}^* \widetilde{p}_j^* D_j).
	\end{equation}
	Since $D_j|_{\{1\} \times X}$, $j=1,\cdots,m$ is a generic tuple of smooth divisors in the linear system containing $D$,
	we have the following equality of Chow classes:
	\begin{equation} \label{eqndivisorequant}
		\begin{aligned}
			P\ev_*([\overline{\widetilde{\ev}(C_{m+\widetilde{D}^{\widetilde{D}}+E'!,d})}] \cap \Delta_1 )
			 & \overset{\textnormal{\eqref{eqneffectiveintersection}}}{=} P\ev_*[\overline{\widetilde{\ev}(C_{m+\widetilde{D}^{\widetilde{D}}+E'!,d})_1}]
			\\ & = P\ev_*[\overline{\widetilde{\ev}(C_{m'+E'|_1!,d})}] \cap \pi_{P_{V^{\boxtimes m''}}}^*\widetilde{D}|_{(\{1\} \times X)^{m''}}
			\\ & = P\ev_*[\overline{\widetilde{\ev}(C_{m'+E!,d})}] \cap \pi_{P_{V^{\boxtimes m''}}}^*\widehat{D}
		\end{aligned}
	\end{equation}
	inside $P_{\widehat{V}^{\boxtimes m''}}$.
	Hence, \eqref{eqnprecapproduct} is equal to the right-hand side of \eqref{eqncapproductrhs}.


\end{proof}

\begin{prop} \label{propsegreboundsforDvolume}
	Let $C_\bullet = (C_{m,d})_{m \in \bN, \ d \in H_2^+(X)}$ be a system of cones for $V$ and suppose that $V$ is strongly convex.
	Let $D$ be $V$-large, generically transverse to $C_\bullet$ and so that $\calO(D)$ and $V^* \otimes \calO(D)$ are globally generated.
	Also, let $m \in \bN$ and $d \in H_2^+(X)$ and define $m' = m + D \cdot d$.
	Suppose that $C_{0,d} \to X_{0,d}$ has pure relative dimension $h$ and that $X_{0,d}$ is of pure dimension too.
	Then
	\begin{equation} \label{eqndvolumebound}
		D\Vol(C_{m,d}) \leq \frac{1}{h! m! (D\cdot d)!} \sum_{i_1,\cdots,i_{m'}=0}^{\dim X} \left|\langle D^{i_1},\cdots,D^{i_{m'}} \rangle_{m',d}^{s(C_\bullet)}\right|.
	\end{equation}
\end{prop}
\begin{proof}
	Let $m'' := m' + D \cdot d$.
	Define $p_j : X^{m''} \to X$ to be the $j$th projection map for each $j=1,\cdots,m''$.
	We also let $E$ be a generic linearly equivalent perturbation of $D$.
	Let $\widehat{D} = \sum_{j=m+1}^{m'} p_j^* D$.

	We have
	\begin{equation}
		D\Vol(C_{m,d})
	\end{equation}
	\begin{equation}
		\overset{\textnormal{Lemma \ref{lemmaDDDfactorialinequality}}}{\leq} \frac{1}{m! ((D \cdot d)!)^2} \int_{P\ev_*[\overline{\widetilde{\ev}(C_{m'+E!,d})}] \cap \pi_{P_{V^{\boxtimes m''}}}^*\widehat{D}} \exp(\infty_{V^{\boxtimes m''}} + \sum_{j=1}^{m'}\pi_{P_{V^{\boxtimes m''}}}^* p_j^* D)
	\end{equation}
	\begin{equation}
		=\frac{1}{m! ((D \cdot d)!)^2} \int_{P\ev_*[\overline{\widetilde{\ev}(C_{m'+E!,d})}] \cap \pi_{P_{V^{\boxtimes m''}}}^*\widehat{D}} \sum_{k = 0}^\infty \sum_{i_1,\cdots,i_{m'} = 0}^{\dim(X)} \frac{(\infty_{V^{\boxtimes m''}})^k}{k!} \pi_{P_{V^{\boxtimes m''}}}^*\left(\prod_{j=1}^{m'}p_j^*\frac{D^{i_j}}{i_j!}\right)
	\end{equation}
	\begin{equation}
		=\frac{1}{m! ((D \cdot d)!)^2} \int_{P\ev_*[\overline{\widetilde{\ev}(C_{m'+E!,d})}]} \sum_{k = 0}^\infty \sum_{i_1,\cdots,i_{m'} = 0}^{\dim(X)} \frac{(\infty_{V^{\boxtimes m''}})^k}{k!} \pi_{P_{V^{\boxtimes m''}}}^*\left(\prod_{j=1}^hp_j^*\frac{D^{i_j}}{i_j!}
		\prod_{j=m+1}^{m'} p_j^* \frac{D^{i_j+1}}{i_j!}\right)
	\end{equation}
	\begin{equation}
		\begin{aligned}
			 & \leq \frac{1}{m! ((D \cdot d)!)^2}  \sum_{k = 0}^\infty
			\sum_{i_1,\cdots,i_{m'} = 0}^{\dim(X)}                     \\ & \frac{1}{k! \prod_{j=1}^{m'} i_j!}
			\left| \int_{P\ev_*[\overline{\widetilde{\ev}(C_{m'+E!,d})}]}(\infty_{V^{\boxtimes m''}})^k \pi_{P_{V^{\boxtimes m''}}}^*\left(\prod_{j=1}^hp_j^*D^{i_j}
			\prod_{j=m+1}^{m'} p_j^* D^{i_j+1}\right) \right|
		\end{aligned}
	\end{equation}
	\begin{equation}
		\begin{aligned}
			 & \leq \frac{1}{m! ((D \cdot d)!)^2}  \sum_{k = 0}^\infty
			\sum_{i_1,\cdots,i_{m'} = 0}^{\dim(X)}                     \\ & \frac{1}{k!} \left| \int_{P\ev_*[\overline{\widetilde{\ev}(C_{m'+E!,d})}]}(\infty_{V^{\boxtimes m''}})^k \pi_{P_{V^{\boxtimes m''}}}^*\left(\prod_{j=1}^hp_j^*D^{i_j}
			\prod_{j=m+1}^{m'} p_j^* D^{i_j+1}\right) \right|
		\end{aligned}
	\end{equation}
	\begin{equation}
		\begin{aligned}
			 & \leq \frac{1}{m! ((D \cdot d)!)^2}  \sum_{k = 0}^\infty
			\sum_{i_1,\cdots,i_{m'} = 0}^{\dim(X)}                     \\ & \frac{1}{k!} \left| \int_{P\ev_*[\overline{\widetilde{\ev}(C_{m'+E!,d})}]}(\infty_{V^{\boxtimes m''}})^k \pi_{P_{V^{\boxtimes m''}}}^*\left(\prod_{j=1}^{m'}p_j^*D^{i_j}\right) \right|
		\end{aligned}
	\end{equation}
	\begin{equation}
		\begin{aligned}
			 & \overset{\textnormal{\eqref{eqnpushforwardclosurerelation}}}{\leq} \frac{1}{m! ((D \cdot d)!)^2}  \sum_{k = 0}^\infty
			\sum_{i_1,\cdots,i_{m'} = 0}^{\dim(X)}                                                                                   \\ & \frac{1}{k!} \left| \int_{P\ev_*P\widetilde{\ev}_*[\overline{(C_{m'+E!,d})}]}(\infty_{V^{\boxtimes m''}})^k \pi_{P_{V^{\boxtimes m''}}}^*\left(\prod_{j=1}^{m'}p_j^*D^{i_j}\right) \right|
		\end{aligned}
	\end{equation}
	\begin{equation}
		\begin{aligned}
			 & =  \frac{1}{m! ((D \cdot d)!)^2}  \sum_{k = 0}^\infty
			\sum_{i_1,\cdots,i_{m'} = 0}^{\dim(X)}                   \\ & \frac{1}{k!} \left| \int_{[\overline{C_{m'+E!,d}}]}P\widetilde{\ev}^*P\ev^*(\infty_{V^{\boxtimes m''}})^k P\widetilde{\ev}^*P\ev^*\pi_{P_{V^{\boxtimes m''}}}^*\left(\prod_{j=1}^{m'}p_j^*D^{i_j}\right) \right|
		\end{aligned}
	\end{equation}
	\begin{equation}
		\begin{aligned}
			\overset{\textnormal{\eqref{eqnpulbackofo1r}}}{=} \frac{1}{m! ((D \cdot d)!)^2}  \sum_{k = 0}^\infty
			\sum_{i_1,\cdots,i_{m'} = 0}^{\dim(X)} \\ \frac{1}{k!} \left| \int_{[\overline{C_{m'+E!,d}}]}\infty_{V_{m'+E!,d}}^k \pi_{P_{V_{m+E!,d}}}^*\left(\prod_{j=1}^{m'}\ev_j^*D^{i_j}\right) \right|
		\end{aligned}
	\end{equation}
	\begin{equation}
		= \frac{1}{m! ((D \cdot d)!)^2}  \sum_{k = h}^\infty
		\sum_{i_1,\cdots,i_{m'} = 0}^{\dim(X)} \frac{1}{k!} \left| \int_{[\overline{C_{m'+E!,d}}]}\infty_{V_{m'+E!,d}}^k \pi_{P_{V_{m+E!,d}}}^*\left(\prod_{j=1}^{m'}\ev_j^*D^{i_j}\right) \right|
	\end{equation}
	\begin{equation}
		\leq \frac{1}{h! m! ((D \cdot d)!)^2}  \sum_{k = h}^\infty
		\sum_{i_1,\cdots,i_{m'} = 0}^{\dim(X)} \left| \int_{[\overline{C_{m'+E!,d}}]}\infty_{V_{m'+E!,d}}^k \pi_{P_{V_{m+E!,d}}}^*\left(\prod_{j=1}^{m'}\ev_j^*D^{i_j}\right) \right|
	\end{equation}
	\begin{equation}
		\overset{\textnormal{Definition \ref{defnsegrefromsystemsofcones}}}{=} \frac{1}{h! m! (D \cdot d)!}  \sum_{k = h}^\infty
		\sum_{i_1,\cdots,i_{m'} = 0}^{\dim(X)} \left| \int_{[\overline{C_{m',d}}]}(\infty_{V_{m',d}})^k \pi_{P_{V_{m',d}}}^*\prod_{j=1}^{m'}\ev_j^*D^{i_j} \right|
	\end{equation}
	\begin{equation}
		\overset{\textnormal{Definitions \ref{defnsefsef} and \ref{defnsegrefromsystemsofcones}}}{=} \frac{1}{h! m! (D \cdot d)!}
		\sum_{i_1,\cdots,i_{m'} = 0}^{\dim(X)} \left| \left< D^{i_1},\cdots,D^{i_{m'}} \right>^{s(C_\bullet)} \right|.
	\end{equation}
	The last equality also holds because $C_{m',d}$ has pure dimension.

\end{proof}

\subsection{Bounds for Summands} \label{subsectionboundsforothersummands}

We need to compute how $D$-volume changes when we remove a cone (See Proposition \ref{propboundsforsummands}). This is important in the proof of the main theorem in Section \ref{sectionmainargument}.
Let $X$ be a smooth projective variety.

\begin{lemma} \label{lemmaineqaulityrestriction}
	Let $U \subset X$ be open and $I$ a finite set.
	Suppose $(D'_i)_{i \in I}$, $D$ are divisors on $X$ so that $D$ is smooth and let
	$(V_i)_{i \in I}$ be a finite collection of vector bundles over $X$.
	Also, let $C^i_\bullet$ be a system of cones for $V_i$ for each $i \in I$.
	Suppose that $\calO(D'_i)$ and $V_i^* \otimes \calO(D'_i)$ are globally generated for each $i \in I$
	and that $D - \sum_{i \in I} D'_i$ is nef.
	Then
	\begin{equation}
		D\Vol(C^i_{m,d})_{i \in I} \geq D\Vol(C^i_{m,d}|_U)_{i \in I}
	\end{equation}
	for each $m \in \bN$, $d \in H_2^+(X)$ (Definition \ref{defnrestrictionofcones}).
\end{lemma}
\begin{proof}
	Let $d \in H_2^+(X)$ and $m \in \bN$.
	Let $m' = m + D \cdot d$.
	By Corollary \ref{corollarynefsum2} we have that
	$D_{\sum m', (V_i)_{i \in I}}$ is nef.
	Hence,
	\begin{equation}
		D\Vol(C^i_{m,d})_{i \in I} \overset{\textnormal{Definition \ref{defnvolumeformultiplecones}}}{=}
		\frac{1}{m!(D \cdot d)!} \int_{P\ev_*\left[\overline{\widetilde{\ev}((C^i_{m+D!,d})_{i \in I})}\right]} \exp(D_{\sum m',(V_i)_{i \in I}})
	\end{equation}
	\begin{equation}
		\geq \frac{1}{m!(D \cdot d)!} \int_{P\ev_*\left[\overline{\widetilde{\ev}((C^i_{m+D!,d}|_U)_{i \in I})}\right]} \exp(D_{\sum m',(V_i)_{i \in I}})
		\overset{\textnormal{Definition \ref{defnvolumeformultiplecones}}}{=}  D\Vol(C^i_{m,d}|_U)_{i \in I}.
	\end{equation}
\end{proof}

\begin{lemma} \label{lemmaboundDvolbyanother}
	Let $(V_i)_{i \in I}$ be a finite collection of vector bundles over $X$.
	Let $C_\bullet^i$ be a system of cones for $V_i$ for each $i \in I$.
	Also, let $D'_i$ be a divisor on $X$ so that $\mathcal{O}(D'_i)$ and $V_i^* \otimes \mathcal{O}(D'_i)$ are globally generated for each $i \in I$. Let $D, F$ be smooth divisors on $X$ so that
	\begin{enumerate}
		\item they are generically transverse to $C_ \bullet^i$ for each $i \in I$,
		\item $\calO(D)$ is globally generated,
		\item  $D-\sum_{i \in I} D'_i$ is nef
		\item and $F \sim 2D$.
	\end{enumerate}
	Then
	$$D\Vol(C_{m,d}^i)_{i \in I} \leq \frac{1}{m!(F \cdot d)!}\int_{P\ev_*[\overline{\widetilde{\ev}(C_{m+F!,d})}]} \exp(D_{\sum m'', (V_i)_{i \in I}})$$
	for each $d \in H_2(X)^+$ and $m \in \mathbb{N}$ where $m'' = m + F \cdot d$.
\end{lemma}
\begin{proof}
	Since $\calO(D)$ is globally generated, we can choose a smooth linearly equivalent divisor $E$
	so that $D + E$ is reduced and both $E$ and $D + E$ are generically transverse to $C_\bullet^i$ for each $i \in I$.
	Let $d \in H_2(X)^+$ and $m \in \mathbb{N}$.
	Define $m' = m + D \cdot d$.
	Let $F : X^{m''} \to X^{m''-m'}$ be the projection map to the last $m''-m'$ factors.
	We have
	\begin{equation}
		D\Vol(C_{m,d}^i)_{i \in I} \overset{\textnormal{Definition \ref{defnvolumeformultiplecones}}}{=} \frac{1}{m!(D \cdot d)!} \int_{P\ev_*[\overline{\widetilde{\ev}(C^i_{m+D!,d})_{i\in I}}]} \exp(D_{\sum m', (V_i)_{i \in I}})
	\end{equation}
	\begin{equation} \label{eqnDDequationlast}
		\overset{\textnormal{Lemma \ref{lemmawelldefinedprojectionmap}}}{=} \frac{1}{m!(D \cdot d)!(D \cdot d)!} \int_{P\ev_*[\overline{\widetilde{\ev}(C^i_{m+D!+D!,d})_{i \in I}}]} \exp(D_{\sum m'', (V_i)_{i \in I}} -
		\pi_{P_{(V_i^{\boxtimes m''})_{i \in I}}}^*F^*D_{\sum m''-m'})
	\end{equation}
	(See Definition \ref{defndfactorialefactorial}).
	By Lemma \ref{lemmawelldefinedprojectionmap} combined with Corollary \ref{corollarynefsum2}, we have that the restriction of
	$D_{\sum m'', (V_i)_{i \in I}} -
		\pi_{P_{(V_i^{\boxtimes m''})_{i \in I}}}^*F^*D_{\sum m''-m'}$ to $\overline{\widetilde{\ev}(C^i_{m+D!+D!,d})}$ is nef. Also, $\pi_{P_{(V_i^{\boxtimes m''})_{i \in I}}}^*F^*D_{\sum m''-m'}$ is nef.
	Since the integral over a positive sum of integral substacks of a wedge product of nef classes is non-negative, we get that \eqref{eqnDDequationlast} is less than or equal to:
	\begin{equation}
		\frac{1}{m!(D \cdot d)!(D \cdot d)!} \int_{P\ev_*[\overline{\widetilde{\ev}(C^i_{m+D!+D!,d})_{i \in I}}]} \exp(D_{\sum m'', (V_i)_{i \in I}})
	\end{equation}
	\begin{equation}
		\overset{\textnormal{Lemma \ref{lemmaperturbingdivisors}}}{=}
		\frac{1}{m!(D \cdot d)!(E \cdot d)!} \int_{P\ev_*[\overline{\widetilde{\ev}(C^i_{m+D!+E!,d})_{i \in I}}]} \exp(D_{\sum m'', (V_i)_{i \in I}})
	\end{equation}
	\begin{equation}
		= \frac{1}{m!(D \cdot d + E \cdot d)!} \int_{P\ev_*[\overline{\widetilde{\ev}(C^i_{m+(D+E)!,d})_{i \in I}}]} \exp(D_{\sum m'', (V_i)_{i \in I}})
	\end{equation}
	\begin{equation}
		\overset{\textnormal{Lemma \ref{lemmaperturbingdivisors}}}{=} \frac{1}{m!(F \cdot d)!} \int_{P\ev_*[\overline{\widetilde{\ev}(C^i_{m+F!,d})_{i \in I}}]} \exp(D_{\sum m'', (V_i)_{i \in I}}).
	\end{equation}
\end{proof}

\begin{prop} \label{propboundsforsummands}
	Let $(V_i)_{i \in I}$ be a finite collection of vector bundles over $X$ and let $\star \in I$.
	Suppose that $V_\star|_U$ is strongly convex for some dense open subset $U \subset X$.
	Let $C^i_\bullet$ be a system of cones for $V_i$ for each $i \in I$ and suppose
	$C^\star_\bullet|_U = (V_\star)_\bullet|_U$.
	Suppose that $D'_i$ is a divisor in $X$ so that $\calO(D'_i)$ and $V_i^* \otimes \calO(D'_i)$ are globally generated for each $i \in I$.
	Suppose that $D$ and $F$ are smooth divisors on $X$ satisfying
	\begin{itemize}
		\item $D - \sum_{i \in I}D'_i$ is nef,
		\item
		      $F$ is $V_\star$-large,
		      (Definition \ref{defncompactieab}),
		\item $F \sim 2D$,
		\item and $D$ and $F$
		      are generically transverse to $(C^i_\bullet)_{i \in I}$ (Definition \ref{defngenericallytransverse}).
	\end{itemize}
	Then
	\begin{equation}
		D\Vol(C^i_{m,d}|_U)_{i \in I-\star} \leq g_d!F\Vol(C^i_{m,d})_{i \in I}
	\end{equation}
	where
	\begin{equation}
		g_d =
		c_1(V_\star)(d) + \dim(V_\star)
	\end{equation}
	for each $m \in \bN$ and $d \in H_2^+(X)-0$.
\end{prop}
\begin{proof}
	Let $d \in H_2^+(X)-0$ and $m \in \bN$.
	Let $m'' = m + F \cdot d$.
	Also, let
	\begin{equation}
		P_\star : P_{(V_i^{\boxtimes m''})_{i \in I}} \to P_{V_\star^{\boxtimes m''}}
	\end{equation}
	\begin{equation}
		P_\perp : P_{(V_i^{\boxtimes m''})_{i \in I}} \to P_{(V_i^{\boxtimes m''})_{i \in I-\star}}
	\end{equation}
	be the natural projection maps.
	We have
	\begin{equation} \label{eqndvolume2}
		D\Vol(C^i_{m,d}|_U)_{i \in I-\star}
	\end{equation}
	\begin{equation} \label{eqndvolume22}
		\overset{\textnormal{Lemma \ref{lemmaboundDvolbyanother}}}{\leq}
		\frac{1}{m!(F \cdot d)!} \int_{P\ev_*\left[\overline{\widetilde{\ev}((C^i_{m+F!,d}|_U)_{i \in I-\star})}\right]} \exp(D_{\sum m'', (V_i)_{i \in I-\star}}).
	\end{equation}
	Since $V_\star|_U$ is convex, we get that \begin{equation}
		(V_\star)_{m+F!,d}|_U = (C^\star)_{m+F!,d}|_U
	\end{equation} is a vector bundle over $X_{m+F!,d}|_U$ of dimension $g_d$. Combining this with Lemma \ref{lemmainjectivefromvlarge} and the fact that $F$ is $V_\star$-large, the map $P_\perp|_{\ev(\widetilde{\ev}((C^i_{m+F!,d}|_U))_{i \in I})}$
	has fibers equal to projective space of dimension $g_d$.
	Therefore,
	\eqref{eqndvolume22} is equal to
	\begin{equation} \label{eqnintegratefibers}
		\frac{1}{m!(F \cdot d)!} \int_{P\ev_*\left[\overline{\widetilde{\ev}((C^i_{m+F!,d}|_U)_{i \in I})}\right]} ((P_\star)^*\infty_{V_\star^{\boxtimes m''}})^{g_d} \cdot P_\perp^*\exp(D_{\sum m'', (V_i)_{i \in I-\star}})
	\end{equation}
	\begin{equation}
		=\frac{1}{m!(F \cdot d)!} \int_{P\ev_*\left[\overline{\widetilde{\ev}((C^i_{m+F!,d}|_U)_{i \in I})}\right]} ((P_\star)^*D_{\sum m'', V_\star})^{g_d} \cdot P_\perp^*\exp(D_{\sum m'', (V_i)_{i \in I-\star}})
	\end{equation}
	\begin{equation}
		=\frac{g_d!}{m!(F \cdot d)!} \int_{P\ev_*\left[\overline{\widetilde{\ev}((C^i_{m+F!,d}|_U)_{i \in I})}\right]} \frac{((P_\star)^*D_{\sum m'', V_\star})^{g_d}}{g_d!} \cdot P_\perp^*\exp(D_{\sum m'', (V_i)_{i \in I-\star}}).
	\end{equation}
	Since the integral of a wedge product of nef classes is non-negative, we get by Corollary \ref{corollarynefsum2} that this is less than or equal to
	\begin{equation}
		\leq \frac{g_d!}{m!(F \cdot d)!} \sum_{r \geq 0} \int_{P\ev_*\left[\overline{\widetilde{\ev}((C^i_{m+F!,d}|_U)_{i \in I})}\right]} \frac{((P_\star)^*D_{\sum m'', V_\star})^{r}}{r!} \cdot P_\perp^*\exp(D_{\sum m'', (V_i)_{i \in I-\star}})
	\end{equation}
	\begin{equation}
		\leq \frac{g_d!}{m!(F \cdot d)!}\int_{P\ev_*\left[\overline{\widetilde{\ev}((C^i_{m+F!,d}|_U)_{i \in I})}\right]} \exp((P_\star)^*D_{\sum m'', V_\star}) \cdot \exp(P_\perp^*D_{\sum m'', (V_i)_{i \in I-\star}})
	\end{equation}
	\begin{equation}
		= \frac{g_d!}{m!(F \cdot d)!}\int_{P\ev_*\left[\overline{\widetilde{\ev}((C^i_{m+F!,d}|_U)_{i \in I})}\right]} \exp((P_\star)^*D_{\sum m'', V_\star} + P_\perp^*D_{\sum m'', (V_i)_{i \in I-\star}})
	\end{equation}
	\begin{equation}
		= \frac{g_d!}{m!(F \cdot d)!} \int_{P\ev_*\left[\overline{\widetilde{\ev}((C^i_{m+F!,d}|_U)_{i \in I})}\right]} \exp(F_{\sum m'', (V_i)_{i \in I}})
	\end{equation}
	\begin{equation}
		\overset{\textnormal{Definition \ref{defnvolumeformultiplecones}}}{=} g_d!F\Vol(C^i_{m,d}|_U)_{i \in I}
		\overset{\textnormal{Lemma \ref{lemmaineqaulityrestriction}}}{\leq} g_d!F\Vol(C^i_{m,d})_{i \in I}.
	\end{equation}

\end{proof}

\subsection{Curves in Compactifications of Vector Bundles} \label{curvesincompatificationsofconvexvecctorbundles}

We need to compare the $E$-volume of the projective compactification of a vector bundle with the $D$-volume of its base for appropriate $D$ and $E$.
We will use the following well known lemma.

\begin{lemma} \label{lemmaequivariantpositiity1}
	Let $Z$ be a proper variety with a $(\bC^*)^k$-action and let $L \to Z$ be a line bundle.
	Then $L$ is nef if and only if
	$c_1(L) \cdot C \geq 0$ for each $(\bC^*)^k$-equivariant proper curve $C$ in $Z$.
\end{lemma}
\begin{proof}
	If $L$ is nef then $c_1(L) \cdot C \geq 0$ for each $(\bC^*)^k$-equivariant proper curve $C$ in $Z$ by definition.

	Now suppose that $c_1(L) \cdot C \geq 0$ for each $(\bC^*)^k$-equivariant proper curve $C$ in $Z$.
	It is sufficient for us to show that
	$c_1(L) \cdot C \geq 0$ for each proper irreducible curve $C$ in $Z$.
	Let us fix such a proper irreducible curve $C$ in $Z$.

	Let $d \in H_2^+(Z)$ be its homology class and $g$ the genus of its normalization.
	Suppose $f : \widetilde{C} \to Z$ is the corresponding map from its normalization to $Z$.
	Also, let $Z_{g,0,d}$ be the moduli space of stable genus $g$ nodal curves mapping to $Z$ representing $d$.
	Recall that $\underline{Z_{g,0,d}}$ is its coarse moduli space.
	Then there is a one parameter subgroup $\bC^* \subset (\bC^*)^k$
	so that its fixed point locus on $\underline{Z_{g,0,d}}$ is the same as the corresponding $(\bC^*)^k$ fixed point locus.
	Let us fix this $\bC^*$ action on $\underline{Z_{g,0,d}}$.

	Then $f$ corresponds to an element $[f]$ of $\underline{Z_{g,0,d}}$. Let $\widehat{f} : \bC^* \to \underline{Z_{g,0,d}}$ be the $\bC^*$ orbit satisfying $\widetilde{f}(0) = [f]$.
	Since $\underline{Z_{g,0,d}}$ is proper, this extends to a $\bC^*$-equivariant map $F : \bP^1 \to \underline{Z_{g,0,d}}$.
	Any curve $u : C' \to Z$ in $Z_{g,0,d}$ mapping to $F(\infty)$ is invariant under the $\bC^*$-action.
	Therefore, $c_1(L) \cdot u_*(C') \geq 0$. Since $u$ is homologous to $C$, we get $c_1(L) \cdot C = c_1(L) \cdot u_*(C') \geq 0$.

\end{proof}

\begin{lemma} \label{lemmmasmoothingsums}
	Let $X$ be a smooth projective variety.
	Let $(V_i)_{i \in I}$ be a finite collection of vector bundles over $X$ and let $\check{I} \subset I$.
	Define $V_{\check{I}} = \bigoplus_{i \in \check{I}} V_i$.
	Let $D'_i \subset X$ be a reduced divisor so that $V_i^* \otimes \calO(D'_i)$ and $\calO(D'_i)$ are globally generated for each $i \in I$ and let $D$ be a reduced divisor so that $D - \sum_{i \in I} D'_i$ is nef.
	Let
	\begin{equation}
		P_0 := P_{(V_i)_{i \in (I-I') \sqcup \{\check{I}\}}}, \quad P_1 := P_{(V_i)_{i \in I}}
	\end{equation}
	and $\pi_j : P_j \to X$, $j=0,1$ the natural projection maps.
	Suppose $f : \chi \to X$ is a proper morphism from a stack $\chi$.
	Let $C \subset f^*\oplus_{i \in I} V_i$ be a substack that is invariant under the $\bC^* \times \bC^*$ action given by
	\begin{equation} \label{eqnactiononvectorbundle}
		\begin{aligned}
			(\bC^* \times \bC^*) \times f^*\oplus_{i \in I} V_i \to f^*\oplus_{i \in I} V_i \\
			(\xi_1,\xi_2) \cdot (v_i)_{i \in I} \to ((\xi_1 v_i)_{i \in I-\check{I}}, (\xi_2 v_i)_{i \in \check{I}}).
		\end{aligned}
	\end{equation}
	We let $C_0$ be its closure inside $f^*P_0$
	and $C_1$ its closure inside $f^*P_1$.
	Let $Pf_* : f^*P_j \to P_j$ be the natural map covering $f$ for $j=0,1$.
	Then
	\begin{equation} \label{inequalitysmoothedvspresmoothed}
		\int_{Pf_*[C_0]} \exp(\pi_0^*D + \infty_{(V_i)_{i \in (I-\check{I})\sqcup \{\check{I}\}}}) \leq \int_{Pf_*[C_1]} \exp(\pi_1^*D + \infty_{(V_i)_{i \in I}}).
	\end{equation}
\end{lemma}

\begin{remark}
	Note that this inequality is very poor in some cases. For instance, if $X$ is a point, $V_i = \bC$ for each $i \in I$, $f$ is the identity map and $C = \bigoplus_{i \in I} V_i$, then the left-hand side of \eqref{inequalitysmoothedvspresmoothed} is the volume of $\bP^{|I|}$ and the right-hand side of \eqref{inequalitysmoothedvspresmoothed} is the volume of $(\bP^1)^{|I|}$ which is $|I|!$ times bigger.
\end{remark}

\begin{proof}[Proof of Lemma \ref{lemmmasmoothingsums}]
	We have that $Pf_*[C_0]$ and $Pf_*[C_1]$ are sums $\sum_{j=0}^k \alpha_j [\scrC^0_j]$ and $\sum_{j=0}^k \alpha_j [\scrC^1_j]$
	respectively where $\alpha_1,\cdots,\alpha_k \geq 0$ and $\scrC^0_j$ and $\scrC^1_j$ are closures of the same $\bC^* \times \bC^*$-invariant integral subscheme $\scrC_j \subset \oplus_{i \in I} V_i$ in $P_0$ and $P_1$ respectively for each $j=1,\cdots,k$ where $\bC^* \times \bC^*$ acts on \begin{equation}
		\oplus_{i \in I} V_i = \left(\oplus_{i \in I-\check{I}} V_i\right) \bigoplus \left(\oplus_{i \in \check{I}} V_i\right)
	\end{equation}
	by scaling the first summand by the first $\bC^*$ factor and scaling the second summand by the second $\bC^*$ factor.
	Therefore, by linearity, it is sufficient to prove that Equation \eqref{inequalitysmoothedvspresmoothed} holds where $\chi = X$ and $f$ is the identity map and hence $C \subset \oplus_{i \in I} V_i$, $C_0 \subset P_0$ and $C_1 \subset P_1$.

	Let \begin{equation}
		P' := P_{(V_i)_{i \in I - \check{I}}} \subset P_1.
	\end{equation}
	Let $B$ be the blowup of $0 \times P'$ inside $\bP^1 \times P_1$ and let $E \cong P_0$ be the exceptional divisor.
	Define $p : B \to \bP^1$ to be the composition
	\begin{equation}
		B \lra{\BL} \bP^1 \times P_1 \lra{} \bP^1
	\end{equation}
	where the first map is the blowdown map and the second map is the natural projection map.
	Define
	\begin{equation} \label{eqndefnDhat}
		\widehat{D} := \BL^*(\bP^1 \times (\pi_1^*D + \infty_{(V_i)_{i \in I}})) + p^{-1}\{\infty\} - E.
	\end{equation}
	We will now prove that $\widehat{D}$ is nef.
	We have a $\bC^* \times \bC^*$ action on $B$ extending the one on $(\bC-0) \times \oplus_{i \in I} V_i$ where $(\xi_1,\xi_2) \in \bC^* \times \bC^*$ sends $(z,(v_i)_{i \in I})$
	to $(\xi_1 z, (\xi_2 v)_{i \in I})$.
	By Lemma \ref{lemmaequivariantpositiity1},
	it is sufficient to check that $\widehat{D}$ pairs non-negatively with each $\bC^* \times \bC^*$-equivariant proper irreducible curve.

	Let us fix such a curve $C$.
	If $C$ is contained in $E$ then it pairs non-negatively with $\widehat{D}$ since the restriction of $\widetilde{D}$ to $E \cong P_0$ is $\pi_0^*D + \infty_{(V_i)_{i \in (I-\check{I})\sqcup \{\check{I}\}}}$ which is nef by Corollary \ref{corollarynefsum2}.
	If $C$ is disjoint from $E$ then it
	is contained in $(\bP^1-0) \times P_1$
	and since $\widehat{D} = (\bP^1 - 0) \times (\pi_1^*D + D_{(V_i)_{i \in I}}) + p^{-1}(\{\infty\})$ in this region we get it pairs non-negatively with $\widehat{D}$
	by Corollary \ref{corollarynefsum2} too.
	If $C$ is contained in the proper transform of $\bP^1 \times x$ for some $x \in X$ then $\widehat{D} \cdot C = p^{-1}(\infty) \cdot C - E \cdot C \geq 0$ since $p^{-1}(\infty) \cdot C = 1$ and $E \cdot C \leq 1$.
	Also, if $C$ is any other $(\bC^* \times \bC^*)$-equivariant curve then it is the closure of the proper transform of $0 \times L$ in $\bP^1 \times \oplus_{i \in I} V_i$ where $L$ is a line in a fiber of $\oplus_{i \in I} V_i$.
	Such a line intersects $E$ transversely once, $\bP^1 \times \BL^*\infty_{(V_i)_{i \in I}}$ at least once and the remaining summands of \eqref{eqndefnDhat} zero times. Hence, $\widehat{C} \cdot C \geq 0$ too.
	Hence, $\widehat{D}$ is nef.

	Let $\widehat{C}$ be the closure of the proper transform of $\bP^1 \times C_1$ in $B$.
	Since $C$ is invariant under the action \eqref{eqnactiononvectorbundle} and since $\widehat{D}$ is nef, we have that
	we have
	\begin{equation}
		\int_{[C_0]} \exp(\pi_0^*D + \infty_{(V_i)_{i \in (I-\check{I})\sqcup \{\check{I}\}}}) =
		\int_{[\widehat{C}] \cap [E]} \exp(\widehat{D})
	\end{equation}
	\begin{equation}
		\leq \int_{[\widehat{C}] \cap [p^{-1}(0)]} \exp(\widehat{D})
	\end{equation}
	\begin{equation}
		=\int_{[\widehat{C}] \cap [p^{-1}(1)]} \exp(\widehat{D})
		= \int_{[C_1]} \exp(\pi_1^*D + \infty_{(V_i)_{i \in I}}).
	\end{equation}
\end{proof}

\begin{prop} \label{propcurvesincompatificationsofconvexvecctorbundles}
	Let $V$ be a strongly convex vector bundle over a smooth projective variety $X$.
	Suppose $(V_i)_{i \in I}$ is a finite collection of vector bundles over $P_V$ so that $V_i = \pi_{P_V}^*(V_i|_X)$ for each $i \in I$.
	Let $m \in \bN$ and $d \in H_2^+(X)$ and let $\delta \in H_2^+(P_V)$ be the image of $d$ induced by the natural inclusion map.
	Let $C^i_\bullet$ be a system  of cones for $V_i$ and $\underline{C}^i_\bullet$ a system of cones for $V_i|_X$ for each $i \in I$. Suppose that $C^i_{m,\delta}|_{V_{m,d}} = \Pi_i^*\underline{C}^i_{m,d}$ via the natual identification $\Pi_i^*(V_i|_X)_{m,d} = (V_i)_{m,d}$ where $\Pi_i : V_{m,d} \to X_{m,d}$ is the natural projection map for each $i \in I$.
	Define $\underline{C}^V_\bullet := V_\bullet$.

	Let $D$ be a smooth divisor in $X$ so that $D-\sum_{i \in I} D'_i - D'$ is nef where $\calO(D'_i)$ and $V_i^*|_X \otimes \calO(D'_i)$ are globally generated for some divisor $D'_i$ in $X$ for each $i \in I$ and $V^* \otimes \calO(D')$ is globally generated for some divisor $D'$ on $X$.
	Let $E$ be a smooth divisor on $P_V$ which is generically transverse to $(C^i_\bullet)_{i \in I}$ so that $E - \pi_{P_V}^*D - \infty_V$ is nef
	and $E - \sum_{i \in I} E'_i$ is nef with $\calO(E'_i)$ and $V_i^* \otimes \calO(E'_i)$ globally generated for some divisor $E'_i$ for each $i \in I$.
	Then
	\begin{equation}
		E\Vol(C^i_{m,\delta})_{i \in I} \geq D\Vol(\underline{C}^i_{m,d})_{i \in I \sqcup \{V\}}.
	\end{equation}
\end{prop}
\begin{proof}
	Let $m' := m + D \cdot d$.
	Let
	$p_i : X^{m'} \to X$
	be the $i$th projection map for each
	$i=1,\cdots,m'$.
	Also, let
	\begin{equation}
		\begin{aligned}
			P_0       & \equiv P_{V^{\boxtimes m'}} \times_{X^{m'}} P_{((V_i|_X)^{\boxtimes m'})_{i \in I}} = P_{V^{\boxtimes m'},  ((V_i|_X)^{\boxtimes m'})_{i \in I}},
			\\
			\quad P_1 & \equiv (P_V^{m'}) \times_{X^{m'}} P_{((V_i|_X)^{\boxtimes m'})_{i \in I}} = P_{((p_i^*V)_{i \in \{1,\cdots,m'\}},(V_i|_X)^{\boxtimes m'})_{i \in I}}
			= P_{(V_i^{\boxtimes m'})_{i \in I}}
		\end{aligned}
	\end{equation}
	be two different compactifications of $V^{\boxtimes m'} \oplus \bigoplus_{i \in I} (V_i|_X)^{\boxtimes m'}$.
	Suppose
	\begin{equation}
		\begin{aligned}
			 & \pi_j : P_j \to X^{m'},  \quad j=1,2,               \\
			 & p : P_1 \to (P_V)^{m'},                             \\
			 & P : P_1 \to P_{((V_i|_X)^{\boxtimes m'})_{i \in I}}
		\end{aligned}
	\end{equation}
	are the natural projection maps.

	Let $C_0$ and $C_1$ be the closures of \begin{equation}
		\widetilde{\ev}(V_{m+D!,d},(\underline{C}^i_{m+D!,d})_{i \in I})
	\end{equation}
	(Definition \ref{defnvolumeformultiplecones}) inside $\ev^*P_0$ and  $\ev^*P_1$
	respectively.
	Since $E - \pi_{P_V}^*D - \infty_V$ is nef,
	we get that
	\begin{equation}
		p^* E_{\sum m'} + P^*\infty_{(V_i|_X)_{i \in I}} - (\pi_1^*D_{\sum m'} + \infty_{(p_i^*V)_{i \in \{1,\cdots,m'\}},((V_i|_X)^{\boxtimes m'})_{i \in I}})
	\end{equation}
	\begin{equation}
		= p^*\left((E - \pi_{P_V}^*D - \infty_V)_{\sum m'} \right)
	\end{equation}
	is nef (see Definition \ref{defnprojectivecompactification}).
	Combining this with the fact that $E - \sum_{i \in I}E'_i$ is nef we get
	\begin{equation}
		E\Vol(C^i_{m,\delta})_{i \in I}
		\overset{\textnormal{Lemma \ref{lemmaineqaulityrestriction}}}{\geq}
		E\Vol(C^i_{m,\delta}|_V)_{i \in I}
	\end{equation}
	\begin{equation}
		\overset{\textnormal{Definition } \ref{defnvolumeformultiplecones}}{=}
		\frac{1}{m!(E \cdot \delta)!}\int_{P\ev_*[\overline{\widetilde{\ev}((C^i_{m+D!,d}|_V)_{i \in I})}]} \exp(E_{\sum m', (V_i)_{i \in I}})
	\end{equation}
	\begin{equation}
		=
		\frac{1}{m!(E \cdot \delta)!}\int_{P\ev_*[C_1]} \exp(p^* E_{\sum m'} + \infty_{(\pi_V^*(V_i|_X))_{i \in I}})
	\end{equation}
	\begin{equation}
		\geq
		\frac{1}{m!(E \cdot \delta)!} \int_{P\ev_*[C_1]} \exp(\pi_1^*D_{\sum m'} + \infty_{(p_i^*V)_{i \in \{1,\cdots,m'\}},((V_i|_X)^{\boxtimes m'})_{i \in I}})
	\end{equation}
	\begin{equation} \label{eqnintegraloverC1}
		\geq
		\frac{1}{m!(D \cdot d)!} \int_{P\ev_*[C_1]} \exp(\pi_1^*D_{\sum m'} + \infty_{(p_i^*V)_{i \in \{1,\cdots,m'\}},((V_i|_X)^{\boxtimes m'})_{i \in I}}).
	\end{equation}
	By Lemma \ref{lemmmasmoothingsums}, we have \eqref{eqnintegraloverC1} is greater than or equal to
	\begin{equation}
		\frac{1}{m!(D \cdot d)!}\int_{P\ev_*[C_0]} \exp(\pi_0^*D_{\sum m'} + \infty_{V^{\boxtimes m'},((V_i|_X)^{\boxtimes m'})_{i \in I}})
		= D\Vol(\underline{C}^i_{m,d})_{i \in I \sqcup \{V\}}.
	\end{equation}

\end{proof}

\subsection{From Many Cones to One Cone} \label{subsectionfrommanyconestoonecone}

The aim of this section is to bound the $D$-volume of a fiber product of cones. This is because in Section \ref{subsectionvfcfromcones} we have many cones that need to be consolidated into one cone in order to give an upper bound for a Gromov-Witten count.
We fix a smooth projective variety $X$.

\begin{prop} \label{propmallestcompactificationinequality}
	Let $(V_i)_{i \in I}$ be a finite collection of vector bundles over $X$.
	Let $C^i_\bullet$ be a system of cones for $V_i$ (Definition \ref{defnsystemofcones}) for each $i \in I$.
	Suppose that there is a divisor $D'_i$ in $X$ so that $\calO(D'_i)$ and $V_i^* \otimes \calO(D'_i)$ are globally generated for each $i \in I$.
	Let $D$ be a divisor in $X$
	so that $D - \sum_{i \in I}D'_i$ is nef.
	Then for each $m \in \bN$ and $d \in H_2^+(X)$,
	\begin{equation}
		D\Vol((C^i_{m,d})_{i \in I}) \geq D\Vol\left(\fprod{X_{m,d}}_{i \in I}C^i_{m,d}\right).
	\end{equation}
\end{prop}
\begin{proof}
	Let $C_0$ be the closure of $\widetilde{\ev}((C^i_{m,d})_{i \in I})$ inside $\ev^*P_{(V_i^{\boxtimes m'})_{i \in I}}$
	and $C_1$ its closure in
	\begin{equation}
		\ev^*P_{\bigoplus_{i \in I} V_i^{\boxtimes m'}} = \ev^*P_{(\oplus_{i \in I} V_i)^{\boxtimes m'}}.
	\end{equation}
	We have
	\begin{equation}
		D\Vol((V_i)_{i \in I}) = \int_{P\ev_*[C_0]} \exp(D_{\sum m', (V_i)_{i \in I}})
	\end{equation}
	\begin{equation}
		\overset{\textnormal{Lemma \ref{lemmmasmoothingsums}}}{\geq}
		\int_{P\ev_*[C_1]} \exp(D_{\sum m', \bigoplus_{i \in I} V_i})
		= D\Vol\left(\fprod{X_{m,d}}_{i \in I}C^i_{m,d}\right).
	\end{equation}

\end{proof}

\section{Bounds for Normal Cones.} \label{boundsfornormalcones}

In this section we wish to bound the $D$-volume of the normal cone of a moduli space of curves inside a larger one. We will need this in the proof of the main theorem in Section \ref{sectionmainargument} since normal cones are needed when computing the Virtual fundamental class (See Section \ref{sectioncomprisinvfc}).

\subsection{Divisors Compatible with Blowups} \label{sectiondivisorscompatiblewithblowups}

Normal cones are connected with blowups, so we need a blowup compatibility condition for divisors.
Let $X$ be a smooth projective subvariety of a smooth projective variety $Y$ and let $\iota_X : X \hookrightarrow Y$ be the corresponding inclusion map.
We let $\pi : V \to X$ be the normal bundle of $X$ inside $Y$.

\begin{defn} \label{defneac}
	A divisor $D$ on $Y$ is \emph{$\BL_XY$-compatible}
	if $\BL^*D - E$ is nef on the blowup $\BL_XY$ of $Y$ along $X$ where $\BL : \BL_XY \to Y$ is the blowdown map.
\end{defn}

We let $z$ be the coordinate on $\bP^1$ after fixing an identification $\bP^1 = \bC \cup \{\infty\}$.

\begin{lemma} \label{lemmacompaet}
	If $D$ is a $\BL_XY$ compatible nef divisor
	then $p_Y^*(D) + p_{\bP^1}^*\{\infty\}$ is $\BL_{0 \times X} (\bP^1 \times Y)$-compatible
	where $p_Y : \bP^1 \times Y \to Y$ and $p_{\bP^1} : \bP^1 \times Y \to \bP^1$ are the natural projection maps.
\end{lemma}
\begin{proof}
	We have a $\bC^*$ action on $\bP^1 \times Y$ given by sending $(z,y) \in \bP^1 \times Y$ to $(\xi z, y)$ for each $\xi \in \bC^*$.
	This lifts uniquely to a $\bC^*$-action on $\overline{\chi} := \BL_{0 \times X} (\bP^1 \times Y)$.
	Let $E'$ be the exceptional divisor of
	$\overline{\chi}$ and $E$ the exceptional divisor of $\BL_XY$.
	Let $\BL : \overline{\chi} \to \bP^1 \times X$ be the blowdown map.

	We need to show that $D' := \BL^*(p_Y^*D + p_{\bP^1}^*\{\infty\}) - E'$ is nef.
	By Lemma \ref{lemmaequivariantpositiity1} it is sufficient to check that this divisor is non-negative on $\bC^*$-equivariant proper irreducible curves.
	If such a curve $C$ is contained $\BL_XY = \BL_{0 \times X} (0 \times Y) \subset \BL_{0 \times X} (\bP^1 \times Y)$, then $D' \cdot C = (\BL_X^*D-E) \cdot C \geq 0$.
	If it is a different curve then it is either
	\begin{enumerate}
		\item \label{item1} contained in $\{\infty\} \times Y$ or
		\item \label{item2} is in a fiber of $\BL$ or
		\item \label{item3} is the proper transform of $\bP^1 \times \{y\}$  for some $y \in Y$.
	\end{enumerate}
	In case \eqref{item1} $D' \cdot C \geq 0$ since $D$ is nef. In case \eqref{item2} $D' \cdot C \geq 0$  holds due to the fact that the dual of the normal bundle of $E'$ is $\calO_{\bP(V \oplus \bC)}(1)$.
	Also, in case \eqref{item3} $D' \cdot C \geq 0$ since the curve intersects  $E'$ at most once, but it intersects $\BL^*(p_{\bP^1}^*\{\infty\})$ exactly once and has intersection number zero with $\BL^*(p_Y^*D)$.
	Hence, $D'$ is nef.
\end{proof}

The following lemma gives a sufficient condition for $D$ to be $\BL_XY$-compatible. We will use this in Section \ref{sectionmainargument}.

\begin{lemma} \label{lemmanicesufficientcondtionforblxycompatible}
	Let $D$ be a divisor in $Y$ so that $\calO(D)$ and $\calO(D) \otimes \frI_X$ are globally generated where $\frI_X$ is the ideal sheaf of $X$ in $Y$.
	Then $D$ is $\BL_XY$-compatible.
	Also, $\iota_X^*\calO(D) \otimes V^*$ is globally generated.
\end{lemma}
\begin{proof}
	Since $\calO(D)$ is globally generated, we can perturb $D$ so that its support does not contain any irreducible component of $X$, allowing us to define $D|_X := \iota_X^* D$.
	We have an exact sequence:
	\begin{equation}
		\calO_Y^a(-D) \to \frI_X \to 0
	\end{equation}
	for some $a \in \bN$
	since  $\calO(D) \otimes \frI_X$ is globally generated.
	Since taking pullbacks is right exact and since $\iota_X^*\frI_X = V^*$, we get an exact sequence
	\begin{equation}
		\calO_X(-D|_X) \to V^* \to 0.
	\end{equation}
	Hence, $V^* \otimes \calO(D|_X) = \iota_X^*\calO(D) \otimes V^*$ is globally generated.

	Let $E$ be the exceptional divisor of the blowdown map $\BL : \BL_XY \to Y$.
	Let $C$ be a proper irreducible curve in $\BL_XY$.
	Suppose that $C$ is not contained in $E$ and let $p \in C - E$.
	Then since $\calO(D) \otimes \frI_X$ is globally generated, there exists an element $s$ of $H^0(\calO(D))$ which vanishes along $X$ and which is non-zero at $p$.
	Since $\BL^*s$ vanishes along $E$, we get that $\BL^*s$ corresponds to a section $\widetilde{s}$ of $\calO(\BL^*D - E)$ which is non-zero at $p$.
	Then $(\BL^*D - E) \cdot C = \widetilde{s}^{-1}(0) \cdot C \geq 0$.

	Now suppose that $C$ is contained in $E \cong \bP(V)$.
	The restriction of $\calO(\BL^*D - E)$ to $E \cong \bP(V)$
	is $F := q^*(\calO(D|_X)) \otimes \calO_{\bP(V)}(1)$ where $q : \bP(V) \to X$ is the natural projection map.
	By Lemma \ref{lemagloballygenerated} combined with the fact that $V^* \otimes \calO(D|_X)$ is globally generated we get that $F$ is nef and hence its first Chern class pairs non-negatively with $C$.
\end{proof}

\subsection{Bounds for Normal Cones} \label{subsectionboundsfornormalcones}

In this section we give bounds for the $D$-volume of the normal cone of a moduli space in a larger one.
Let $X \subset Y$ be a smooth projective subvariety of a smooth projective variety and $\iota_X : X \hookrightarrow Y$ the natural inclusion map.
Let $\frI_X$ be the ideal sheaf defining $X$ in $Y$.
Let $\pi : V \to X$ be the normal bundle of $X$.
We let $C_{m,d}$ be the normal cone of $X_{m,d}$ inside $Y_{m,(\iota_X)_*d}$ for each $m \in \bN$ and $d \in H_2^+(X)$.
We have the following commutative diagram:
\begin{equation} \label{eqncommdiagram}
	\begin{tikzcd}
		{Y_{m,(\iota_X)_*d}} & {Y_{m+1,(\iota_X)_*d}} & Y \\
		{X_{m,d}} & {X_{m+1,d}} & X
		\arrow["{\pi_{m+1}}"', from=1-2, to=1-1]
		\arrow["{\ev_{m+1}}", from=1-2, to=1-3]
		\arrow[hook, from=2-1, to=1-1]
		\arrow[hook, from=2-2, to=1-2]
		\arrow["{\pi_{m+1}}", from=2-2, to=2-1]
		\arrow["{\ev_{m+1}}"', from=2-2, to=2-3]
		\arrow["{\iota_X}"', hook, from=2-3, to=1-3]
	\end{tikzcd}
\end{equation}
which induces the following diagram of normal cones:
\[\begin{tikzcd}
		{\pi_{m+1}^*C_{m,d}} & {C_{m+1,d}} & V \\
		& {C_{m,d}}
		\arrow[Rightarrow, no head, from=1-1, to=1-2]
		\arrow["{\ev_{m+1}}", from=1-2, to=1-3]
		\arrow["{\pi_{m+1}}"', from=1-2, to=2-2]
	\end{tikzcd}\]
for each $m \in \bN$ and $d \in H_2^+(X)$.
This represents a flat family of stable genus zero curves parameterized by $C_{m,d}$ mapping to $V$
since the corresponding curves projected to $X$ are stable.
So, by the universal property of moduli spaces, we get an induced map
\begin{equation} \label{eqnmapfocones}
	\iota_C : C_{m,d} \to V_{m,d}
\end{equation}
for each $m \in \bN$ and $d \in H_2^+(X)$.

\begin{defn} \label{defnsystemofnormalcones}
	The \emph{system of normal cones}
	\begin{equation}
		C^{X/Y}_\bullet = (C^{X/Y}_{m,d})_{m \in \bN, \ d \in H_2^+(X)}
	\end{equation}
	for $X/Y$ is defined as follows:
	We define $C^{X/Y}_{m,d} \subset V_{m,d}$
	to be the image of $\iota_C$ inside $V_{m,d}$ for each $m \in \bN$ and $d \in H_2^+(X)$.
\end{defn}

\begin{prop} \label{propositionboundsfornormalcones}
	Let $m \in \bN$ and $d \in H_2^+(X)$.
	Let $U \subset X$ be an open subset.
	Suppose $(V_i)_{i \in I}$ is a finite collection of vector bundles over $Y$. Let $C^i_\bullet$ be a system of cones for $V_i$ for each $i \in I$.
	Let $D \subset Y$ be a smooth divisor so that
	\begin{enumerate}
		\item $D-\sum_{i \in I} D'_i$
		      and $(D-\sum_{i \in I} D'_i) \otimes \frI_X$
		      are globally generated
		      where $D_i$ is a divisor in $Y$ so that $\calO(D'_i)$ and $V_i^* \otimes \calO(D'_i)$ are globally generated
		      for each $i \in I$,
		\item $D$ is generically transverse to $C^i_\bullet$ for each $i \in I$,
		\item $D$ is transverse to $X$,
		\item $D|_X := D \cap X$ is generically transverse to $C^{X/Y}_\bullet$ and $C^i_\bullet|_X$ for each $i \in I$.
	\end{enumerate}
	Then for each $m \in \bN$ and $d \in H_2^+(X)$,
	\begin{equation} \label{eqnkeyvoleumnormalinequality}
		D\Vol(C^i_{m,(\iota_X)_*d})_{i \in I} \geq (D|_X)\Vol(\check{C}^i_{m,d})_{i \in  I \sqcup \{X/Y\}}
	\end{equation}
	where
	$\check{C}^i_\bullet = C^i_\bullet|_U$ for each $i \in I$ and $\check{C}^{X/Y}_\bullet = C^{X/Y}_\bullet|_U$.
\end{prop}
\begin{proof}
	We will prove this proposition by a deformation to the normal cone argument, essentially comparing integrals associated with the central fiber and a generic one.
	By Lemma \ref{lemmaineqaulityrestriction}, we can assume that $C^i_\bullet = C^i_\bullet|_{\widetilde{U}}$ for each $i \in I$ where $\widetilde{U} \subset Y$ is an open subset satisfying $\widetilde{U} \cap X = U$.
	Define $\overline{\chi} := \BL_{0 \times X} (\bP^1 \times Y)$ and let $\pi_{\overline{\chi}} : \overline{\chi} \to \bP^1$ be the composition of the blowdown map with the natural projection map $\bP^1 \times Y \to \bP^1$.
	Let $\chi \subset \overline{\chi}$ be the complement of the proper transform of $0 \times Y$
	and $\pi_\chi := \pi_{\overline{\chi}}|_{\chi}$.
	Then $\pi_\chi$ is the deformation to the normal cone of $X$. Its central fiber is the normal bundle $V$ of $X$ in $Y$.
	Let $m \in  \bN$, $d \in H_2^+(X)$, $m' := m + (D|_X) \cdot d$. and $\delta = (\iota_X)_*d$.
	We let
	\begin{equation}
		\Pi_\Delta : \Delta_{m+D!,d} \to \bP^1
	\end{equation}
	be the deformation to the normal cone of $X_{m+(D|_X)!,d}$ inside $Y_{m+D!,\delta}$.
	Similarly, let
	\begin{equation}
		\Pi'_\Delta : \Delta_{m+1+D!,d} \to \bP^1
	\end{equation}
	be the deformation to the normal cone of $X_{m+1+(D|_X)!,d}$ inside $Y_{m+1+D!,\delta}$.
	The diagram
	\[\begin{tikzcd}
			{Y_{m+D!,\delta}} & {Y_{m+1+D!,\delta}} & Y \\
			{X_{m+(D|_X)!,d}} & {X_{m+1+(D|_X)!,d}} & X
			\arrow["{\pi_{m'+1}}"', from=1-2, to=1-1]
			\arrow["{\ev_{m'+1}}", from=1-2, to=1-3]
			\arrow[hook, from=2-1, to=1-1]
			\arrow[hook, from=2-2, to=1-2]
			\arrow["{\pi_{m'+1}}", from=2-2, to=2-1]
			\arrow["{\ev_{m'+1}}"', from=2-2, to=2-3]
			\arrow["{\iota_X}"', hook, from=2-3, to=1-3]
		\end{tikzcd}\]
	induces a diagram:

	\[\begin{tikzcd}
			{\Delta_{0,h+1+D!,d}} & \chi \\
			{\Delta_{0,h+D!,d}}
			\arrow["{\ev_{h+1}}", from=1-1, to=1-2]
			\arrow["{\pi_{h+1}}"', from=1-1, to=2-1]
		\end{tikzcd}\]
	which is a flat family of curves mapping to $\chi$.
	Hence, we have an induced map:
	\begin{equation}
		\iota_\Delta : \Delta_{m+D!,d} \to \chi_{m',\delta'}
	\end{equation}
	where $\delta'$ is the image of $d$ in $\chi$ under pushforward by the natural inclusion map
	\begin{equation}
		X = 0 \times X \hookrightarrow \bP^1 \times X  \hookrightarrow \chi
	\end{equation}
	where $\bP^1 \times X \subset \bP^1 \times Y$ is the proper transform of $\bP^1 \times X$ in $\chi$.


	Let $\delta''$ be the image of $\delta'$ under the natural inclusion map $\chi \hookrightarrow \overline{\chi}$. Let $i : \chi_{m',\delta'} \hookrightarrow \overline{\chi}_{m',\delta''}$ be the natural inclusion map
	and define
	\begin{equation}
		\overline{\iota}_{\Delta} := i \circ \iota_\Delta.
	\end{equation}
	This fits into a commutative diagram:
	\[\begin{tikzcd}
			{\Delta_{m+D!,d}} & {\overline{\chi}_{m',\delta''}} & {\overline{\chi}^{m'}} \\
			& {Y_{m+D!,\delta}} & {Y^{m'}}
			\arrow["{\overline{\iota}_\Delta}", from=1-1, to=1-2]
			\arrow["{F_{Y_{m+D!,\delta}}}"', from=1-1, to=2-2]
			\arrow["\ev", from=1-2, to=1-3]
			\arrow[from=1-2, to=2-2]
			\arrow["{(p_Y \circ \BL)^{m'}}", from=1-3, to=2-3]
			\arrow["\ev", from=2-2, to=2-3]
		\end{tikzcd}\]
	where $\BL : \overline{\chi} \to \bP^1 \times Y$ is the blowdown map and $p_Y : \bP^1 \times Y \to Y$ is the natural projection map and
	$F_{Y_{m+D!,\delta}}$
	is the natural projection map coming from the deformation to the normal cone construction.

	Now let $\widetilde{V}_i$ be the pullback of $V_i$ to $\overline{\chi}$
	via the map $\overline{\chi} \lra{\BL} \bP^1 \times Y \lra{p_Y} Y$ for each $i \in I$.
	We have the following evaluation maps
	\begin{equation}
		\widetilde{\ev}_{\widetilde{V}_i} : F_{Y_{m+D!,\delta}}^*((V_i)_{m,\delta}) \to \ev^*\widetilde{V}_i^{\boxtimes m'}, \quad i \in I
	\end{equation}
	covering $\overline{\iota}_\Delta$
	and hence an evaluation map:
	\begin{equation}
		\widetilde{\ev} := \bigoplus_{i \in I} \widetilde{\ev}_{\widetilde{V}_i}
		: \oplus_{i \in I} F_{Y_{m+D!,\delta}}^*((V_i)_{m,\delta}) \to \oplus_{i \in I} \ev^*\widetilde{V}_i^{\boxtimes m'}.
	\end{equation}

	Let $\widetilde{C}^i_{m+D!,\delta} \subset F_{Y_{m+D!,d}}^*((V_i)_{m,\delta})$
	be the pullback of the subcone $C^i_{m+D!,\delta} \subset (V_i)_{m+D!,\delta}$ for each $i \in I$.
	Define $\overline{C}$ to be the closure inside $\ev^*P_{(\widetilde{V}_i^{\boxtimes m'})_{i \in I}}$ of
	\begin{equation}
		\widetilde{\ev}\left(\fprod{\Delta_{m+D!,\delta}}_{i \in I} \widetilde{C}^i_{m+D!,\delta}\right).
	\end{equation}
	%
	%
	%
	%
	%
	%
	%
	%
	%
	%
	Let $\overline{C}_0$ be the closure of
	\begin{equation}
		\widetilde{\ev}((\check{C}^i_{m+(D|_X)!,d})_{i \in I \sqcup \{X/Y\}})
	\end{equation}
	inside $\ev^*P_{((V_i|_X)^{\boxtimes m'})_{i \in I \sqcup \{X/Y\}}}$ where we define $V_{X/Y}|_X := V$.
	Let
	\begin{equation}
		\nu : P_{(\widetilde{V}_i^{\boxtimes m'})_{i \in I}} \to \overline{\chi}^{m'}, \quad \mu : P_{((V_i|_X)^{\boxtimes m'})_{i \in I \sqcup \{X/Y\}}} \to P_{V^{\boxtimes m'}}
	\end{equation}
	be the natural projection maps.
	The composition
	\begin{equation}
		\overline{C} \to \ev^*P_{(\widetilde{V}_i^{\boxtimes m'})_{i \in I}} \lra{P\ev} P_{(\widetilde{V}_i^{\boxtimes m'})_{i \in I}} \lra{\nu} \overline{\chi}^{m'} \lra{(\pi_{\overline{\chi}})^{m'}} (\bP^1)^{m'},
	\end{equation}
	where the first map is the natural inclusion map,
	factors through the small diagonal $\bP^1 \subset (\bP^1)^{m'}$ because curves in $\overline{\chi}_{m',\delta''}$ become constant curves in $\bP^1$ after composition with $\pi_{\overline{\chi}}$
	and so let
	\begin{equation}
		\widetilde{\Pi} : \overline{C} \to \bP^1
	\end{equation}
	be the corresponding map.
	Let $\widetilde{D}$ be the proper transform of $\bP^1 \times D$ inside $\overline{\chi}$.
	Then
	\begin{equation} \label{eqnwidetildeformula}
		\BL^*D \cong \widetilde{D},
	\end{equation}
	where $\BL : \overline{\chi} \to \bP^1 \times Y$ is the blowdown map
	since $D$ is transverse to $X$.
	Also,
	\begin{equation} \label{eqnblowupmunuseprime}
		(\BL^*D - E')|_{E'} \cong \pi_{P_V}^*(D|_X) + \infty_V
	\end{equation}
	via the identification $E' \cong P_V$ where $E'$ is the exceptional divisor of $\BL$.
	Combining this with \eqref{eqnwidetildeformula}, we get
	\begin{equation} \label{eqnblowupmunuseprime2}
		(\widetilde{D} - E')|_{E'} \cong \pi_{P_V}^*(D|_X) + \infty_V.
	\end{equation}

	Since $V \cong \pi_{\chi}^{-1}(0) \subset \overline{\chi}$ we have a natural subspace $V^{\boxtimes m'} \subset \overline{\chi}^{m'}$.
	We also have a natural subspace $V^{\boxtimes m'} \subset P_{V^{\boxtimes m'}}$.
	We also have natural identifications:
	\begin{equation}
		\nu^{-1}(V^{\boxtimes m'}) \cong \mu^{-1}(V^{\boxtimes m'}) \cong \fprod{X}_{i \in I} P_{(V_i|_X)^{\boxtimes m'}} \times_{X^{m'}} V^{\boxtimes m'}.
	\end{equation}
	Under these identifications, we have
	\begin{equation} \label{eqnintersectionofC0withC1}
		P\ev^{-1}\nu^{-1}(V^{\boxtimes m'}) \cap \overline{C} = P\ev^{-1}\mu^{-1}(V^{\boxtimes m'}) \cap \overline{C}_0.
	\end{equation}
	Also, $P\ev^{-1}\nu^{-1}((E')^{m'}) \cap \overline{C}$ is a codimension $0$ substack of $\widetilde{\Pi}^{-1}(0)$
	because all curves in $\overline{\chi}_{m',\delta''}$
	with image intersecting $V \subset E'$ are contained in $E'$
	and so
	\begin{equation} \label{equationsumofclasses}
		[\widetilde{\Pi}^{-1}(0)] =
		[P\ev^{-1}\nu^{-1}((E')^{m'}) \cap \overline{C}] + C'
	\end{equation}
	inside $CH_*(\ev^*P_{(\widetilde{V}_i^{\boxtimes m'})_{i \in I}})$ where $C'$ is a positive sum of integral closed substacks of $\ev^*P_{(\widetilde{V}_i^{\boxtimes m'})_{i \in I}}$.
	We have a natural identification:
	\begin{equation} \label{eqncopiesofv}
		P_{(p_i^*V)_{i \in \{1,\cdots,m'\}},(V_i|_X)_{i \in I}} \cong \nu^{-1}((E')^{m'})
	\end{equation}
	where $p_i : X^{m'} \to X$ is the $i$th projection map for each $i=1,\cdots,m'$.
	By Equation \eqref{eqnblowupmunuseprime2},
	we get
	\begin{equation} \label{eqnproductrestrictioneqn}
		(D|_X)_{\sum m', (p_i^*V)_{i=1}^{m'},(V_i|_X)_{i \in I}} \cong (\widetilde{D}_{\sum m', (\widetilde{V}_i)_{i \in I}}-\pi_{P_{(\widetilde{V}_i^{\boxtimes m'})_{i \in I}}}^*(E')_{\sum m'})|_{\nu^{-1}((E')^{m'})}
	\end{equation}
	under the identification \eqref{eqncopiesofv}.
	We have
	\begin{equation}
		(D|_X)\Vol(\check{C}^i_{m,\delta})_{i \in  I \sqcup \{X/Y\}}
		= \frac{1}{m!(D|_X \cdot d)!}\int_{P\ev_*[\overline{C}_0]} \exp((D|_X)_{\sum m', (V_i|_X)_{i \in I \sqcup \{X/Y\}}})
	\end{equation}
	\begin{equation}
		=\frac{1}{m!(D \cdot \delta)!} \int_{P\ev_*\overline{P\ev^{-1}\mu^{-1}(V^{\boxtimes m'}) \cap \overline{C}_0}} \exp((D|_X)_{\sum m', (V_i|_X)_{i \in I \sqcup \{X/Y\}}})
	\end{equation}
	\begin{equation} \label{eqnlastpartae}
		\overset{\textnormal{\eqref{eqnintersectionofC0withC1}}}{=}
		\frac{1}{m!(D \cdot \delta)!}\int_{P\ev_*\overline{P\ev^{-1}\nu^{-1}(V^{\boxtimes m'}) \cap \overline{C}}} \exp((D|_X)_{\sum m', (V_i|_X)_{i \in I\sqcup \{X/Y\}}})
	\end{equation}
	where $\overline{P\ev^{-1}\mu^{-1}(V^{\boxtimes m'}) \cap \overline{C}_0}$ and $\overline{P\ev^{-1}\nu^{-1}(V^{\boxtimes m'}) \cap \overline{C}}$ is the closure of $P\ev^{-1}\mu^{-1}(V^{\boxtimes m'}) \cap \overline{C}_0$ and $P\ev^{-1}\nu^{-1}(V^{\boxtimes m'}) \cap \overline{C}$ respectively inside $\ev^*P_{((V_i|_X)^{\boxtimes m'})_{i \in I \sqcup \{X/Y\}}}$.

	By
	Lemma \ref{lemmmasmoothingsums}
	applied to the intersection of
	$\overline{P\ev^{-1}\nu^{-1}(V^{\boxtimes m'}) \cap \overline{C}}$
	with
	\begin{equation}
		\ev^*\oplus_{i=1}^{m'} p_i^* V \bigoplus \oplus_{i \in I} (V_i|_X)^{\boxtimes m'}
	\end{equation}
	compactifying to
	$\ev^*P_{((V_i|_X)^{\boxtimes m'})_{i \in I \sqcup \{X/Y\}}}$
	and the pullback via $\ev$ of
	\eqref{eqncopiesofv}  respectively
	we get
	\eqref{eqnlastpartae} is less than or equal to
	\begin{equation} \label{eqnaefaeeytaaioegjaio0}
		\frac{1}{m!(D \cdot \delta)!} \int_{P\ev_*[P\ev^{-1}\nu^{-1}({E'}^{m'}) \cap \overline{C}]} \exp((D|_X)_{\sum m', (p_i^*V)_{i \in \{1,\cdots,m'\}},(V_i|_X)_{i \in I}})
	\end{equation}
	\begin{equation} \label{eqnaefaeeytaaioegjaio}
		\overset{\textnormal{\eqref{eqnproductrestrictioneqn}}}{=} \frac{1}{m!(D \cdot \delta)!}\int_{P\ev_*[P\ev^{-1}\nu^{-1}({E'}^{m'}) \cap \overline{C}]} \exp(\widetilde{D}_{\sum m', (\widetilde{V}_i)_{i \in I}}-\pi_{P_{(\widetilde{V}_i)_{i \in I}}}^*(E')_{\sum m'}).
	\end{equation}

	Let $\widetilde{D'_i}$ be
	the pullback of $D'_i$ via the projection map $\overline{\chi} \to Y$ for each $i \in I$.
	Let \begin{equation}
		P : P_{(\widetilde{V}_i^{\boxtimes m'})_{i \in I}} \to P_{(V_i^{\boxtimes m'})_{i \in I}}
	\end{equation}
	be the natural projection map.
	We have by Corollary \ref{corollarynefsum2} that $(\sum_{i \in I} D'_i)_{\sum m',(V_i)_{i \in I}}$ is nef.
	Also, by Lemmas \ref{lemmacompaet} and \ref{lemmanicesufficientcondtionforblxycompatible} and Equation \eqref{eqnwidetildeformula}, we have that $\widetilde{D} - \sum_{i \in I}\widetilde{D}'_i - E'$ is nef for each $i \in I$.
	Therefore,
	\begin{equation} \label{eqnthingisnef}
		\widetilde{D}_{\sum m', (\widetilde{V}_i)_{i \in I}}-\pi_{P_{(\widetilde{V}_i)_{i \in I}}}^*(E')_{\sum m'}
		= P^*((\sum_{i \in I}D_i)_{\sum m',(V_i)_{i \in I}}) + \pi_{P_{(\widetilde{V}_i)_{i \in I}}}^*((\widetilde{D}-\sum_{i \in I} \widetilde{D}'_i-E')_{\sum m'})
	\end{equation}
	is nef.
	Combining this fact with Equation \eqref{equationsumofclasses} gives us that \eqref{eqnaefaeeytaaioegjaio}
	is less than or equal to:
	\begin{equation}
		\frac{1}{m!(D \cdot \delta)!}\int_{P\ev_*[\widetilde{\Pi}^{-1}(0)]} \exp(\widetilde{D}_{\sum m', (\widetilde{V}_i)_{i \in I}}-\pi_{P_{(\widetilde{V}_i)_{i \in I}}}^*(E')_{\sum m'})
	\end{equation}
	\begin{equation}
		= \frac{1}{m!(D \cdot \delta)!}\int_{P\ev_*[\widetilde{\Pi}^{-1}(1)]} \exp(\widetilde{D}_{\sum m', (\widetilde{V}_i)_{i \in I}}-\pi_{P_{(\widetilde{V}_i)_{i \in I}}}^*(E')_{\sum m'})
	\end{equation}
	\begin{equation}
		= \frac{1}{m!(D \cdot \delta)!}\int_{P\ev_*[\widetilde{\Pi}^{-1}(1)]} \exp(\widetilde{D}_{\sum m', (\widetilde{V}_i)_{i \in I}})
		= D\Vol(C^i_{m,d})_{i \in I}.
	\end{equation}

\end{proof}

\section{Connection with Virtual Fundamental Class} \label{vfcconnection}

\subsection{Siebert's description of the Virtual Fundamental Class} \label{sectioncomprisinvfc}

In this section we give Siebert's description of the virtual fundamental class of the moduli space of stable genus zero curves in terms of the normal cone from \cite{siebert2004virtual}.
Let $\iota_X : X \hookrightarrow \bP^N$ be a closed immersion from a smooth projective variety $X$ and $V \to X$ the normal bundle of this closed immersion.
We define $\frM_{m}$ to be the moduli stack of genus zero pre-stable curves with $m$ marked points with the fppf topology.


We have a perfect obstruction theory
\begin{equation} \label{eqnobstructiontheory}
	((R\pi_{m+1})_*\ev_{m+1}^*(T_X))^\vee \to L^\bullet_{X_{m,d}/\frM_{m}}
\end{equation}
defining the virtual fundamental class on $X_{m,d}$
where $L^\bullet_{X_{m,d}/\frM_{m}}$ is the cotangent complex
(See \cite[Proposition 6.2]{BehrendFantechinormalcone} and \cite{BehrendGW}).
We can truncate these complexes giving a map:
\begin{equation} \label{eqnobstructiontheorytr}
	((R\pi_{m+1})_*\ev_{m+1}^*(T_X))^\vee \to \tau_{\geq -1}L^\bullet_{X_{m,d}/\frM_{m}}.
\end{equation}
By \cite[Chapitre III, Corollaire 3.1.3]{illusie2006complexe},
this corresponds to a commutative diagram:
\begin{equation} \label{eqndiagramnonsese2}
	\begin{tikzcd}
		{((\pi_{m+1})_*\ev_{m+1}^* {\mathcal N^*}^\vee_{X/\bP^N})^\vee} & {{\mathcal N^*}_{X_{m,d}/\bP^N_{m,\delta}} } \\
		{((\pi_{m+1})_*\ev_{m+1}^*(\Omega^\vee_{\bP^N}|_X))^\vee} & {\Omega_{\bP^N_{m,\delta}/\frM_{m}}|_{X_{m,d}}}
		\arrow[from=1-1, to=1-2]
		\arrow[from=1-1, to=2-1]
		\arrow[from=1-2, to=2-2]
		\arrow["\cong", from=2-1, to=2-2]
	\end{tikzcd}
\end{equation}
where ${\mathcal N}^*_{A/B}$ is the conormal sheaf of $A$ in $B$ for any closed substack $A \subset B$.
The bottom map in this diagram is an isomorphism since $T\bP^N$ is convex.
The top map in \eqref{eqndiagramnonsese2}  is surjective since it comes from a perfect obstruction theory.
The following proposition allows us to compute the virtual fundamental classes from the data above.
The top map in Equation \eqref{eqndiagramnonsese2} gives us a morphism of cones
\begin{equation} \label{eqnmorphismofcones1}
	C_{m,d} \to V_{m,d}
\end{equation}
where $C_{m,d}$ is the normal cone of $X_{m,d}$ inside $\bP^N_{m,\delta}$.
(See \cite[Section 3]{siebert2004virtual}).

\begin{prop} \label{propcomparisonwithvfc} (see \cite{siebert2004virtual}).
	Let
	\begin{equation}
		[]^{s(C_\bullet) \cap c(V_\bullet)} := ([X_{m,d}]^{s(C_\bullet) \cap c(V_\bullet)})_{m \in \bN, \ d \in H_2^+(X)}
	\end{equation}
	where \begin{equation}
		[X_{m,d}]^{s(C_\bullet) \cap c(V_\bullet)} := s(C_{m,d}) \cap c(V_{m,d}), \quad \forall \ m \in \bN, \ d \in H_2^+(X).
	\end{equation}
	Then
	\begin{equation}
		[X_{m,d}]^{\virt} = \{ [X_{m,d}]^{s(C_\bullet) \cap c(V_\bullet)} \}_{\dim=c_1(TX)(d) + n-3 + m}, \ \forall \ m \in \bN, \ d \in H_2^+(X)
	\end{equation}
	where $c$ is the total Chern class and $\{-\}_{\dim=h}$ means that we only consider the dimension $h$ component of such a class for any $h \in \bN$.
\end{prop}


\subsection{Embeddings of Normal Cones} \label{sectionnormalconeembedding}

In Section \ref{subsectionboundsfornormalcones}, we gave bounds for the $D$-volume of the image of the normal cone of a moduli space of curves in another one under the evaluation map. However, we really wish to use the normal cone since Siebert's formula describing the virtual fundamental class from Section \ref{sectioncomprisinvfc} uses the normal cone rather than its image.
Therefore, we need to show such cones are isomorphic under certain conditions.
This is the purpose of this section.

Let $Y$ be a smooth (not necessarily projective) variety and $X \subset Y$ a smooth closed subvariety where $\iota_X : X \hookrightarrow Y$ is the natural inclusion map.
Suppose that $TY$ is convex.
Let $\pi : V \to X$ be the normal bundle of $X$ in $Y$.
%
%
%


Now let us fix $m \in \bN$ and $d \in H_2^+(X)$.
We wish to show $\iota_C$ from Equation \eqref{eqnmapfocones} is an immersion.
Before we do this we need a preliminary lemma and some notation.
Let $\delta = (\iota_X)_*d$.
We let $\scrI_X$ be the ideal sheaf defining $X$ in $Y$.
We let $\iota_{m,d} : X_{m,d} \hookrightarrow Y_{m,\delta}$
and $\iota_{m+1,d} : X_{m+1,d} \hookrightarrow Y_{m+1,\delta}$
be the natural inclusion maps.
Let $A := \iota_{m+1,d}^*\ev_{m+1}^*\scrI_X \in \Coh(X_{m+1,d})$.

\begin{lemma} \label{lemmaunversalrpero}
	Consider the two functors:
	\begin{equation}
		F_1 : \Coh(X_{m,d}) \to \CVect, \quad F_1(B) = \Hom(A,\pi_{m+1}^*B)
	\end{equation}
	\begin{equation} \label{eqnaefaef}
		\begin{aligned}
			F_2 : \Coh(X_{m,d}) \to \CVect, \quad F_2(B) = & \Hom(((\pi_{m+1})_*A^\vee)^\vee,B)                        \\
			=                                              & \Hom((\pi_{m+1}^*(\pi_{m+1})_*A^\vee)^\vee,\pi_{m+1}^*B).
		\end{aligned}
	\end{equation}
	where $\CVect$ is the category of $\bC$ vector spaces.
	Then composing with the natural map \begin{equation}
		A \to A^{\vee \vee} \to (\pi_{m+1}^*(\pi_{m+1})_*(A^\vee))^\vee
	\end{equation} induces a natural isomorphism $F_2 \lra{\cong} F_1$.
\end{lemma}

Note that the second equality in \eqref{eqnaefaef} holds since $(\pi_{m+1})_*A^\vee$ is locally free by Lemma \ref{lemmaconvexquotient}
combined with the fact that the map $B \to (\pi_{m+1})_*\pi_{m+1}^*B$ is an isomorphism by \cite[0E0L, 07US]{stacks-project}.



\begin{proof}
	Instead of working with the category of coherent sheaves, we will work with the category of Abelian cones (See \cite[Section 1]{BehrendFantechinormalcone}, \cite[Section 1.7]{EGA2}).
	In other words, for any stack $Z$, the category $\AbC(Z)$ whose objects are stacks over $Z$ of the form $\Spec(\Sym_{\calO_Z} \scrF)$
	where $\Sym_{\calO_Z} \scrF$ is the free symmetric algebra generated by a coherent sheaf $\scrF$ and where morphisms between such stacks are the ones induced by morphism between sheaves (or, equivalently, morphisms respecting the linear structure on these Abelian cones).

	Consider the functors:
	\begin{equation}
		F'_1 : \AbC(X_{m,d}) \to \CVect, \quad F'_1(B') = \Hom_{\AbC(X_{m+1,d})}(\pi_{m+1}^*B',\ev_{m+1}^*V),
	\end{equation}
	\begin{equation}
		\begin{aligned}
			F'_2 : \AbC(X_{m,d}) \to \CVect, \quad F'_2(B') =
			  & \Hom_{\AbC(X_{m,d})}(B',V_{m,d})                 \\
			= & \Hom_{\AbC(X_{m+1,d})}(\pi_{m+1}^*B',V_{m+1,d}).
		\end{aligned}
	\end{equation}
	Then composing each element of $F'_2(B')$ with the evaluation map \begin{equation}\label{eqnevsheaf}
		\ev_{m+1} : V_{m+1,d} \to \ev_{m+1}^*V
	\end{equation} for each $B' \in \AbC(X_{m,d})$
	induces a natural transformation $\check{\eta} : F'_2 \lra{} F'_1$.

	Since the map $\Coh(Y) \to \AbC(Y)$ sending $\scrF$ to $\Spec(\Sym_{\calO_Y} \scrF)$ is an anti-equivalence for any stack $Y$ (\cite[Section 1.7]{EGA2}), it is sufficient to prove that $\check{\eta}$ is a natural isomorphism.
	%
	By definition, the following universal property is satisfied:
	\[\begin{tikzcd}
			{\pi_{m+1}^*B'} & {\ev_{m+1}^*V} & {V_{m+1,d}} \\
			{B'} && {V_{m,d}}
			\arrow["\forall"{description}, from=1-1, to=1-2]
			\arrow["{\exists !}"', curve={height=18pt}, dashed, from=1-1, to=1-3]
			\arrow["{\pi_{m+1}}"', from=1-1, to=2-1]
			\arrow["{\ev_{m+1}}"', from=1-3, to=1-2]
			\arrow["{\pi_{m+1}}", from=1-3, to=2-3]
			\arrow["{\exists !}"', dashed, from=2-1, to=2-3]
		\end{tikzcd}\]
	By examining this diagram we see that $\check{\eta}$ is a natural isomorphism.
	The uniqueness property ensures that $\check{\eta}(B')$ is injective, and the existence property ensures that it is surjective.
\end{proof}

\begin{lemma} \label{lemmaembedding}
	The map $\iota_C$ from Equation \eqref{eqnmapfocones} is a closed immersion.
\end{lemma}
\begin{proof}
	Let $F_1$, $F_2$ be as in Lemma \ref{lemmaunversalrpero} above
	and let $\nu : F_2 \lra{\cong} F_1$ be the natural isomorphism from the same Lemma.
	Let $m \in \bN$, $d \in H_2^+(X)$ and $\delta = (\iota_X)_*d$.
	Also, let $\scrI_{X_{m,d}}$ be the ideal sheaf in $\calO_{Y_{m,\delta}}$
	defining $X_{m,d}$.

	We let $N_{\check{m},d}$ be the normal bundle of $X_{\check{m},d}$ inside $Y_{\check{m},\delta}$ for each $\check{m} \in \bN$.
	The commutative diagram \eqref{eqncommdiagram} gives us the following diagram of normal bundles:
	\[\begin{tikzcd}
			{\pi_{m+1}^*N_{m,d}} & {N_{m+1,d}} & V \\
			& {N_{m,d}}
			\arrow[Rightarrow, no head, from=1-1, to=1-2]
			\arrow["{\ev_{m+1}}", from=1-2, to=1-3]
			\arrow["{\pi_{m+1}}"', from=1-2, to=2-2]
		\end{tikzcd}\]
	Since this represents a flat family of curves mapping to $V$
	we get an induced map
	\begin{equation} \label{eqnnormaltov0hd}
		\iota_N : N_{m,d} \to V_{m,d}.
	\end{equation}
	The map of cones $\iota_C$ factors as:
	\begin{equation} \label{eqnfactorization}
		C_{m,d} \hookrightarrow N_{m,d} \lra{\iota_N} V_{m,d}.
	\end{equation}
	where  $C_{m,d} \hookrightarrow N_{m,d}$ is the natural closed immersion from the normal cone to the normal bundle.
	Therefore, it is sufficient to prove that $\iota_N$ is a closed immersion.
	Equivalently, it is sufficient to show that the corresponding map of coherent sheaves:
	\begin{equation} \label{eqnmapidealshweaf}
		((\pi_{m+1})_*A^\vee)^\vee \to \iota_{m,d}^*\scrI_{X_{m,d}}
	\end{equation}
	is surjective.
	Now the map \eqref{eqnmapidealshweaf} is equal to
	$\nu^{-1}(\ev_{m+1}^*)$ where
	\begin{equation} \label{eqnevaev}
		\ev_{m+1}^* : A \to \iota_{m+1,d}^*\scrI_{X_{m+1,d}} = \pi_{m+1}^*\iota_{m,d}^*\scrI_{X_{m,d}}
	\end{equation}
	is induced by the evaluation map.
	Here $\ev_{m+1}^* \in F_2(B)$ where $B = \iota_{m,d}^*\scrI_{X_{m,d}}$.

	We let $\omega = \omega_{X_{m+1,d}/X_{m,d}}$ be the relative dualizing sheaf.
	Consider the sequence of natural isomorphisms in the derived category $D^b(\mod{\calO_{X_{m,d}}})$:
	\begin{equation} \label{eqnsequenceofnt}
		\begin{aligned}
			H^0(\Hom(A,\pi_{m+1}^*(-))) \lra{\cong} & H^0(\RHom(A,\pi_{m+1}^*(-)))
			\\  \lra{\cong} & H^0(\RHom(A \otimes \omega, \pi_{m+1}^*(-) \otimes \omega))                        \\
			\lra{\cong}                             & H^0(\RHom(A \otimes \omega, \pi_{m+1}^!(-)[-1]))
			\\
			\lra{\cong}                             & H^0 (\RHom((R\pi_{m+1})_*(A \otimes \omega [1]), (-)))                              \\
			\lra{\cong}                             & H^0 (\RHom((R\pi_{m+1})_*((A^\vee)^\vee \otimes \omega [1]), (-)))                  \\
			\lra{\cong}                             & H^0 (\RHom((R\pi_{m+1})_*((A^\vee)^\vee \otimes \pi^!_{m+1} \calO_{X_{m,d}}), (-))) \\
			\lra{\cong}                             & H^0 (\RHom((R\pi_{m+1})_*(\RHom(A^\vee, \pi^!_{m+1} \calO_{X_{m,d}})), (-)))        \\
			\lra{\cong}                             & H^0 (\RHom((R\pi_{m+1})_*(\RHom((R\pi_{m+1})_* A^\vee, \calO_{X_{m,d}}))), (-))     \\
			\lra{\cong}                             & H^0(\RHom((R\pi_{m+1})_*(A^\vee)^\vee, (-)))                                        \\
			\lra{\cong}                             & H^0(\Hom((\pi_{m+1})_*(A^\vee)^\vee, (-))).
		\end{aligned}
	\end{equation}
	We have \begin{equation}
		H^0(\Hom(A,\pi_{m+1}^*(B))) = F_1(B)
	\end{equation} for each coherent sheaf $B$. Similarly,
	we get \begin{equation}
		H^0(\Hom((\pi_{m+1})_*((A^\vee))^\vee, B)) = F_2(B)
	\end{equation} for each coherent sheaf $B$.
	Therefore, \eqref{eqnsequenceofnt} induces a natural isomorphism $\nu'$ between $F_1$ and $F_2$.
	Hence, $\nu' \circ \nu$ induces an automorphism $\Phi$ of $(\pi_{m+1})_*((A^\vee))^\vee$ in $\Coh(X_{m,d})$ by the Yoneda lemma.

	The map $\nu'(\ev_{m+1}^*)$ is equal to the top map in \eqref{eqndiagramnonsese2}
	by the discussion before Proposition 6.2 in \cite{BehrendFantechinormalcone}.
	Also, the top map in  \eqref{eqndiagramnonsese2} is surjective since the diagram  \eqref{eqndiagramnonsese2}  comes from a perfect obstruction theory.
	Hence, $\nu'(\ev_{m+1}^*)$ is surjective.
	Since $\Phi^{-1} \circ \nu'(\ev_{m+1}^*) = \nu^{-1}(\ev_{m+1}^*)$ and since $\Phi$ comes from an automorphism of $(\pi_{m+1})_*((A^\vee))^\vee$,
	we get that $\nu^{-1}((\ev_{m+1})^*)$ is also surjective.
	Therefore, since \eqref{eqnmapidealshweaf} is equal to $\nu^{-1}((\ev_{m+1})^*)$ we get that $\iota_N$ is a closed immersion.
	Hence, $\iota_C$ is a closed immersion.
\end{proof}

\begin{corollary} \label{corollarynormalconenonsence}
	Let $\iota_X : X \hookrightarrow Y$ be a closed immersion from a projective variety $X$ into a smooth variety $Y$ with $TY$ convex
	(Definition \ref{defnconvex}).
	Let
	\begin{equation}
		C^{X/Y}_\bullet = (C^{X/Y}_{m,d})_{m \in \bN, \ d \in H_2^+(X)}
	\end{equation}
	be a system of normal cones for $X/Y$ as in Definition \ref{defnsystemofnormalcones}.
	Suppose $m \in \bN$ and $d \in H_2^+(X)$.
	Also, let $C_{m,d}$ be the normal cone of $X_{m,d}$ inside $Y_{m,(\iota_X)_*d}$.
	Then $C_{m,d} \cong C^{X/Y}_{m,d}$.
\end{corollary}

\subsection{Virtual Fundamental Class in Terms of Segre Classes} \label{subsectionvfcfromcones}

In this Section we wish to compute the virtual fundamental class as well as give estimates of Gromov-Witten counts using Segre classes only and hence use the $D$-volume to bound such Gromov-Witten counts.
The other purpose of this section is to use an auxillary bundle $T$, which in the main theorem in Section \ref{sectionmainargument} is easier to bound than the normal bundle $V$ of a smooth projective variety in projective space.

\begin{lemma} \label{lemmasegreclasscomputation}
	Let $\pi : B' \to B''$ be a morphism of stacks and let $W''$ be a vector bundle over $B''$. Define $W' := \pi^*W''$.
	Suppose that the of dimension $B''$ is at most $\delta$.
	Let $C$ be any cone over $B'$.
	Then:
	\begin{equation}
		s(C) \cap c(W')
		= \sum_{\ell = 0}^\delta \sum_{k=0}^{\delta - \ell} (-1)^\ell {k+\ell \choose \ell} s(C \oplus (W')^\ell).
	\end{equation}
\end{lemma}
\begin{proof}
	Let $c^{-1}(W')$ be the inverse of the total Chern class $c(W')$ of $W'$.
	We have
	\begin{equation}
		c(W') = 1/(1 + (c^{-1}(W')-1)) = \sum_{k = 0}^\delta (-1)^k (c^{-1}(W')-1)^k
	\end{equation}
	\begin{equation}
		= \sum_{k = 0}^\delta \sum_{\ell=0}^k (-1)^{k+(k-\ell)} {k \choose \ell} c^{-1}(W')^\ell
		= \sum_{\ell=0}^\delta \sum_{k=\ell}^\delta (-1)^\ell {k \choose \ell} c^{-1}(W')^\ell
	\end{equation}
	\begin{equation}
		= \sum_{\ell = 0}^\delta \sum_{k=0}^{\delta - \ell} (-1)^\ell {k+\ell \choose \ell} c^{-1}(W')^\ell
		= \sum_{\ell = 0}^\delta \sum_{k=0}^{\delta - \ell} (-1)^\ell {k+\ell \choose \ell} c^{-1}((W')^\ell).
	\end{equation}
	Hence,
	\begin{equation} \label{eqncandwequantion}
		s(C) \cap c(W') = \sum_{\ell = 0}^\delta \sum_{k=0}^{\delta - \ell} (-1)^\ell {k+\ell \choose \ell} s(C) \cap c^{-1}((W')^\ell)
	\end{equation}
	\begin{equation}
		= \sum_{\ell = 0}^\delta \sum_{k=0}^{\delta - \ell} (-1)^\ell {k+\ell \choose \ell} s(C \oplus (W')^\ell).
	\end{equation}
\end{proof}

Let $\iota_X : X \hookrightarrow \bP^N$ be the inclusion map of a smooth projective variety $X$ into projective space
and let $V$ be its normal bundle.
Also, let \begin{equation}
	C^{X/\bP^N}_\bullet = (C^{X/\bP^N}_{m,d})_{m \in \bN, \ d \in H_2^+(X)}
\end{equation}
be a system of normal cones for $X/\bP^N$ as in Definition
\ref{defnsystemofnormalcones}.
Now let us suppose we have a short exact sequence of strongly convex vector bundles
\begin{equation} \label{eqnexactsequencesforV}
	0 \to V \to T \to W \to 0
\end{equation}
over $X$.

\begin{lemma} \label{lemmaesfsefesf}
	Let $m \in \bN$, $d \in H_2^+(X)$ and let $a_1,\cdots,a_m \in H^*(X)$ be cohomology classes of pure degree.
	Let \begin{equation} \label{eqnforhd}
		e_d \geq (N+1)\deg(d) + N-3
	\end{equation}
	and suppose that $\sum_{j=1}^m \deg(a_j) = 2\dim([X_{m,d}]^{\virt})$.
	Then
	\begin{equation} \label{eqnsegregw}
		\begin{aligned}
			 & \langle a_1,\cdots,a_m \rangle^X_{m,d} = \\ &
			\sum_{\ell = 0}^{e_d} \sum_{k=0}^{e_d-\ell} (-1)^\ell {k+\ell \choose \ell} \langle a_1,\cdots,a_m \rangle_{m,d}^{s(C^{X/\bP^N}_\bullet \oplus T^\ell_\bullet \oplus W_\bullet)}
		\end{aligned}
	\end{equation}
\end{lemma}
\begin{proof}
	Since $V,T, W$ are convex, we have a short exact sequence:
	\begin{equation}
		0 \to V_{m,d} \to T_{m,d} \to W_{m,d} \to 0
	\end{equation}
	and so
	\begin{equation}
		s(C^{X/\bP^N}) \cap c(V_{m,d}) = s(C^{X/\bP^N}_{m,d} \oplus W_{m,d}) \cap c(T_{m,d}).
	\end{equation}
	Therefore, by Lemma \ref{lemmasegreclasscomputation} with
	\begin{itemize}
		\item $B' = X_{m,d}$, $B'' = X_{0,d}$,
		\item $\pi$ the map forgetting all marked points,
		\item $C = C^{X/\bP^N}_{m,d} \oplus W_{m,d}$,
		\item $W' = T_{m,d}$ and
		\item $\delta = e_d$ (Lemma \ref{lemmadimensionofX})
	\end{itemize}
	combined with Proposition \ref{propcomparisonwithvfc} and Corollary \ref{corollarynormalconenonsence}
	we get that \eqref{eqnsegregw} holds.

\end{proof}

\begin{lemma} \label{lemmavolumeboundsimplyGWbounds}
	Let $D, D'$ be divisors on $X$ so that
	\begin{enumerate}
		\item $D$ is linearly equivalent to $4D'$ where $D'$ ample,
		\item $\calO(D')$, $V^* \otimes \calO(D')$, $W^* \otimes \calO(D')$, $T^* \otimes \calO(D')$ are all globally generated and
		\item $D$ is smooth and generically transverse to the system of normal cones $C^{X/\bP^N}_\bullet$ (Definitions \ref{defngenericallytransverse} and \ref{defnsystemofnormalcones})
		      and is $V$, $T$ and $W$-large (Definition \ref{defncompactieab}).
	\end{enumerate}
	Let $\omega$ be a K\"{a}hler form Poincar\'{e} dual to $D'$.
	Also, let $m \in \bN$, $d \in H_2^+(X)$ and let $\alpha_1,\cdots,\alpha_m$ be closed homogenous de Rham forms on $X$.
	Then
	\begin{equation} \label{eqndvolinequalityforbundles}
		\begin{aligned}
			 & \left| \langle [\alpha_1],\cdots,[\alpha_m] \rangle^X_{m,d} \right| \leq \\ &
			\sum_{\ell = 0}^{r_d} \sum_{k=0}^{r_d-\ell}  {k+\ell \choose \ell} (23r_d + 2 m N + 18m +\rho_{d,T}+23)! (\ell\rho_{d,T})!
			D\Vol(C^{X/\bP^N}_{m,d}, T_{m,d}^\ell, W_{m,d})
			\prod_{j=1}^m |\alpha_j|_{C^0,\omega}
		\end{aligned}
	\end{equation}
	if $X_{m,d}$ is non-empty
	where
	\begin{equation} \label{eqnaefaefforrd}
		r_d :=  (N+1)\deg(d) + N + c_1(V)(d) + c_1(W)(d) + \dim(V) + \dim(W)
	\end{equation}
	\begin{equation} \label{eqnrhodT}
		\rho_{d,T} := c_1(T)(d) + \dim(T).
	\end{equation}
\end{lemma}

Note that $r_d$ and $\rho_{d,T}$ are non-negative numbers since
\begin{equation} \label{eqnnonnegativity}
	c_1(V)(d) \geq 0, \quad c_1(W)(d) \geq 0, \quad c_1(T)(d) \geq 0
\end{equation}
due to the fact that $V,W,T$ are strongly convex.

\begin{proof}[Proof of Lemma \ref{lemmavolumeboundsimplyGWbounds}.]
	We can assume that $\langle [\alpha_1],\cdots,[\alpha_m] \rangle^X_{m,d}$ is non-zero. Hence,
	$\sum_{j=1}^m \deg(\alpha_j) = \dim{[X_{m,d}]^{\virt}}$.
	Let $e_d$ be as in Equation \eqref{eqnedeqn}.
	By Lemma \ref{lemmaesfsefesf}, we have
	\begin{equation} \label{eqnaefaefeaf}
		\left| \langle [\alpha_1],\cdots,[\alpha_m] \rangle^X_{m,d} \right|
	\end{equation}
	\begin{equation} \label{eqnrighofequanfromlemma77}
		\leq  \sum_{\ell = 0}^{e_d} \sum_{k=0}^{e_d-\ell} {k+\ell \choose \ell} \left| \langle [\alpha_1],\cdots,[\alpha_m] \rangle_{m,d}^{s(C^{X/\bP^N}_\bullet \oplus T^\ell_\bullet \oplus W_\bullet)} \right|.
	\end{equation}
	For each vector bundle $W'$ define
	\begin{equation} \label{eqndefonrorho}
		\rho_{d,W'} := c_1(W')(d) + \dim(W').
	\end{equation}
	Since $D$ is $V$, $T$ and $W$-large, it is also $V \oplus T^\ell \oplus W$-large for each $\ell \in \bN$.
	Hence, by Proposition \ref{propboundsforsegreclass}, we have that \eqref{eqnrighofequanfromlemma77} is bounded above by:
	\begin{equation} \label{eqnconsequenceofprop54}
		\begin{aligned}
			\sum_{\ell = 0}^{e_d} \sum_{k=0}^{e_d-\ell}  {k+\ell \choose \ell}
			(9e_d + 2 m N + 6m +9)! (\rho_{d,V \oplus T^\ell \oplus W}+1)^{4(e_d+m+1)}(\rho_{d,V \oplus T^\ell \oplus W})!
			\\
			\times D\Vol(C^{X/\bP^N}_{m,d} \oplus T_{m,d}^\ell \oplus W_{m,d})
			\prod_{j=1}^m |\alpha_j|_{C^0,\omega}.
		\end{aligned}
	\end{equation}
	By Proposition \ref{propmallestcompactificationinequality}, this is bounded above by:
	\begin{equation}
		\begin{aligned}
			\sum_{\ell = 0}^{e_d} \sum_{k=0}^{e_d-\ell}  {k+\ell \choose \ell} (9e_d + 2 m N + 6m +9)! (\rho_{d,V \oplus T^\ell \oplus W}+1)^{4(e_d+m+1)}(\rho_{d,V \oplus T^\ell \oplus W})!
			\\
			\times D\Vol(C^{X/\bP^N}_{m,d}, T_{m,d}^\ell, W_{m,d})
			\prod_{j=1}^m |\alpha_j|_{C^0,\omega}
		\end{aligned}
	\end{equation}
	\begin{equation}
		\begin{aligned}
			\leq \sum_{\ell = 0}^{e_d} \sum_{k=0}^{e_d-\ell}  {k+\ell \choose \ell} (9e_d + 2 m N + 6m +9)! (\ell\rho_{d,T} + \rho_{d,V\oplus W}+1)^{4(e_d+m+1)}(\ell\rho_{d,T} + \rho_{d,V\oplus W})!
			\\
			\times D\Vol(C^{X/\bP^N}_{m,d}, T_{m,d}^\ell, W_{m,d})
			\prod_{j=1}^m |\alpha_j|_{C^0,\omega}
		\end{aligned}
	\end{equation}
	\begin{equation}
		\begin{aligned}
			= \sum_{\ell = 0}^{e_d} \sum_{k=0}^{e_d-\ell}  {k+\ell \choose \ell} (9e_d + 2 m N + 6m +9)! (\ell\rho_{d,T} + \rho_{d,V\oplus W}+1)^{4(e_d+m+1)}
			\\ \times
			{\ell\rho_{d,T} + \rho_{d,V\oplus W} \choose \rho_{d,V \oplus W}}(\rho_{d,V \oplus W})!(\ell\rho_{d,T})!
			D\Vol(C^{X/\bP^N}_{m,d}, T_{m,d}^\ell, W_{m,d})
			\prod_{j=1}^m |\alpha_j|_{C^0,\omega}
		\end{aligned}
	\end{equation}
	\begin{equation}
		\begin{aligned}
			\leq \sum_{\ell = 0}^{e_d} \sum_{k=0}^{e_d-\ell}  {k+\ell \choose \ell} (9e_d + 2 m N + 6m +9)! (\ell\rho_{d,T} + \rho_{d,V\oplus W}+1)^{4(e_d+m+1)}
			\\ \times
			(\ell\rho_{d,T} + \rho_{d,V\oplus W})^{\rho_{d,V \oplus W}}(\rho_{d,V \oplus W})!(\ell\rho_{d,T})!
			D\Vol(C^{X/\bP^N}_{m,d}, T_{m,d}^\ell, W_{m,d})
			\prod_{j=1}^m |\alpha_j|_{C^0,\omega}
		\end{aligned}
	\end{equation}
	\begin{equation}\label{eqnpresetidlingaprox}
		\begin{aligned}
			\leq \sum_{\ell = 0}^{e_d} \sum_{k=0}^{e_d-\ell}  {k+\ell \choose \ell} (9e_d + 2 m N + 6m +9)! (\ell\rho_{d,T} + \rho_{d,V\oplus W}+1)^{4(e_d+m+1)+\rho_{d,V \oplus W}}
			\\
			\times
			\rho_{d,V \oplus W}!(\ell\rho_{d,T})!
			D\Vol(C^{X/\bP^N}_{m,d}, T_{m,d}^\ell, W_{m,d})
			\prod_{j=1}^m |\alpha_j|_{C^0,\omega}
		\end{aligned}
	\end{equation}
	\begin{equation}\label{eqnpresetidlingaprox25}
		\begin{aligned}
			\leq \sum_{\ell = 0}^{e_d} \sum_{k=0}^{e_d-\ell}  {k+\ell \choose \ell} (9e_d + 2 m N + 6m +9)! (\ell\rho_{d,T} + \rho_{d,V\oplus W}+1)^{4(e_d+m+1)+\rho_{d,V \oplus W}}
			\\
			\times
			\rho_{d,V \oplus W}^{\rho_{d,V \oplus W}}(\ell\rho_{d,T})!
			D\Vol(C^{X/\bP^N}_{m,d}, T_{m,d}^\ell, W_{m,d})
			\prod_{j=1}^m |\alpha_j|_{C^0,\omega}
		\end{aligned}
	\end{equation}
	\begin{equation} \label{eqnpresetidlingaprox2}
		\begin{aligned}
			\leq \sum_{\ell = 0}^{e_d} \sum_{k=0}^{e_d-\ell}  {k+\ell \choose \ell} (9e_d + 2 m N + 6m +9)! (\ell\rho_{d,T} + \rho_{d,V\oplus W}+1)^{4(e_d+m+1)+2\rho_{d,V \oplus W}}
			\\
			\times(\ell\rho_{d,T})!
			D\Vol(C^{X/\bP^N}_{m,d}, T_{m,d}^\ell, W_{m,d})
			\prod_{j=1}^m |\alpha_j|_{C^0,\omega}
		\end{aligned}
	\end{equation}
	\begin{equation}
		\begin{aligned}
			\overset{\textnormal{\eqref{eqnedeqn},\eqref{eqndefonrorho},\eqref{eqnaefaefforrd}},\eqref{eqnnonnegativity}}{\leq} \sum_{\ell = 0}^{e_d} \sum_{k=0}^{e_d-\ell}  {k+\ell \choose \ell} (9r_d + 2 m N + 6m +9)! (\ell\rho_{d,T} + r_d +1)^{4(r_d+m+1)+2r_d}
			\\ \times (\ell\rho_{d,T})! D\Vol(C^{X/\bP^N}_{m,d}, T_{m,d}^\ell, W_{m,d})
			\prod_{j=1}^m |\alpha_j|_{C^0,\omega}
		\end{aligned}
	\end{equation}
	\begin{equation}
		\begin{aligned}
			\leq \sum_{\ell = 0}^{e_d} \sum_{k=0}^{e_d-\ell}  {k+\ell \choose \ell} (9r_d + 2 m N + 6m +9)! (\ell\rho_{d,T} + r_d +1)^{6(r_d+m+1)}
			\\ \times (\ell\rho_{d,T})! D\Vol(C^{X/\bP^N}_{m,d}, T_{m,d}^\ell, W_{m,d})
			\prod_{j=1}^m |\alpha_j|_{C^0,\omega}
		\end{aligned}
	\end{equation}
	\begin{equation}
		\begin{aligned}
			\leq \sum_{\ell = 0}^{e_d} \sum_{k=0}^{e_d-\ell}  {k+\ell \choose \ell} (9r_d + 2 m N + 6m +9)! ((\ell+1)(\rho_{d,T} + r_d +1))^{6(r_d+m+1)}
			\\ \times (\ell\rho_{d,T})! D\Vol(C^{X/\bP^N}_{m,d}, T_{m,d}^\ell, W_{m,d})
			\prod_{j=1}^m |\alpha_j|_{C^0,\omega}
		\end{aligned}
	\end{equation}
	\begin{equation}
		\begin{aligned}
			= \sum_{\ell = 0}^{e_d} \sum_{k=0}^{e_d-\ell}  {k+\ell \choose \ell} (9r_d + 2 m N + 6m +9)! (\ell+1)^{6(r_d+m+1)}(\rho_{d,T} + r_d +1)^{6(r_d+m+1)}
			\\ \times (\ell\rho_{d,T})! D\Vol(C^{X/\bP^N}_{m,d}, T_{m,d}^\ell, W_{m,d})
			\prod_{j=1}^m |\alpha_j|_{C^0,\omega}
		\end{aligned}
	\end{equation}
	\begin{equation}
		\begin{aligned}
			\leq \sum_{\ell = 0}^{e_d} \sum_{k=0}^{e_d-\ell}  {k+\ell \choose \ell} (9r_d + 2 m N + 6m +9)! (r_d+1)^{6(r_d+m+1)}(\rho_{d,T} + r_d +1)^{6(r_d+m+1)}
			\\ \times (\ell\rho_{d,T})! D\Vol(C^{X/\bP^N}_{m,d}, T_{m,d}^\ell, W_{m,d})
			\prod_{j=1}^m |\alpha_j|_{C^0,\omega}
		\end{aligned}
	\end{equation}
	\begin{equation}
		\begin{aligned}
			\overset{\textnormal{\eqref{eqnreallreallyterribleinequality}}}{\leq}
			\sum_{\ell = 0}^{e_d} \sum_{k=0}^{e_d-\ell}  {k+\ell \choose \ell} \\ \times (9r_d + 2 m N + 6m +9+r_d+1+6(r_d+m+1)+\rho_{d,T} + r_d +1+6(r_d+m+1))!
			\\ \times (\ell\rho_{d,T})! D\Vol(C^{X/\bP^N}_{m,d}, T_{m,d}^\ell, W_{m,d})
			\prod_{j=1}^m |\alpha_j|_{C^0,\omega}
		\end{aligned}
	\end{equation}
	\begin{equation}
		\begin{aligned}
			\leq \sum_{\ell = 0}^{e_d} \sum_{k=0}^{e_d-\ell}  {k+\ell \choose \ell} (23r_d + 2 m N + 18m +\rho_{d,T}+23)!
			(\ell\rho_{d,T})! D\Vol(C^{X/\bP^N}_{m,d}, T_{m,d}^\ell, W_{m,d})
			\prod_{j=1}^m |\alpha_j|_{C^0,\omega}.
		\end{aligned}
	\end{equation}

\end{proof}

\section{Coarse Bounds for Gromov-Witten invariants in Products of Projective space} \label{sectionboundsforproductsofprojhectve}

In this section we wish to give crude, but explicit, bounds for twisted Gromov-Witten invariants of products of projective space. We will use the localization technique of \cite{kontsevich1995enumeration} to do this.

\subsection{Computing Gromov-Witten Invariants for Products of Projective Space} \label{subsectioncomputingviaequivariantlocalization}

Let \begin{equation} \label{eqn49a94f4f4}
	N, M, m, a, a_1,\cdots,a_m, a'_1,\cdots,a'_m, b_1,\cdots,b_m, \ell, i_1,\cdots,i_\ell
\end{equation}
be natural numbers.
Define $P := \bP^N \times \bP^M$ and let $\calO(c,d) := \pi_1^*\calO_{\bP^N}(c) \otimes \pi_2^* \calO_{\bP^M}(d)$ for each $c,d \in \bZ$
where $\pi_1 : P \to \bP^N$ and $\pi_2 : P \to \bP^M$ are the natural projection maps.
Let $d \in H_2^+(P)-0$.
We are interested in an explicit bound for the norm of the following twisted Gromov-Witten invariant:
\begin{equation} \label{eqngromovwitteninvariantbdd}
	\left< c_1(\calO(a_1,a'_1))^{b_1},\cdots,c_1(\calO(a_m,a'_m))^{b_m} ; \prod_{j=1}^lc_{i_j}(\calO(a,0)) \right>^P_{m,d}
\end{equation}
\begin{equation}
	:= \int_{[P_{m,d}]^{\virt}} \prod_{j=1}^\ell c_{i_j}((\pi_{m+1})_*(\ev_{m+1}^*\calO(a,0))) \cup \cup_{j=1}^m \ev_j^*c_1(\calO(a_j,a'_j))^{b_j}.
\end{equation}
We will do this by an equivariant localization argument as in \cite{kontsevich1995enumeration}.
This bound will be very crude. For instance, it is not good enough to prove the corresponding generating function converges, unlike results in \cite{iritani2007convergence} or \cite{zinger2014genus}.
On the other hand, this is an explicit bound depending on $d$ and the parameters \eqref{eqn49a94f4f4} above.

We write $([X_0,\cdots,X_N],[Y_0,\cdots,Y_M])$ for the natural homogenous coordinates on $P$.
We have that $H_2^+(P) \cong \bN^2$
where the basis elements $(1,0)$ and $(0,1)$ correspond to lines in $\bP^N \times [1,0,\cdots,0]$ and $[1,0,\cdots,0] \times \bP^M$ respectively.
Let $d_1,d_2 \in \bN$ be such that $d = (d_1,d_2)$ under this identification.

The group
\begin{equation}
	\bT:= (\bC^*)^{N+1+M+1}
\end{equation}
acts on $P$ via the action:
\begin{equation}
	\bT \times P \to P,
\end{equation}
\begin{equation}
	\begin{aligned}
		(\xi_0,\cdots,\xi_N,\nu_0,\cdots,\nu_M) \cdot ([X_0,\cdots,X_N],[Y_0,\cdots,Y_M]) \\ \to ([\xi_0 X_0,\cdots,\xi_N X_N],[\nu_0Y_0,\cdots,\nu_MY_M]).
	\end{aligned}
\end{equation}
The fixed points of this action are given by $(p_i,q_j)$, $i=0,\cdots,N$, $j=0,\cdots,N$
where
\begin{equation}
	p_i = [\underbrace{0,\cdots,0}_i,1,\underbrace{0,\cdots,0}_{N-i}] \in \bP^N, \quad q_j = [\underbrace{0,\cdots,0}_j,1,\underbrace{0,\cdots,0}_{M-j}] \in \bP^M.
\end{equation}

We have that $\bT$ acts on $P_{m,d}$ by post-composition.
We define $P^{\bT}_{m,d} \subset P_{m,d}$ to be the preimage of the fixed point locus of this action on the underlying coarse moduli space of $P_{m,d}$.

\begin{defn}
	For us a \emph{tree} $T$ will be a finite one dimensional CW complex with some $0$-cells removed so that the resulting space is contractible with at least one zero cell.
	We let $\ver(T)$ be its set of vertices ($0$-cells) and $\edge(T)$ its set of edges ($1$-cells). We let $\leaf(T) \subset \edge(T)$ be the subset of edges where one vertex on that edge has been removed. These are called the \emph{leaves} of $T$. Any edge that is not a leaf is called an \emph{internal edge}.
	We let $i\edge(\Gamma) := \edge(\Gamma) - \leaf(\Gamma)$ be the set of internal edges.
	An \emph{isomorphism} between two such trees is a homeomorphism sending the vertices of one tree to the vertices of the other.

	For each vertex $v \in \ver(T)$, let $\val(v)$ be the number of edges with $v$ as a vertex.
	We define $\ival(v)$ to be the number of internal edges with $v$ as a vertex.
	For each edge $e \in \edge(\Gamma)$, we let $v_e, v'_e$ be its vertices.
\end{defn}

Note that there are no edges with two vertices removed since that would force $T$ to be disconnected or have no vertices.
Let $f : \Sigma \to P$ represent a $\bC$-point of $P^{\bT}_{m,d}$.
We construct a labeled tree $\Gamma_f$ as follows:
We assign one vertex $v$ for each connected component of $f^{-1}(\cup_{i=0}^N \cup_{j=0}^M \{(p_i,q_j)\})$. Let us call this connected component $\Sigma_v \subset \Sigma$.
We label this vertex by the pair of natural numbers $(i_v,j_v)$ satisfying $f(\Sigma_v) = \{(p_{i_v},q_{j_v})\}$.
Note that this connected component $\Sigma_v$ could be a union of irreducible components of $\Sigma$, or it could just be a single point.
We assign one internal edge to each non-constant irreducible component of $\Sigma$. Let $\Sigma_e \subset \Sigma$ be this irreducible component.
We have that $\Sigma_e$ intersects exactly two connected components $\Sigma_v$, $\Sigma_{v'}$ where $v$, $v'$ are vertices of $\Gamma_f$.
In this case we require the edge $e$ to have endpoints $v$ and $v'$ respectively.
We label $e$ by a pair of natural numbers $(d_{e,1},d_{e,2})$ corresponding to $[f|_{\Sigma_e}] \in H_2^+(P)$.
Finally, we assign one leaf for each marked point and label this leaf by the element of $\{1,\cdots,m\}$ corresponding to this marking.
The vertex $v$ of this leaf is the connected component $\Sigma_v$ that the marked point lives in.
Note that if $d_{e,1} \neq 0$ then $d_{e,2} = 0$ and if $d_{e,2}\neq 0$ then $d_{e,1} = 0$.


\begin{defn}
	An \emph{$(m,d)$-admissible tree} is a labelled tree $\Gamma$ where vertices $v$ are labelled by pairs of natural numbers $(i_v,j_v) \in \{0,\cdots,N\} \times \{0,\cdots,M\}$,
	each internal edge $e$ is labelled by a pair of natural numbers $(d_{e,1},d_{e,2}) \in \bN^2 - \{(0,0)\}$ and where there are $m$ leaves labelled by $1,\cdots,m$.
	We also require that
	\begin{enumerate}
		\item $\Gamma$ has at least one internal edge,
		\item if $d_{e,1} \neq 0$ then $d_{e,2} = 0$, if $d_{e,2} \neq 0$ then $d_{e,1} = 0$,
		\item if $v, v'$ are the vertices of an internal edge $e$, then $(i_v,j_v) \neq (i'_v,j'_v)$ and $i_v=i'_v$ if $d_{e,1}=0$ and $j_v=j'_v$ if $d_{e,2}=0$,
		\item and $\sum_{e \in i\edge(\Gamma)} d_{e,1} = d_1$ and $\sum_{e \in i\edge(\Gamma)} d_{e,2} = d_2$.
	\end{enumerate}
	An \emph{isomorphism} between two such $(m,d)$-admissible trees $\Gamma$ and $\Gamma'$ is an isomorphism of the corresponding trees so that
	labelings of the vertices, internal edges and leaves are respected.

	For each $(m,d)$-admissible tree $\Gamma$, define $P^\Gamma \subset P_{m,d}^{\bT}$
	to be the subspace of curves $f$ where $\Gamma_f$ is isomorphic to $\Gamma$.
	Also, define $\edge_1(\Gamma) \subset \edge(\Gamma)$ to be the subset of edges where $d_2 = 0$ and $d_1 \neq 0$
	and $\edge_2(\Gamma) \subset \edge(\Gamma)$ the subset of edges where $d_1 = 0$ and $d_2 \neq 0$.
\end{defn}

Observe that if the $(m,d)$-admissible trees $\Gamma_{f_1}$ and $\Gamma_{f_2}$ are isomorphic for some $\bC$-points $f_1,f_2$ of $P_{m,d}^{\bT}$ then $f_1$ and $f_2$ lie in the same connected component of $P_{m,d}^{\bT}$.

\begin{defn}
	Let $\Gamma$ be a $(m,d)$-admissible tree.
	We let $\pi_\Gamma : \Sigma_\Gamma \to P^\Gamma$ be the corresponding universal curve.
	Let $\ev : \Sigma_\Gamma \to P$ be the evaluation map.

	Let $I(\Gamma)$ be the set of irreducible components of $\Sigma_\Gamma$ and $I_{\edge}(\Gamma) \subset I(\Gamma)$ the subset of non-constant components
	and $I_{\ver}(\Gamma) \subset I(\Gamma)$ the subset of constant components.
	We define $I_1(\Gamma) \subset I(\Gamma)$ the subspace of irreducible components $\Sigma$ so that $\pi_2 \circ \ev|_{\Sigma}$ is constant but $\pi_1 \circ \ev|_{\Sigma}$ is not.
	We define $I_2(\Gamma) \subset I(\Gamma)$ the subspace of irreducible components $\Sigma$ so that $\pi_1 \circ \ev|_{\Sigma}$ is constant but $\pi_2 \circ \ev|_{\Sigma}$ is not.

	Let $n(\Gamma)$ be the set of sections of $\pi_\Gamma$ whose image is contained in the nodal locus of $\Sigma_\Gamma$.
	Define $n_{2\edge}(\Gamma) \subset n(\Gamma)$ to be the subspace of sections whose image is contained in two non-constant components,
	$n_{1\edge}(\Gamma) \subset n(\Gamma)$ to be the subspace of sections whose image is contained in exactly one non-constant component of $\Sigma_\Gamma$.
	For each $s \in n(\Gamma)$, let $\Sigma_s$ and $\Sigma'_s$ be the irreducible components of $\Sigma_\Gamma$ containing the image of $s$.
	If $s \in n_{1\edge}(\Gamma)$, then we will order these components so that $\Sigma_s$ is the constant component of $\Sigma_\Gamma$.

	For each irreducible component $\Sigma' \in I(\Gamma)$, let $T_{\Sigma'/P^\Gamma}$ be the dual of the relative dualizing sheaf.
	For each $s \in n(\Gamma)$, define $T^\otimes_s$ to be $s^* (T_{\Sigma_s/P^\Gamma} \otimes T_{\Sigma'_s/P^\Gamma})$
	and define $T^\oplus_s$ to be $s^* (T_{\Sigma_s/P^\Gamma} \oplus T_{\Sigma'_s/P^\Gamma})$.
	Also, for each $s \in n_{1\edge}(\Gamma)$, define $T^{\ver}_s := s^*T_{\Sigma_s/P^\Gamma}$ and $T^{\edge}_s := s^*T_{\Sigma'_s/P^\Gamma}$.

	We let $h(\Gamma)$ be the set of marked point sections of $\Sigma_\Gamma$.
	For each $s \in h(\Gamma)$ let $\Sigma_s$ be the irreducible component of $\Sigma_\Gamma$ containing the image of $s$ and define $T_s := s^* T_{\Sigma_s/P^\Gamma}$.

\end{defn}

We have that $n(\Gamma) = n_{2\edge}(\Gamma) \sqcup n_{1\edge}(\Gamma)$ since a generic element of $P^\Gamma$ has no nodes connecting constant components since any such node can be smoothed and also since $(d_1,d_2) \neq (0,0)$.

We wish to compute \eqref{eqngromovwitteninvariantbdd}
using the Atiyah-Bott localization formula.
As a result, we need to compute the equivariant Euler class of the normal bundle $N_\Gamma$ to $P^\Gamma$ inside $P_{m,d}$ for each $(m,d)$-admissible tree $\Gamma$.
We will first compute it in the equivariant $K$-theory group $K^0_{\bT}(P^\Gamma) \otimes \bQ = K^0(P^\Gamma) \otimes \bQ[\bT^\vee]$
where $\bT^\vee$ is the group of characters $\bT\to \bC^*$ of $\bT$.

By reasoning as in \cite[Section 3.3.1]{kontsevich1995enumeration} and \cite[Section 3.3.5]{kontsevich1995enumeration}, we have
the following identities in $K^0_{\bT}(P^\Gamma) \otimes \bQ$:
\begin{equation}
	[N_\Gamma] = [T(P_{m,d})|_{P^\Gamma}] - [TP^\Gamma],
\end{equation}
\begin{equation}
	[T(P_{m,d})|_{P^\Gamma}] = [(TP)_{m,d}|_{P^\Gamma}] + \sum_{s \in n(\Gamma)} T^\otimes_s + T^\oplus_s - \sum_{\Sigma \in I(\Gamma)} [(\pi_\Gamma|_{\Sigma})_*T_{\Sigma/P^\Gamma}] +
	\sum_{s \in h(\Gamma)} T_s
\end{equation}
and
\begin{equation}
	[TP^\Gamma] = \sum_{s \in n_{1\edge}(\Gamma)} [T^\ver_\Gamma] - \sum_{\Sigma \in I_{\ver}(\Gamma)} [(\pi_\Gamma|_\Sigma)_*T_{\Sigma/P^\Gamma}] + \sum_{s \in h(\Gamma)} [T_s].
\end{equation}

Hence:
\begin{equation} \label{eqnaefaef444422}
	[N_\Gamma] = [(TP)_{m,d}|_{P^\Gamma}] + [N'_\Gamma]
\end{equation}
where
\begin{equation}
	[N'_\Gamma] = \sum_{s \in n(\Gamma)} T^\otimes_s + T^\oplus_s - \sum_{\Sigma \in I(\Gamma)} [(\pi_\Gamma|_{\Sigma})_*T_{\Sigma/P^\Gamma}]
	- \sum_{s \in n_{1\edge}(\Gamma)} [T^\ver_\Gamma] + \sum_{\Sigma \in I_{\ver}(\Gamma)} [(\pi_\Gamma)_*T_{\Sigma/P^\Gamma}]
\end{equation}
\begin{equation}
	= \sum_{s \in n(\Gamma)} [T^\otimes_s] + \sum_{s \in n_{2\edge}(\Gamma)} [T^{\oplus s}] + \sum_{s \in n_{1\edge}(\Gamma)} [T^{\oplus s}] - \sum_{s \in n_{1\edge}(\Gamma)} [T^\ver_\Gamma]
	- \sum_{\Sigma \in I_{\edge}(\Gamma)} [(\pi_\Gamma)_*T_{\Sigma/P^\Gamma}]
\end{equation}
\begin{equation}
	= \sum_{s \in n(\Gamma)} [T^\otimes_s] + \sum_{s \in n_{2\edge}(\Gamma)} [T^{\oplus s}] + \sum_{s \in n_{1\edge}(\Gamma)} [T^{\edge}_s]
	- \sum_{\Sigma \in I_{\edge}(\Gamma)} [(\pi_\Gamma)_*T_{\Sigma/P^\Gamma}]
\end{equation}
\begin{equation} \label{eqnaefaef4sum2}
	= \sum_{s \in n_{2\edge}(\Gamma)} [T^\otimes_s] + \sum_{s \in n_{1\edge}(\Gamma)} [T^\otimes_s] + \sum_{s \in n_{2\edge}(\Gamma)} [T^{\oplus}_s] + \sum_{s \in n_{1\edge}(\Gamma)} [T^{\edge}_s] - \sum_{\Sigma \in I_{\edge}(\Gamma)} [(\pi_\Gamma)_*T_{\Sigma/P^\Gamma}]
\end{equation}
\begin{equation} \label{eqnaefaef4sum3}
	= \sum_{s \in n_{2\edge}(\Gamma)} ([T^\otimes_s]+[T^{\oplus}_s]) + \sum_{s \in n_{1\edge}(\Gamma)} ([T^\otimes_s] + [T^{\edge}_s]) - \sum_{\Sigma \in I_{\edge}(\Gamma)} [(\pi_\Gamma)_*T_{\Sigma/P^\Gamma}]
\end{equation}
\begin{equation} \label{eqnaefaef4sum4}
	= \sum_{s \in n_{2\edge}(\Gamma)} ([T^\otimes_s]+[T^{\oplus}_s]) + \sum_{s \in n_{1\edge}(\Gamma)} ([T^{\ver}_s \otimes T^{\edge}_s] + [T^{\edge}_s]) - \sum_{\Sigma \in I_{\edge}(\Gamma)} [(\pi_\Gamma)_*T_{\Sigma/P^\Gamma}].
\end{equation}

Combining the calculation above with \eqref{eqnaefaef444422} gives
\begin{equation} \label{eqnaefaefforaefe}
	[N_\Gamma] =  [(TP)_{m,d}|_{P^\Gamma}] + \sum_{s \in n_{2\edge}(\Gamma)} ([T^\otimes_s]+[T^{\oplus}_s]) + \sum_{s \in n_{1\edge}(\Gamma)} ([T^{\ver}_s \otimes T^{\edge}_s] + [T^{\edge}_s]) - \sum_{\Sigma \in I_{\edge}(\Gamma)} [(\pi_\Gamma)_*T_{\Sigma/P^\Gamma}].
\end{equation}

We now wish to compute the integral of the equivariant Euler class of
\begin{equation} \label{eqnaefaef24054295}
	\sum_{j=1}^m b_j[\ev_j^* \calO(a_j,a'_j)|_{P^\Gamma}] - [N_\Gamma] \in K^0_{\bT}(P^\Gamma) \otimes \bQ
\end{equation}
plus an additional term coming from summands of $(\pi_{m+1})_*(\ev_{m+1}^* \calO(a,0))|_{P^\Gamma}$
against $[P^\Gamma]$.
By the Atiyah-Bott localization formula the sum of such integrals over all $(m,d)$-admissible trees $\Gamma$ will give \eqref{eqngromovwitteninvariantbdd}.
The only vector bundle in the sum \eqref{eqnaefaef24054295} which is non-trivial as a non-equivariant vector bundle is $-[N_\Gamma]$
and the only sum in \eqref{eqnaefaefforaefe} which is non-trivial as a non-equivariant vector bundle is $\sum_{s \in n_{1\edge}(\Gamma)} [T^{\ver}_s \otimes T^{\edge}_s]$.
Also, $(\pi_{m+1})_*(\ev_{m+1}^* \calO(a,0))|_{P^\Gamma}$ is trivial as a non-equivariant vector bundle.
As a result, we need to compute the integral over $[P^\Gamma]$ of the equivariant Euler class $e_{\bT}$ of
\begin{equation} \label{eaefanefsumne}
	-\sum_{s \in n_{1\edge}(\Gamma)} [T^{\ver}_s \otimes T^{\edge}_s].
\end{equation}
and then multiply this by the equivariant Euler classes of the other bundles in \eqref{eqnaefaef24054295}, \eqref{eqnaefaefforaefe} and summands of $(\pi_{m+1})_*(\ev_{m+1}^* \calO(a,0))$, viewed as Laurent polynomials in the equivariant parameters
since they are pullbacks of $\bT$-representations via the map $P^\Gamma \to \textnormal{point}$.

We have that $\bT^\vee$ is generated by the characters $\lambda_0,\cdots,\lambda_N, \mu_0,\cdots,\mu_N$
where $\lambda_i$ is the projection map to the $(i+1)$\textsuperscript{st} factor of $\bT= (\bC^*)^{N+1+M+1}$ for each $i=0,\cdots,N$
and $\mu_i$ is the projection map to the $(i+N+2)$\textsuperscript{nd} factor for each $i=0,\cdots,M$.
Hence,
\begin{equation}
	H^*_{\bT}(P^\Gamma;\bQ)_{\textnormal{loc}} = H^*(P^\Gamma;\bQ)(\lambda_0,\cdots,\lambda_N,\mu_0,\cdots,\mu_N)
\end{equation}
where ``$\textnormal{loc}$'' means that we localize with respect to the multiplicative system of polynomials in the equivariant parameters with coefficients in $H^0(\bP^\Gamma;\bQ) \cong \bQ$.

Let us start with the term
\begin{equation} \label{eqnaaefaef4440045}
	[N_\Gamma] - [(TP)_{m,d}|_{P^\Gamma}].
\end{equation}
\begin{defn} \label{defnflags}
	A \emph{flag} of $\Gamma$ is a pair $(v,e) \in \ver(\Gamma) \times i\edge(\Gamma)$ where $v$ is a vertex of $e$.
	For any such flag $F$, we write $v_F := v$ and $v'_F$ to be the other vertex of the internal edge $e$.
	A \emph{$\pi_1$-flag} is a flag $(v,e)$ where $e \in \edge_1(\Gamma)$.
	A \emph{$\pi_2$-flag} is a flag $(v,e)$ where $e \in \edge_2(\Gamma)$.
	For each $\pi_1$-flag $F = (v,e)$, we define
	\begin{equation} \label{eqnaeforwF0}
		w_F := \frac{\lambda_{i_{v_F}} - \lambda_{i_{v'_F}}}{d_{e,1}}.
	\end{equation}
	For each $\pi_2$-flag $F = (v,e)$, we define
	\begin{equation} \label{eqnaeforwF1}
		w_F := \frac{\mu_{j_{v_F}} - \mu_{j_{v'_F}}}{d_{e,2}}.
	\end{equation}
	Also, for each flag $F = (v,e)$ we define $s_F \in n(\Gamma)$ to be the unique element with the property that either $\Sigma_{s_F}$ or $\Sigma'_{s_F}$ corresponds to $e$
	and that $(i_v,j_v)$ is in the image of $\ev \circ s_F$.

\end{defn}
Here $w_F$ represents the equivariant Euler class of $T^{\edge}_{s_F}$ if $s_F \in n_{1\edge}(\Gamma)$.
If $F,F'$ are distinct flags with the same vertex $v$ and $s_F \in n_{2\edge}(\Gamma)$ then $w_F + w_{F'}$ represents the equivariant Euler class of $T^\otimes_{s_F}$.
By reasoning as in \cite[Sections 3.3.2 and 3.3.3]{kontsevich1995enumeration}, we get:
\begin{equation} \label{eqnaefaefaefaefeafeafeafaefee}
	\begin{aligned}
		\int_{[P^\Gamma]} e_\bT([N_\Gamma] - [(TP)_{m,d}|_{P^\Gamma}] + \sum_{e \in i\edge(\Gamma)}[\bC]) = \\
		\prod_{v \in \ver(\Gamma)} \left(\left(\sum_{\textnormal{flags} \ F = (v,e)} w_F^{-1} \right)^{\val(v)-3} \prod_{\textnormal{flags} \ F = (v,e)} w_F^{-1} \right)
	\end{aligned}
\end{equation}
where $[\bC]$ is the class of the trivial vector bundle of rank $1$ where $\bT$ acts trivially.
Here the term $\sum_{s \in n_{1\edge}(\Gamma)} [T^{\ver}_s \otimes T^{\edge}_s]$ in
\eqref{eqnaefaef4sum4} corresponds to the product over all $v$ in \eqref{eqnaefaefaefaefeafeafeafaefee} where the valency of $v$ is $\geq 3$. The term $\sum_{s \in n_{2\edge}(\Gamma)} [T^\otimes_s]$ in \eqref{eqnaefaef4sum4} corresponds to the product over all $v$ in \eqref{eqnaefaefaefaefeafeafeafaefee} where the valency of $v$ is $2$.
The remaining terms in \eqref{eqnaefaef4sum4} correspond to the product over all $v$ in \eqref{eqnaefaefaefaefeafeafeafaefee} where the valency of $v$ is $1$.

Now let us now compute the equivariant Euler class of the term
\begin{equation} \label{eqnrqeewnaefaef0404}
	[(TP)_{m,d}|_{P^\Gamma}] = [(\pi_1^*T\bP^N)_{m,d}|_{P^\Gamma}] + [(\pi_2^*T\bP^M)_{m,d}|_{P^\Gamma}].
\end{equation}

\begin{defn}
	Let $Q_\Gamma^1$ be the set of connected components of $(\pi_1 \circ \ev)^{-1}(\{p_0,\cdots,p_N\}) \subset \Sigma_\Gamma$.
	Similarly, let $Q_\Gamma^2$ be the set of connected components of $(\pi_2 \circ \ev)^{-1}(\{q_0,\cdots,q_N\}) \subset \Sigma_\Gamma$.
	Let $k=0$ or $1$.
	For each connected component $\Sigma' \in Q_\Gamma^k$, let $\ival(\Sigma') \in \bN$ be the number of $e \in \edge_k(\Gamma)$
	satisfying $\Sigma_e \cap Q_\Gamma \neq \emptyset$.
	Define $\check{Q}_\Gamma^k \subset Q_\Gamma^k$ to be the subset of $\Sigma'$ satisfying $\ival(\Sigma')>1$.
	For each $\Sigma' \in Q_{\Gamma}^1$, let $i_{\Sigma'} \in \{0,\cdots,N\}$ be such that $\{p_{i_{\Sigma'}}\} = \pi_1(\ev(\Sigma'))$.
	For each $\Sigma' \in Q_{\Gamma}^2$, let $j_{\Sigma'} \in \{0,\cdots,M\}$ be such that $\{q_{j_{\Sigma'}}\} = \pi_2(\ev(\Sigma'))$.
\end{defn}


By the discussion in \cite[3.3.4]{kontsevich1995enumeration}, we get:
\begin{equation} \label{eqnaefaeaefegamma1}
	\begin{aligned}
		e_\Gamma([(\pi_1^*T\bP^N)_{m,d}|_{P^\Gamma}] - \sum_{e \in \edge_1(\Gamma)}[\bC]) = \\
		\prod_{e \in \edge_1(\Gamma)} \left(
		\frac{(-1)^{d_{e,1}}\left(\frac{d_{e,1}}{\lambda_{i_{v_e}} - \lambda_{i_{v'_e}}} \right)^{2d_{e,1}}}{(d_{e,1}!)^2} \prod_{k \neq i_{v_e},i_{v'_e}} \prod_{\overset{a',b' \geq 0}{a'+b'=d_{e,1}}}
		\frac{1}{\frac{a'}{d_{e,1}}\lambda_{i_{v_e}}+\frac{b'}{d_{e,1}}\lambda_{i_{v'_e}}-\lambda_k}
		\right) \times                                                                      \\
		\prod_{\Sigma' \in \check{Q}_\Gamma^1}  \left( \prod_{j \neq i_{\Sigma'}} \lambda_{i_{\Sigma'}} - \lambda_j \right)^{\ival(\Sigma')-1}
	\end{aligned}
\end{equation}
and
\begin{equation} \label{eqnaefaeaefegamma12}
	\begin{aligned}
		e_\Gamma(([(\pi_2^*T\bP^N)_{m,d}|_{P^\Gamma}] - \sum_{e \in \edge_2(\Gamma)}[\bC])) = \\
		\prod_{e \in \edge_2(\Gamma)} \left(
		\frac{(-1)^{d_{e,2}}\left(\frac{d_{e,2}}{\mu_{j_{v_e}} - \mu_{j_{v'_e}}} \right)^{2d_{e,2}}}{(d_{e,2}!)^2} \prod_{k \neq j_{v_e},j_{v'_e}} \prod_{\overset{a',b' \geq 0}{a'+b'=d_{e,2}}}
		\frac{1}{\frac{a'}{d_{e,2}}\mu_{j_{v_e}}+\frac{b'}{d_{e,2}}\mu_{j_{v'_e}}-\mu_k}
		\right) \times                                                                        \\
		\prod_{\Sigma' \in \check{Q}_\Gamma^2}  \left( \prod_{j \neq j_{\Sigma'}} \lambda_{j_{\Sigma'}} - \lambda_j \right)^{\ival(\Sigma')-1}
	\end{aligned}
\end{equation}
So, $e_\bT([(TP)_{m,d}|_{P^\Gamma}])$ is the product of \eqref{eqnaefaeaefegamma1} with \eqref{eqnaefaeaefegamma12}.
%
%
%
\begin{defn} \label{defnequivariantchernclass}
	We define $c_i^\bT(W)$ to be the $i$th equivariant Chern class of any $\bT$-equivariant vector bundle $W$. We define $c^\bT(W) = \sum_{i \in \bN} c_i^\bT(W)$ to be the total equivariant Chern class.
\end{defn}

The bundle $(\pi_{m+1})_*(\ev_{m+1}^*\calO(a,0))|_{P^\Gamma}$ sits in an exact sequence:
\begin{equation} \label{eqnshortexactsequecneoa0}
	\begin{aligned}
		0 \to (\pi_\Gamma)_*(\ev^*\calO(a,0)) \to \bigoplus_{e \in \edge_1(\Gamma)} (\pi_\Gamma|_{\Sigma_e})_*(\ev^* \calO(a,0)|_{\Sigma_e})
		\\
		\to \bigoplus_{\Sigma' \in \check{Q}_\Gamma^1} (\pi_\Gamma|_{\Sigma'})_* (\ev^*\calO(a,0)|_{\Sigma'}) \otimes \bC^{\ival(\Sigma')-1} \to 0.
	\end{aligned}
\end{equation}
For each $e \in \edge_1(\Gamma)$, we get $(\pi_\Gamma|_{\Sigma_e})_*(\ev^* \calO(a,0)|_{\Sigma_e})$
splits as a direct sum of line bundles $\oplus_{a' = 0}^{ad_{e,1}} L_{e,a'}$
where
\begin{equation} \label{eqnc1caluclation1}
	c_1^{\bT}(L_{e,a'}) = \frac{a' \lambda_{i_{v_e}}+(ad_{e,1}-a') \lambda_{i_{v'_e}}}{d_{e,1}}.
\end{equation}
Also, for $\Sigma' \in Q_\Gamma^1$, we have:
\begin{equation} \label{eqnc1caluclation2}
	c_1^{\bT}((\pi_\Gamma|_{\Sigma'})_* (\ev^*\calO(a,0))|_{\Sigma'}) = a \lambda_{i_{\Sigma'}}.
\end{equation}

\begin{defn} \label{defneagamma}
	We define
	\begin{equation}
		\edge(\Gamma,a) := \left\{ (e,a')  \ : \ e \in \edge(\Gamma), \ d_{e,1} \neq 0, \ a' \in \{0,\cdots,ad_{e,1}\} \right\}.
	\end{equation}
\end{defn}


By \eqref{eqnshortexactsequecneoa0},  \eqref{eqnc1caluclation1} and  \eqref{eqnc1caluclation2}, we have
\begin{equation}
	\begin{aligned}
		 & c^{\bT}((\pi_\Gamma)_*(\ev^*\calO(a,0))) = \\ & \prod_{(e,a) \in \edge(\Gamma,a)} \left(1+\frac{a' \lambda_{i_{v_e}}+(ad_{e,1}-a') \lambda_{i_{v'_e}}}{d_{e,1}} \right)
		\prod_{\Sigma' \in \check{Q}_\Gamma^1} \left(1 + a\lambda_{i_{\Sigma'}}\right)^{1-\ival(\Sigma')}
		=                                             \\ &
		\prod_{(e,a) \in \edge(\Gamma,a)} \left(1+\frac{a' \lambda_{i_{v_e}}+(ad_{e,1}-a') \lambda_{i_{v'_e}}}{d_{e,1}} \right)
		\prod_{\Sigma' \in \check{Q}_\Gamma^1} \left(\sum_{k \in \bN} (-1)^k {k + \ival(\Sigma') - 2 \choose k}(a \lambda_{i_{\Sigma'}})^k \right).
	\end{aligned}
\end{equation}
Hence:
\begin{equation} \label{eqnawfawfrrr4}
	\begin{aligned}
		\prod_{j=1}^\ell c_{i_j}^{\bT}((\pi_\Gamma)_*(\ev^*\calO(a,0))) =
		\prod_{j=1}^\ell \sum_{\overset{S_1 \subset \edge(\Gamma,a), \ f : \check{Q}_\Gamma^1 \to \bN,}{|S_1|+\sum_{\Sigma' \in \check{Q}_\Gamma^1}f(\Sigma')=i_j}}
		\\
		\left(\prod_{(e,a') \in S_1} \frac{a' \lambda_{i_{v_e}}+(ad_{e,1}-a') \lambda_{i_{v'_e}}}{d_{e,1}} \right)
		\prod_{\Sigma' \in \check{Q}_\Gamma^1} (-1)^{f(\Sigma')} {f(\Sigma')+\ival(\Sigma')-2 \choose f(\Sigma')} (a\lambda_{i_{\Sigma'}})^{f(\Sigma')}.
	\end{aligned}
\end{equation}

\begin{defn} \label{defnnumberofmarkedpoitsine}
	For each $h \in \leaf(\Gamma)$, define $(i_h,j_h) \in \{0,\cdots,N\} \times \{0,\cdots,M\}$ so that the image of $\ev_h|_{\Sigma_\Gamma}$ is $(p_{i_h},q_{j_h})$.
\end{defn}

We have
\begin{equation} \label{aef494a0w9kg40w49gkw0g9kdsvsk0}
	e_\bT([\ev_h^* \calO(a_h,a'_h)|_{P^\Gamma}]) = \lambda_{i_h}^{a_h} \mu_{j_h}^{a'_h} \quad \forall \ h \in \leaf(\Gamma).
\end{equation}

Hence, by Atiyah-Bott localization we have:
\begin{equation} \label{eqnaefaefeat4t4t}
	\textnormal{\eqref{eqngromovwitteninvariantbdd}} =
	\sum_{(m,d)\textnormal{-admissible trees} \ \Gamma}
	\frac{1}{|\Aut(\Gamma)|\prod_{e \in \edge(\Gamma)}(d_{e,1}+d_{e,2})}  \textnormal{\eqref{eqnaefaefaefaefeafeafeafaefee}} \times \textnormal{\eqref{eqnaefaeaefegamma1}} \times \textnormal{\eqref{eqnaefaeaefegamma12}} \times \textnormal{\eqref{eqnawfawfrrr4}} \times \textnormal{\eqref{aef494a0w9kg40w49gkw0g9kdsvsk0}}.
\end{equation}

\subsection{Bounds for Gromov-Witten Invariants of Products of Projective Space} \label{subsectionaefeafaef4t4g4g4}

We need to find an upper bound for the absolute value of \eqref{eqnaefaefeat4t4t}.
The left-hand side of this equation does not depend on the parameters $\lambda_0,\cdots,\lambda_N,\mu_0,\cdots,\mu_M$.
As a result, we only need to bound the right-hand side of \eqref{eqnaefaefeat4t4t} for specific values of these parameters.
Let
\begin{equation}
	|d| = \sum_{e \in \edge(\Gamma)} (d_{e,1} + d_{e,2}).
\end{equation}
We will now fix the following values for $\lambda_j$ and $\mu_j$ for each appropriate $j$:
\begin{equation} \label{eqnaefaef494944}
	\lambda_j = 1 + j + \left(\frac{1}{|d|+1}\right)^{N+1-j}, \quad \mu_j = 1 + j + \left(\frac{1}{|d|+1}\right)^{M+1-j}.
\end{equation}

Cayley's formula states that the number of trees with $k$ vertices, each one uniquely labeled from $1$ to $k$, is $k^{k-2}$.
Hence:
\begin{equation} \label{eqnverticesesseg}
	\sum_{\textnormal{trees} \ T \ \textnormal{with} \ k \ \textnormal{vertices}} \frac{1}{|\textnormal{Aut}(T)|} \leq k^{k-2}/k!.
\end{equation}
By Stirling approximation we have
\begin{equation} \label{eqnstirlingupperobuund}
	k^k/k! \leq e^{12k+1} \leq 4^{12k+1}, \quad \forall \ k \in \bN.
\end{equation}
Hence,
\begin{equation} \label{eqnverticesesseg2}
	\sum_{\textnormal{trees} \ T \ \textnormal{with} \ k \ \textnormal{vertices}} \frac{1}{|\textnormal{Aut}(T)|} \leq 4^{12k+1}
\end{equation}
where $\Aut(T)$ is the group of automorphism of the tree $T$.
Hence,
\begin{equation} \label{eqne49f949f49}
	\begin{aligned}
		\sum_{(m,d)\textnormal{-admissible trees} \ \Gamma}
		\frac{1}{|\Aut(\Gamma)|} \leq 4^{12(|d|+m+1)+1} \times m! \times |d|^{|d|} & \overset{\textnormal{\eqref{eqnstirlingupperobuund}}}{\leq}  4^{12(|d|+m+1)+1} m! |d|! 4^{12|d|+1}
		\\
		                                                                           & \leq 4^{24(|d|+m+1)} m!|d|!
	\end{aligned}
\end{equation}
where $\Aut(\Gamma)$ is the group of automorphisms of the $(m,d)$-admissible tree $\Gamma$.

\begin{defn}
	Recall that a \emph{partition} of a number $k$ is a sequence $1 \leq a_1 \leq \cdots \leq a_\ell$
	whose sum is $k$.
	The \emph{norm} of such a partition is the product $\prod_{j=1}^\ell a_\ell$.
\end{defn}

Using the fact that the norm of a partition of $k$ is at most $4^k$ (\cite[Page 188]{Halmostproblemsformathematicians}) for each $k \in \bN$
and the fact that the valency of each vertex of each $(m,d)$-admissible tree $\Gamma$ is at most $M+N$, we get that the norm of \eqref{eqnaefaefaefaefeafeafeafaefee} with the substitutions \eqref{eqnaefaef494944}
is less than or equal to:
\begin{equation} \label{eqnaefaef8282}
	\prod_{v \in \ver(\Gamma)} \left( \left((N+M)\max_{\textnormal{flags} \ F=(v,e)}(d_{e,1}+d_{e,2})\right)^{M+N} \prod_{\textnormal{flags} \ F=(v,e)} (d_{e,1}+d_{e,2}) \right)
\end{equation}
\begin{equation}
	\leq \prod_{v \in \ver(\Gamma)}\left((M+N) \left( \prod_{\textnormal{flags} \ F=(v,e)} (d_{e,1}+d_{e,2}) \right)^2  \right)^{M+N}
\end{equation}

\begin{equation}
	\leq \left(\prod_{v \in \ver(\Gamma)}(M+N)^{(M+N)} \right)\left( \prod_{v \in \ver(\Gamma)}\prod_{\textnormal{flags} \ F=(v,e)} (d_{e,1}+d_{e,2}) \right)^{2(M+N)}
\end{equation}
\begin{equation}
	\leq (M+N)^{(M+N)(|d|+1)} \left( \prod_{v \in \ver(\Gamma)}\prod_{\textnormal{flags} \ F=(v,e)} (d_{e,1}+d_{e,2}) \right)^{2(M+N)}
\end{equation}
\begin{equation}
	\leq (M+N)^{(M+N)(|d|+1)} \left( (4^{|d|})^2 \right)^{2(M+N)}
\end{equation}
Hence
\begin{equation} \label{eqnaefaef828}
	|\textnormal{\eqref{eqnaefaefaefaefeafeafeafaefee}}| \leq (4(M+N))^{5(M+N)(|d|+1)}.
\end{equation}

Under the substitutions \eqref{eqnaefaef494944},
each of the terms
\begin{equation}
	\frac{a'}{d_{e,1}}\lambda_{i_{v_e}}+\frac{b'}{d_{e,1}}\lambda_{i_{v'_e}}-\lambda_k, \quad a',b' \geq 0, \ a'+b' = d_{e,1}, \ k \in \{0,\cdots,N\} - \{i_{v_e},i_{v'_e}\}, \ e \in \edge_1(\Gamma)
\end{equation}
has norm at least $\frac{1}{(|d|+1)^{N+2}}$ and hence at least $\frac{1}{(|d|+1)^{N+M+2}}$.
Hence:
\begin{equation} \label{aeaee49494}
	|\textnormal{\eqref{eqnaefaeaefegamma1}}| \leq \prod_{e \in \edge_1(\Gamma)} \frac{d_{e,1}^{2d_{e,1}}}{(d_{e,1}!)^2}  \left( \prod_{k \neq i_{v_e},i_{v'_e}} \prod_{\overset{a',b' \geq 0}{a'+b'=d_{e,1}}}(|d|+1)^{N+M+2} \right) \prod_{v \in \ver(\Gamma)} (N+M+2)^{(N+M+2)}
\end{equation}
\begin{equation}
	\leq |d|^{2|d|} \left(\prod_{e \in \edge_1(\Gamma)}(|d|+1)^{(N+M+2)d_{e,1}(N+M)}\right)(N+M+2)^{(N+M+2)(|d|+1)}
\end{equation}
\begin{equation}
	\leq |d|^{2|d|} (|d|+1)^{(N+M+2)|d|(N+M+2)}(N+M+2)^{(N+M+2)(|d|+1)}
\end{equation}
\begin{equation}
	\leq |d|^{2|d|} (|d|+1)^{|d|(N+M+2)^2}(N+M+2)^{(N+M+2)(|d|+1)}
\end{equation}
\begin{equation}
	\leq |d|^{(N+M+2)|d|} (|d|+1)^{|d|(N+M+2)^2}(N+M+2)^{(N+M+2)(|d|+1)}
\end{equation}
\begin{equation}
	\leq (|d|+1)^{|d|(N+M+2)^2}(|d|(N+M+2))^{(N+M+2)^2(|d|+1)}.
\end{equation}
\begin{equation} \label{aeaee4949423}
	|\textnormal{\eqref{eqnaefaeaefegamma1}}| \leq ((N+M+2)(|d|+1))^{2(|d|+1)(N+M+2)^2}.
\end{equation}
Similarly,
\begin{equation} \label{aeaee4949423M}
	|\textnormal{\eqref{eqnaefaeaefegamma12}}| \leq ((N+M+2)(|d|+1))^{2(|d|+1)(N+M+2)^2}.
\end{equation}

We have
\begin{equation} \label{eqnvalencayofeachvertex}
	\begin{aligned}
		 & |\textnormal{\eqref{eqnawfawfrrr4}}| \leq \prod_{j=1}^\ell \sum_{\overset{S_1 \subset \edge(\Gamma,a) \ f : Q'_\Gamma \to \bN}{|S_1|+\sum_{\Sigma' \in \check{Q}_\Gamma^1}f(\Sigma')=i_j}} \\ &
		\left(\prod_{(e,a') \in S_1}  a(N+2)\right) \prod_{\Sigma' \in \check{Q}_\Gamma^1} {f(\Sigma')+\ival(\Sigma')-2 \choose f(\Sigma')} (a(N+2))^{f(\Sigma')}.
	\end{aligned}
\end{equation}
Again, using the fact that the valency of each vertex of each $(m,d)$-admissible tree $\Gamma$ is at most $N+M$, we get \eqref{eqnvalencayofeachvertex} is less than or equal to
\begin{equation}
	\prod_{j=1}^\ell \sum_{\overset{S_1 \subset \edge(\Gamma,a) \ f : \check{Q}_\Gamma^1 \to \bN}{|S_1|+\sum_{\Sigma' \in \check{Q}_\Gamma^1}f(\Sigma')=i_j}}
	\left(a(N+2)\right)^{i_j} (a(N+2))^{i_j} 2^{i_j + (N+M)-2}
\end{equation}
\begin{equation}
	\leq \prod_{j=1}^\ell \sum_{\overset{S_1 \subset \edge(\Gamma,a) \ f : \check{Q}_\Gamma^1 \to \bN}{|S_1|+\sum_{\Sigma' \in \check{Q}_\Gamma^1}f(\Sigma')=i_j}}
	\left(2a(N+2)\right)^{2i_j+(N+M)-2}
\end{equation}
\begin{equation}
	\leq \prod_{j=1}^\ell (|d| \times a|d|)^{i_j} {|d| + i_j \choose i_j}
	\left(2a(N+2)\right)^{2i_j+(N+M)-2}
\end{equation}
\begin{equation}
	\leq \prod_{j=1}^\ell (|d| \times a|d|)^{i_j} (|d|+i_j)^{i_j}
	\left(2a(N+2)\right)^{2i_j+(N+M)-2}
\end{equation}
\begin{equation}
	\leq (|d| \times a|d|)^{\sum_{j=1}^\ell i_j}
	\left(2a(N+2)\right)^{2\sum_{j=1}^\ell i_j+\ell(N+M)-2\ell} \prod_{j=1}^\ell(|d|+i_j)^{i_j}
\end{equation}
\begin{equation}
	\leq (a|d|)^{2\sum_{j=1}^\ell i_j} \left(2a(N+2)\right)^{2\sum_{j=1}^\ell i_j+\ell(N+M)-2\ell} \prod_{j=1}^\ell(|d|+i_j)^{i_j}.
\end{equation}

Hence,
\begin{equation} \label{eqn445903406936424t204}
	|\textnormal{\eqref{eqnawfawfrrr4}}| \leq (a|d|)^{2\sum_{j=1}^\ell i_j} \left(2a(N+2)\right)^{2\sum_{j=1}^\ell i_j+\ell(N+M)-2\ell} \prod_{j=1}^\ell(|d|+i_j)^{i_j}.
\end{equation}

We have
\begin{equation} \label{aef494a0w9kg40w49gkw0g9kdsvsk02222}
	|\textnormal{\eqref{aef494a0w9kg40w49gkw0g9kdsvsk0}}| \leq (N+2)^{a_j}(M+2)^{a'_j}.
\end{equation}

Hence, by \eqref{eqnaefaefeat4t4t}, \eqref{eqne49f949f49}, \eqref{eqnaefaef828}, \eqref{aeaee4949423}, \eqref{aeaee4949423M}, \eqref{eqn445903406936424t204}, \eqref{aef494a0w9kg40w49gkw0g9kdsvsk02222} we get
\begin{equation} \label{eqnaefaef4f4f00040411w17}
	\begin{aligned}
		|\textnormal{\eqref{eqngromovwitteninvariantbdd}}| \leq &
		m! 4^{24(|d|+m+1)} |d|! \times                            \\ &
		(4(M+N))^{5(M+N)(|d|+1)}
		(((N+M+2)(|d|+1))^{2(|d|+1)(N+M+2)^2})^2
		\times
		\\ & (a|d|)^{2\sum_{j=1}^\ell i_j} \left(2a(N+2)\right)^{2\sum_{j=1}^\ell i_j+\ell(N+M)-2\ell}\prod_{j=1}^\ell(|d|+i_j)^{i_j}  \prod_{j=1}^m (N+2)^{a_j}(M+2)^{a'_j}.
	\end{aligned}
\end{equation}
Now, \eqref{eqngromovwitteninvariantbdd} vanishes if
$\sum_{j=1}^\ell i_j > \dim(P_{0,d})$.
Hence, we can assume
\begin{equation} \label{eqn4vansihesefe}
	\sum_{j=1}^\ell i_j \leq (N+M+2)(|d|+1).
\end{equation}
Combining this with \eqref{eqnaefaef4f4f00040411w17} gives:
\begin{equation} \label{eqnaefaef4f4f00040411w1720}
	\begin{aligned}
		|\textnormal{\eqref{eqngromovwitteninvariantbdd}}| \leq &
		m! 4^{24(|d|+m+1)} |d|!\times                                                                       \\ &
		(4(M+N))^{5(M+N)(|d|+1)}
		(((N+M+2)(|d|+1))^{2(|d|+1)(N+M+2)^2})^2 \times                                                     \\ &
		(a|d|)^{2(M+N+2)(|d|+1)} \left(2a(N+2)\right)^{2(M+N+2)(|d|+1)+(M+N+2)(|d|+1)(N+M)-2(M+N+2)(|d|+1)} \\ &
		(|d|+(M+N+2)(|d|+1))^{(M+N+2)(|d|+1)} \prod_{j=1}^m (N+2)^{a_j}(M+2)^{a'_j}
	\end{aligned}
\end{equation}
\begin{equation} \label{eqnaefaef4f4f00040411w17202}
	\begin{aligned}
		\leq &
		m! 4^{24(|d|+m+1)} |d|!\times                                                                       \\ &
		(4(M+N))^{5(M+N)(|d|+1)}
		(((N+M+2)(|d|+1))^{2(|d|+1)(N+M+2)^2})^2 \times                                                     \\ &
		(a|d|)^{2(M+N+2)(|d|+1)} \left(2a(N+2)\right)^{2(M+N+2)(|d|+1)+(M+N+2)(|d|+1)(N+M)-2(M+N+2)(|d|+1)} \\ &
		((M+N+3)(|d|+1))^{(M+N+2)(|d|+1)} \prod_{j=1}^m (N+2)^{a_j}(M+2)^{a'_j}.
	\end{aligned}
\end{equation}
\begin{equation}
	\begin{aligned}
		\leq &
		m! 4^{24(|d|+m+1)} |d|!\times                                            \\ &
		(4(M+N))^{5(M+N)(|d|+1)}
		(((N+M+2)(|d|+1))^{2(|d|+1)(N+M+2)^2})^2 \times                          \\ &
		(a|d|)^{2(M+N+2)(|d|+1)} \left(2a(N+2)\right)^{3(M+N+2)^2(|d|+1)} \times \\ &
		((M+N+3)(|d|+1))^{(M+N+2)(|d|+1)}\prod_{j=1}^m (N+2)^{a_j}(M+2)^{a'_j}.
	\end{aligned}
\end{equation}
\begin{equation}
	\begin{aligned}
		\overset{\textnormal{\eqref{eqnreallreallyterribleinequality}}}{\leq}
		((m+4+24(|d|+m+1))+|d|+ 4(M+N) + 5(M+N)(|d|+1)                  \\ +2((N+M+2)(|d|+1)+2(|d|+1)(N+M+2)^2) + a|d| + 2(M+N+2)(|d|+1) +\\
		(2a(N+1)) + 3(M+N+2)^2(|d|+1) + (M+N+3)(|d|+1)+(M+N+2)(|d|+1))! \\ \prod_{j=1}^m (N+2)^{a_j}(M+2)^{a'_j}
	\end{aligned}
\end{equation}
\begin{equation}
	\begin{aligned}
		\leq (25m+4+(24+1+4+5+2+4+a+2+2a+3+1+1)(M+N+3)^2(|d|+1))! \\ \times \prod_{j=1}^m (N+2)^{a_j}(M+2)^{a'_j}.
	\end{aligned}
\end{equation}
\begin{equation}
	\begin{aligned}
		\leq (25m+(51+3a)(M+N+3)^2(|d|+1))!  \prod_{j=1}^m (N+2)^{a_j}(M+2)^{a'_j}.
	\end{aligned}
\end{equation}

Hence, we get the following proposition:
\begin{prop} \label{propGWboundsnonses}
	\begin{equation} \label{eqnaefaef4f4f00040411w172}
		\begin{aligned}
			|\textnormal{\eqref{eqngromovwitteninvariantbdd}}| \leq (25m+(51+3a)(M+N+3)^2(|d|+1))!  \prod_{j=1}^m (N+2)^{a_j}(M+2)^{a'_j}.
		\end{aligned}
	\end{equation}
\end{prop}

\subsection{\texorpdfstring{$D$}--Volume Bounds in Products of Projective Space} \label{subsectionDVolumeboundsproductsofprojective}

Let us fix the notation from the last subsection.

\begin{prop} \label{propdvolumeboundsproductsfo}
	Let $h' \in \bN$ satisfy $h' \geq a+2(N+1)(M+1)$ and suppose $D$ is a very general divisor representing $\calO_{\bP^N \times \bP^M}(h',h')$.
	Recall that $T = \calO_{\bP^N \times \bP^M}(a,0)^b$ for some $b \in \bN$.
	Then
	\begin{equation} \label{eqneafeafaeifaoeg}
		\begin{aligned}
			 & D\Vol(T_{m,d}) \leq  (b+(52+3a)(M+N+3)^2(|d|+1) + (N+M+2h'+30)m)!
			/ (abd_1 + b)!
		\end{aligned}
	\end{equation}
\end{prop}
\begin{proof}
	The dimension of $P_{0,d}$ is at most $\ell := (N+M+2)(|d|+1)$.
	We have for each $i_1,\cdots,i_m \in \bN$ that
	\begin{equation}
		\langle D^{i_1},\cdots,D^{i_m} \rangle_{m,d}^{s(T_\bullet)} =
	\end{equation}
	\begin{equation}
		\langle D^{i_1},\cdots,D^{i_m}; c(\calO(a,0)^b)^{-1} \rangle^{\bP^N \times \bP^M}_{m,d}
	\end{equation}
	\begin{equation}
		= \langle D^{i_1},\cdots,D^{i_m}; (1+c_1(\calO(a,0)))^{-b} \rangle^{\bP^N \times \bP^M}_{m,d}
	\end{equation}
	\begin{equation}
		=\sum_{k = 0}^\ell (-1)^k {k+b-1 \choose k} \langle D^{i_1},\cdots,D^{i_m} ; c_1(\calO(a,0))^k \rangle^{\bP^N \times \bP^M}_{m,d}.
	\end{equation}
	Hence, by Proposition \ref{propGWboundsnonses},
	\begin{equation}
		\sum_{i_1,\cdots,i_m=0}^{NM} \left| \langle D^{i_1},\cdots,D^{i_m} \rangle_{m,d}^{s(T_\bullet)} \right|
	\end{equation}
	\begin{equation}
		\begin{aligned}
			\leq (NM+1)^m \left(\sum_{k=0}^\ell {k+b-1 \choose k} \right) \\
			\times (25m+(51+3a)(M+N+3)^2(|d|+1))! \prod_{j=1}^m (N+2)^{h'}(M+2)^{h'}
		\end{aligned}
	\end{equation}
	\begin{equation}
		\begin{aligned}
			\leq (NM+1)^m 2^{\ell+b-1}
			\times (25m+(51+3a)(M+N+3)^2(|d|+1))! \prod_{j=1}^m (N+2)^{h'}(M+2)^{h'}
		\end{aligned}
	\end{equation}
	\begin{equation}
		\begin{aligned}
			\overset{\textnormal{\eqref{eqnreallreallyterribleinequality}}}{\leq} (NM+1+m+2+\ell+b-1+ \\
			25m+(51+3a)(M+N+3)^2(|d|+1) + (N+2 +h' +M+2+h')m)!
		\end{aligned}
	\end{equation}
	\begin{equation}
		\leq (b+\ell+26m+NM+2+
		(51+3a)(M+N+3)^2(|d|+1) + (N+M+2h'+4)m)!
	\end{equation}
	\begin{equation}
		\begin{aligned}
			= (b+(N+M+2)(|d|+1)+26m+NM+2+ \\
			(51+3a)(M+N+3)^2(|d|+1) + (N+M+2h'+4)m)!
		\end{aligned}
	\end{equation}
	\begin{equation}
		\leq (b+
		(52+3a)(M+N+3)^2(|d|+1) + (N+M+2h'+30)m)!
	\end{equation}
	Since $D$ is the restriction of a very general divisor of degree $h'$ pulled back via the Segre embedding
	\begin{equation}
		\bP^N \times \bP^M \hookrightarrow \bP^{(N+1)(N+1)-1},
	\end{equation}
	we have that $D$ is $T$-large by Lemma \ref{lemmavlargeexist}.
	Our result now follows from Proposition \ref{propsegreboundsforDvolume}.
\end{proof}


\section{Main Argument} \label{sectionmainargument}

In this section we prove the main theorem of this paper, Theorem \ref{maintheorem}. We will first construct bounded partial resolutions of the normal bundle of our variety in projective space by sums of twists of the structure sheaf. Such bounds depend on the Castelnuovo-Mumford regularity of the defining ideal sheaf and its square. We use such bounds to embed the projective compactification of the normal bundle into products of projective space. After that we assemble all the ingredients together to prove the main theorem.
Throughout this section we will fix the following notation:
\begin{enumerate}
	\item We let $X \subset \bP^N$ be a smooth projective subvariety with defining ideal sheaf $\frI_X$.
	\item We let $\iota_X : X \hookrightarrow \bP^N$ be the corresponding closed immersion.
	\item We let $\omega$ be the restriction of the Fubini Study form on $\bP^N$ to $X$.
\end{enumerate}

\begin{defn} \label{defncastelnuovo}
	Let $\scrF$ be a sheaf on $\bP^N$.
	We say that $\scrF$ is \emph{$r$-regular} or \emph{$r$-regular in the sense of Castelnuovo-Mumford} if
	\begin{equation}
		H^i(\scrF(r-i)) = 0
	\end{equation}
	for each $i \in \bN_{>0}$.
	Define $\reg(\scrF) \in \bN$ to be the infimum over all $r \in \bN$ so that $\scrF$ is $r$-regular.
\end{defn}


\begin{prop} \label{propositionrgeqkregular} (Castelnuovo) \cite[Lecture 14]{mumford2016lectures}.
	If a sheaf $\scrF$ is $r$-regular then it is $k$-regular for any $k \geq r$.
\end{prop}

\begin{prop}(Castelnuovo)  \cite[Lecture 14]{mumford2016lectures}. \label{propositioncastelnuovothm}
	If a coherent sheaf $\scrF$ is $r$-regular then $\scrF(k)$ is globally generated for each $k \geq r$.
\end{prop}

\begin{lemma} \label{lemmaregularityca}
	Let $0 \to \scrF \to \scrG \to \scrH \to 0$
	be a short exact sequence of sheaves on $\bP^N$.
	Then
	\begin{equation} \label{eqnregequations1}
		\reg(\scrH) \leq \max(\reg(\scrG),\reg(\scrF)-1),
	\end{equation}
	\begin{equation} \label{eqnregequations2}
		\reg(\scrG) \leq \max(\reg(\scrF),\reg(\scrH)).
	\end{equation}
	Also, if the induced map $H^0(\scrG(r-1)) \to H^0(\scrH(r-1))$ is surjective for some $r \geq \max(\reg(\scrH)+1,\reg(\scrG))$ then
	\begin{equation} \label{eqnregequations3}
		\reg(\scrF) \leq r.
	\end{equation}
\end{lemma}
\begin{proof}
	Let $r_\scrH = \max(\reg(\scrG),\reg(\scrF)-1)$, $r_\scrG = \max(\reg(\scrF),\reg(\scrH))$
	and  $r_\scrF = \max(\reg(\scrH)+1,\reg(\scrG))$.
	We have an exact sequence:
	\begin{equation} \label{eqnshortexactsequecneregular}
		H^i(\scrF(r_\scrG-i)) \to H^i(\scrG(r_\scrG-i)) \to H^i(\scrH(r_\scrG-i)).
	\end{equation}
	The outer terms vanish for all $i>0$ by Proposition \ref{propositionrgeqkregular} and hence the middle term does. This proves \eqref{eqnregequations2}.

	We have an exact sequence:
	\begin{equation} \label{eqnshortexactsequecneregular2}
		H^i(\scrG(r_\scrH-i)) \to H^i(\scrH(r_\scrH-i)) \to H^{i+1}(\scrF(r_\scrH+1-(i+1))).
	\end{equation}
	Again, the outer terms vanish for each $i>0$ and hence middle term does. This proves  \eqref{eqnregequations1}.

	Let $r$ be as in the statement of this lemma.
	We have an exact sequence
	\begin{equation} \label{eqnshortexactsequecneregular3}
		H^{i-1}(\scrH(r-1-(i-1))) \to H^i(\scrF(r-i)) \to H^i(\scrG(r-i)).
	\end{equation}
	The outer terms vanish for each $i>1$ since $r \geq r_\scrF$ and hence so does the middle term. We only need to consider the case $i=1$. In this case we have an exact sequence:
	\begin{equation} \label{eqnshortexactsequecneregular4}
		H^0(\scrG(r-1)) \lra{f} H^0(\scrH(r-1)) \to H^1(\scrF(r-1)) \to H^1(\scrG(r-1)).
	\end{equation}
	Since $f$ is surjective by assumption, we get $H^1(\scrF(r-1)) = 0$ since $r \geq \reg(\scrG)$.
\end{proof}

\begin{lemma} \label{lemmadimensionfoquotientsinxe}
	Let $K$ be an $\calO_{\bP^N}$-submodule of $(\frI_X^\ell)^{\oplus k}$ containing $(\frI_X^{\ell+1})^{\oplus k}$ for some $k,\ell \in \bN$. Suppose that $K$ and $\frI_X^{\ell+1}$ are $r$-regular for some $r \in \bN$.
	Then $K/(\frI_X^{\ell+1})^{\oplus k}$ is $r$-regular and
	\begin{equation} \label{eqneeKholds}
		\dim H^0((K/(\frI_X^{\ell+1})^{\oplus k})(r)) \leq k{N+r \choose N}.
	\end{equation}
\end{lemma}
\begin{proof}
	We have a short exact sequence
	\begin{equation}
		0 \to (\frI_X^{\ell+1})^{\oplus k} \to K \to K / (\frI_X^{\ell+1})^{\oplus k} \to 0.
	\end{equation}
	Then by Lemma \ref{lemmaregularityca} we get that $K / (\frI_X^{\ell+1})^{\oplus k}$ is $r$-regular.
	We have that $H^0(K(r)) \subset H^0(\calO_{\bP^N}^{\oplus k}(r))$.
	Hence,
	\begin{equation} \label{eqndimnplusrbounds}
		\dim H^0(K(r)) \leq \dim H^0(\calO_{\bP^N}^{\oplus k}(r)) = k{N + r \choose r}.
	\end{equation}

	Since $H^1((\frI_X^{\ell+1})^{\oplus k}(r)) = 0$ by Proposition \ref{propositionrgeqkregular} we get $H^0(K(r)) \to H^0((K/(\frI_X^{\ell+1})^{\oplus k})(r))$ is surjective.
	Hence,
	\begin{equation} \label{equpperboundforkover}
		\dim H^0((K/(\frI_X^{\ell+1})^{\oplus k})(r)) \leq \dim H^0(K(r)) \overset{\textnormal{\eqref{eqndimnplusrbounds}}}{\leq} k{N + r \choose r}.
	\end{equation}
\end{proof}

\begin{lemma} \label{lemmainductionstepforexactsequences}
	Let $K$ be a coherent subsheaf of $(\frI_X^\ell)^{\oplus k}$ containing $(\frI_X^{\ell+1})^{\oplus k}$ for some $k,\ell \in \bN$. Suppose that $K$ and $\frI_X^{\ell+1}$ are $r$-regular for some $r \in \bN$.
	Let $h \geq k {N + r \choose N}$.
	Then there is a coherent subsheaf $K'$ of $\calO_{\bP^N}^{\oplus h}$ containing $\frI_X^{\oplus h}$ which is $1$-regular together with a short exact sequence:
	\begin{equation} \label{eqnasersersedsiere}
		0 \to K'/ \frI_X^{\oplus h} \to (\calO_{\bP^N}/\frI_X)^{\oplus h} \to (K/(\frI_X^{\ell+1})^{\oplus k})(r) \to 0.
	\end{equation}
\end{lemma}
\begin{proof}
	By Lemma \ref{lemmadimensionfoquotientsinxe} we have $K / (\frI_X^{\ell+1})^{\oplus k}$ is $r$-regular and Equation \eqref{eqneeKholds} holds.
	Therefore, by Proposition \ref{propositioncastelnuovothm} there is a surjection
	\begin{equation}
		f : \calO_{\bP^N}^{\oplus h} \to (K/(\frI_X^{\ell+1})^{\oplus k})(r)
	\end{equation}
	which is also a surjection on $H^0$. Let $K'$ be in the kernel of $f$.
	Since $\calO_{\bP^N}^{\oplus h}$ is $0$-regular by \cite[Proposition 2.1.12]{EGA3}, $(K/(\frI_X^{\ell+1})^{\oplus k})(r)$ is $0$-regular
	and $H^0(f)$ is surjective, we have by Lemma \ref{lemmaregularityca}
	that $K'$ is $1$-regular.
	Since $K$ is a subsheaf of $(\frI_X^\ell)^{\oplus k}$, we get that $\frI_X^{\oplus h} \subset \calO_{\bP^N}^{\oplus h}$ gets sent to $0$ under $f$.
	Therefore, we have an induced surjection
	\begin{equation}
		(\calO_{\bP^N}/\frI_X)^{\oplus h} \to (K/(\frI_X^{\ell+1})^{\oplus k})(r).
	\end{equation}
	The kernel of this map is $K' / \frI_X^{\oplus h}$.
	Hence, we have an exact sequence \eqref{eqnasersersedsiere}.
\end{proof}

\begin{lemma} \label{lemmapulledbackexactsequence}
	Suppose that $\frI_X$, $\frI_X^\ell$ and $\frI_X^{\ell+1}$ are $r$-regular for some $r \in \bN$ and $\ell \in \bN$. Let $r_1 \geq r$ and $r_2,r_3 \geq 1$.
	Let $a_1,a_2,a_3 \in \bN$ satisfy
	\begin{equation}
		a_1 \geq {N+r_1 \choose N}, \ a_2 \geq a_1 {N+r_2 \choose N}, \ a_3 \geq a_2 {N + r_3 \choose N}.
	\end{equation}
	Then there is an exact sequence of sheaves:
	\begin{equation} \label{eqnpulledbackexactsequence}
		\calO_X^{\oplus a_3}(-r_1-r_2-r_3) \to
		\calO_X^{\oplus a_2}(-r_1-r_2) \to \calO_X^{\oplus a_1}(-r_1) \to \iota_X^*(\frI_X^\ell) \to 0.
	\end{equation}
\end{lemma}
\begin{proof}
	By Lemma \ref{lemmainductionstepforexactsequences}, there exists a coherent subsheaf $K_1$ of $\calO_{\bP^N}^{\oplus a_1}$ containing $\frI_X^{a_1}$ which is $1$-regular
	together with a short exact sequence:
	\begin{equation} \label{eqnshortexactsequence1}
		0 \to K_1/\frI_X^{\oplus a_1} \to (\calO_{\bP^N}/\frI_X)^{\oplus a_1} \to (\frI_X^\ell/\frI_X^{\ell+1})(r_1) \to 0.
	\end{equation}
	By Proposition \ref{propositionrgeqkregular}, $K_1$ and $\frI_X$ are $r_2$-regular. Hence, by Lemma \ref{lemmainductionstepforexactsequences} there exists a coherent subsheaf $K_2$ of $\calO_{\bP^N}^{\oplus a_2}$ containing $\frI_X^{a_2}$ which is $1$-regular together with a short exact sequence:
	\begin{equation} \label{eqnshortexactsequence2}
		0 \to K_2/\frI_X^{\oplus a_2} \to (\calO_{\bP^N}/\frI_X)^{\oplus a_2} \to \left(K_1/\frI_X^{\oplus a_1}\right)(r_2) \to 0.
	\end{equation}
	By Proposition \ref{propositionrgeqkregular}, $K_2$ and $\frI_X$ are  $r_3$-regular. Hence, by Lemma \ref{lemmainductionstepforexactsequences} there exists a coherent subsheaf $K_3$ of $\calO_{\bP^N}^{\oplus a_3}$ containing $\frI_X^{a_3}$ which is $1$-regular together with a short exact sequence:
	\begin{equation} \label{eqnshortexactsequence3}
		0 \to K_3/\frI_X^{\oplus a_3} \to (\calO_{\bP^N}/\frI_X)^{\oplus a_3} \to \left(K_2/\frI_X^{\oplus a_2}\right)(r_3) \to 0.
	\end{equation}

	By combining \eqref{eqnshortexactsequence1},\eqref{eqnshortexactsequence2} and \eqref{eqnshortexactsequence3} we get a long exact sequence:
	\begin{equation} \label{eqncombiningthreeexact}
		\begin{aligned}
			(\calO_{\bP^N}/\frI_X)^{\oplus a_3} \to
			(\calO_{\bP^N}/\frI_X)^{\oplus a_2}(r_3) \to \\ (\calO_{\bP^N}/\frI_X)^{\oplus a_1}(r_3+r_2) \to \left(\frI_X^\ell/\frI_X^{\ell+1}\right)(r_3+r_2+r_1) \to 0.
		\end{aligned}
	\end{equation}
	Hence, we get a long exact sequence:
	\begin{equation} \label{eqncombiningthreeexact2}
		\begin{aligned}
			(\calO_{\bP^N}/\frI_X)^{\oplus a_3}(-r_1-r_2-r_3) \to
			(\calO_{\bP^N}/\frI_X)^{\oplus a_2}(-r_1-r_2) \\
			\to  (\calO_{\bP^N}/\frI_X)^{\oplus a_1}(-r_1) \to \left(\frI_X^\ell/\frI_X^{\ell+1}\right) \to 0.
		\end{aligned}
	\end{equation}

	By \cite[Lemma 04CJ]{stacks-project} this restricts to an exact sequence \eqref{eqnpulledbackexactsequence}.
\end{proof}

\begin{lemma} \label{lemmaexactsequence}
	Suppose that $\frI_X$ and $\frI_X^2$ are $r$-regular for some $r>0$.
	Then there is an exact sequence of sheaves
	\begin{equation} \label{eqnexactsequence}
		0 \to V \lra{e_V} \calO_X(r)^{{N+r \choose N}} \to \calO_X(2r)^{{N+r \choose N}^2} \to \calO_X(3r)^{{N+r \choose N}^3}
	\end{equation}
	on $X$
	where $V$ is the normal bundle of $X$.
\end{lemma}
\begin{proof}
	By Lemma \ref{lemmapulledbackexactsequence} we get an exact sequence:
	\begin{equation}
		\calO_X^{\oplus {N+r \choose N}^3}(-3r) \to
		\calO_X^{\oplus {N+r \choose N}^2}(-2r) \to \calO_X^{\oplus {N+r \choose N}}(-r) \to V^* \to 0.
	\end{equation}
	We then have that \eqref{eqnexactsequence} is the dual of the sequence above.
\end{proof}

\begin{lemma} \label{lemmaeafefaefaecastelnuovo}
	Suppose that $\frI_X$ is $r$-regular for some $r>0$.
	Then there is an exact sequence
	\begin{equation} \label{eqnexactsequence2}
		0 \to \calO_X \lra{} \calO_X(2r)^{{N+r \choose N}^2} \lra{} \calO_X(3r)^{{N+r \choose N}^3}
	\end{equation}
\end{lemma}
\begin{proof}
	We have $\calO_{\bP^N}$ is $r$-regular by \cite[Proposition 2.1.12]{EGA3} combined with Proposition \ref{propositionrgeqkregular}.
	Hence, by Lemma \ref{lemmapulledbackexactsequence} we get an exact sequence:
	\begin{equation}
		\calO_X^{\oplus {N+r \choose N}^3}(-3r) \to \calO_X^{\oplus {N+r \choose N}^2}(-2r) \to \calO_X \to 0.
	\end{equation}
	We then have that \eqref{eqnexactsequence2} is the dual of the sequence above.
\end{proof}

Let us now fix $r \in \bN_{>0}$ so that $\frI_X$ and $\frI_X^2$ are $r$-regular.
Let $W$ be the cokernel of the map $e_V$ from the exact sequence \eqref{eqnexactsequence}.
This is a vector bundle since $e_V$ is injective and its domain and codomain are vector bundles.
Define
\begin{equation} \label{eqndefnofmu}
	\mu := {N + r \choose N}.
\end{equation}
By \eqref{eqnexactsequence} and \eqref{eqnexactsequence2}, we have an exact sequence:
\begin{equation} \label{eqnexateeuaenfeaiet}
	0 \to W \oplus \calO_X \lra{e_W} \calO_X(2r)^{2\mu^2} \to \calO_X(3r)^{2\mu^3}.
\end{equation}
The map $e_W$ in \eqref{eqnexateeuaenfeaiet} induces the following closed immersion:
\begin{equation} \label{eqnclosedummersionofPnbundles}
	\iota_{P_W} : P_W \lra{} \bP(\calO_X(2r)^{2\mu^2}) \cong X \times \bP^{2\mu^2-1}
\end{equation}
where $P_W$ is as in Definition \ref{defnprojectivecompactification}.
As a result we have a closed immersion
\begin{equation} \label{eqnclosedimmersionf}
	f : P_W \lra{\iota_{P_W}} X \times \bP^{2\mu^2-1} \lra{\iota_X \times \id} \bP^N \times \bP^{2\mu^2-1}.
\end{equation}
We let \begin{equation}
	\widetilde{V} \to X \times \bP^{2\mu^2-1}
\end{equation} be the normal bundle of $X \times \bP^{2\mu^2-1}$ inside $\bP^N \times \bP^{2\mu^2-1}$.
This is the pullback of $V$ to $X \times \bP^{2\mu^2-1}$ via the natural projection map to $X$.
Let
\begin{equation}
	\widetilde{Q} \to P_W
\end{equation}
be the normal bundle of $\iota_{P_W}(P_W)$ inside $X \times \bP^{2\mu^2-1}$.
Let us now compute this bundle.
The map $\iota_{P_W}$ corresponds to the exact sequence:
\begin{equation} \label{firstexactsequence}
	0 \to W(-2r) \oplus \calO_X(-2r) \lra{e_W \otimes \id_{\calO_X(-2r)}} \calO_X^{2\mu^2}.
\end{equation}
Let $Q'$ be the cokernel of $e_W \otimes \id_{\calO_X(-2r)}$.
Then
\begin{equation} \label{eqnequationformalbundle0}
	\widetilde{Q} \cong \pi_{P_W}^*Q' \otimes f^*\calO_{\bP^N \times \bP^{2\mu^2-1}}(0,1).
\end{equation}
We also have that:
\begin{equation} \label{eqnWprimeW}
	Q' \cong Q \otimes \calO_X(-2r)
\end{equation}
where $Q$ is the cokernel of the map $e_W$.
Hence,
\begin{equation} \label{eqnequationformalbundle}
	\widetilde{Q} \cong \pi_{P_W}^*(Q \otimes \calO_X(-2r)) \otimes \calO_{\bP^N \times \bP^{2\mu^2-1}}(0,1)|_{P_W}
	= \pi_{P_W}^*Q \otimes \calO_{\bP^N \times \bP^{2\mu^2-1}}(-2r,1)|_{P_W}.
\end{equation}


\begin{defn} \label{defnprojectionmapandex}
	Define
	$\calO_{\bP^N \times \bP}(a,b) := \calO_{\bP^N \times \bP^{2\mu^2-1}}(a,b)$
	and $\calO_{P_W}(a,b) := \calO_{\bP^N \times \bP}(a,b)|_{P_W}$
	for each $a,b \in \bZ$.
	We let $\pi_X  : X \times \bP^{2\mu^2-1} \to X$ be the natural projection map.
	Define $e_X$ to be the natural composition:
	\begin{equation}
		e_X : X \hookrightarrow W \hookrightarrow P_W.
	\end{equation}
	Define $d_{P_W} = (e_X)_*d$.
	Also, we define $T := \calO_{\bP^N \times \bP}(r,0)^\mu$.
	We let
	$d_{X \times \bP}$ be the image of $d$ in $H_2^+(X \times \bP^{2\mu^2-1})$
	and $d_{\bP^N \times \bP}$ its image in $H_2^+(\bP^N \times \bP^{2\mu^2-1})$.
\end{defn}

\begin{lemma} \label{lemmapullbackbundles}
	We have
	\begin{equation} \label{eqntrivialrestriction}
		f^*(\calO_{P_W})(a,b) = \pi_{P_W}^*(\calO_X(a + 2br)) \otimes \calO(\infty_W)^{\otimes b}
	\end{equation}
	and
	\begin{equation}  \label{eqntrivialrestriction2}
		e_X^* f^*(\calO_{\bP^N \times \bP})(a,b) = \calO_X(a + 2br)
	\end{equation}
	for each $a,b \in \bZ$.
\end{lemma}
\begin{proof}
	Because $\pi_X \circ f = \pi_{P_W}$, we have
	\begin{equation} \label{eqnfirstoneforab}
		f^*\calO_{\bP^N \times \bP}(1,0) = \pi_{P_W}^*\calO_X(1).
	\end{equation}
	Also, since Equation \eqref{firstexactsequence} is the map inducing the inclusion map $\iota_{P_W}$, we have
	\begin{equation} \label{eqnfirstoneforab2}
		f^*\calO_{P_W}(0,1) = \pi_{P_W}^*\calO_X(2r) \otimes \calO(\infty_W).
	\end{equation}
	Equation \eqref{eqntrivialrestriction} now follows from \eqref{eqnfirstoneforab}, \eqref{eqnfirstoneforab2} and linearity.
	Equation \eqref{eqntrivialrestriction2} also holds because $e_X^* \calO(\infty_W) = \calO_X$ and $\pi_{P_W} \circ e_X = \id_X$.
\end{proof}

\begin{lemma} \label{lemmapullbacks}
	We have the following identifications:
	\begin{equation} \label{eqnvectorbundlepullbackidentifications}
		f^*T \cong \pi_{P_W}^*(f|_X^*T)
	\end{equation}
	\begin{equation} \label{eqnvectorbundlepullbackidentifications2}
		\widetilde{V} \cong \pi_X^*(e_X^*\iota_{P_W}^*\widetilde{V})
	\end{equation}
	\begin{equation} \label{eqnvectorbundlepullbackidentifications3}
		\widetilde{Q}|_W \cong \pi_{P_W}^*(e_X^*\widetilde{Q})|_W \cong (\pi_{P_W}^*Q)|_W.
	\end{equation}
	Also, for each $d \in H_2^+(X)$, $m \in \bN$, we have an identification
	\begin{equation} \label{restrictionofcones}
		C^{X \times \bP^{2\mu^2-1}/\bP^N \times \bP^{2\mu^2-1}}_{m,d_{X \times \bP}}|_W \cong (\pi_{P_W})_{m,d}^*C^{X/\bP^N}_{m,d}|_W
	\end{equation}
	induced by the natural identification
	\begin{equation}
		(\widetilde{V}|_{P_W})_{m,d_{P_W}} \cong (\pi_{P_W})_{m,d}^*V_{m,d}
	\end{equation}
	where
	\begin{equation} \label{eqnidentifactionprojection}
		(\pi_{P_W})_{m,d} : (\widetilde{V}|_{P_W})_{m,d_{P_W}} \to V_{m,d}
	\end{equation}
	is the map of moduli spaces induced by $\pi_{P_W}$.
	We also have a natural identification
	\begin{equation} \label{eqnrestriction2}
		C^{X \times \bP^{2\mu^2-1}/\bP^N \times \bP^{2\mu^2-1}}_{m,d_{X \times \bP}}|_X \cong C^{X/\bP^N}_{m,d}.
	\end{equation}
\end{lemma}
\begin{proof}
	Equation \eqref{eqnvectorbundlepullbackidentifications} follows from Lemma \ref{lemmapullbackbundles}.
	We have that Equation \eqref{eqnvectorbundlepullbackidentifications2} holds because $V = e_X^*\iota_{P_W}^*\widetilde{V}$.
	We have
	\begin{equation} \label{eqnidxeqn}
		\id_X = \pi_{P_W} \circ e_X.
	\end{equation}
	Equation \eqref{eqnvectorbundlepullbackidentifications3} holds by Lemma \ref{lemmapullbackbundles} combined with Equations \eqref{eqnequationformalbundle} and \eqref{eqnidxeqn} and the fact that $\infty_W$ has support disjoint from $W$.

	Let \begin{equation} \label{eqnidentifactionprojection}
		(\pi_X)_{m,d} : \widetilde{V}_{m,d_{X \times \bP}} \to V_{m,d}
	\end{equation}
	be the map of moduli spaces induced by $\pi_X$.
	Let $U \subset (X \times \bP^{2\mu^2-1})_{m,d_{X \times \bP}}$ be the subspace of curves which are still stable after post composing them with $\pi_X$.
	Recall that $\frM_m$ is the moduli stack of genus zero pre-stable curves with $m$ marked points.
	Because the moduli space of prestable curves on a product is the fiber product over $\frM_{m}$ of the respective moduli spaces on each factor, we have an identification
	\begin{equation} \label{eqnpullbacknormalbundlesf}
		C^{X \times \bP^{2\mu^2-1}/\bP^N \times \bP^{2\mu^2-1}}_{m,d_{X \times \bP}}|_U \cong (\pi_X)_{m,d}^*C^{X/\bP^N}_{m,d}|_U.
	\end{equation}
	Equation \eqref{restrictionofcones}
	now holds since $\pi_X \circ \iota_{P_W} = \pi_{P_W}$.
	Also Equation \eqref{eqnrestriction2} holds by \eqref{eqnpullbacknormalbundlesf} combined with $\pi_X \circ \iota_{P_W} \circ e_X = \id_X$.
\end{proof}

\begin{lemma} \label{lemmaconvexityofbundles}
	We have that $V$, $W$, $Q$ and $\widetilde{Q}|_W$ are strongly convex (Definition \ref{defnconvex}).
	Also, we have $c_1(\widetilde{Q})(d_{p_W})=c_1(Q)(d)$ and
	\begin{equation} \label{eqnhomologyclassofcurveinproduct}
		f_*(d_{P_W}) = (\deg(d),2r\deg(d)) \in \bN^2 \cong H_2^+(\bP^N \times \bP^{2\mu^2-1})
	\end{equation}
	for each $d \in H_2^+(X)$
	where
	$\deg(d)$ is the degree of the image of $d$ in $H_2^+(\bP^N) \cong \bN$.
\end{lemma}
\begin{proof}
	We have surjections $T\bP^N|_X \twoheadrightarrow V$,
	$\calO_X(r)^{\mu} \twoheadrightarrow W$ and $\calO_X(2r)^{2\mu^2} \to Q$.
	Hence, by Lemma \ref{lemmaconvexquotient}, we get that $V$, $W$ and $Q$ are strongly convex.
	Lemma \ref{lemmapullbacks} then tells us that $\widetilde{Q}|_W$ is also strongly convex.
	By Equation \eqref{eqnidxeqn}, we have $(\pi_{P_W})_*(d_{P_W}) = d$.
	Hence, by Lemma \ref{lemmapullbacks}, we have
	$c_1(\widetilde{Q})(d_{p_W})=c_1(Q)(d)$.
	Equation \eqref{eqnhomologyclassofcurveinproduct}
	follows from Lemma \ref{lemmapullbackbundles}.
\end{proof}

\begin{lemma} \label{lemmanormalconeforW0hd}
	Let $d \in H_2^+(X)$ and $m \in \bN$.
	We have that
	$C^{P_W/X \times \bP^{2\mu^2-1}}_{m,d_{P_W}}|_W$ is equal to
	$\widetilde{Q}_{m,d_{P_W}}|_W$ (See Definition \ref{defnrestrictionofcones}).
\end{lemma}
\begin{proof}
	Let $T^{\ver} P_W = \ker(d\pi_{P_W})$ and $T^{\ver} (X \times \bP^{2\mu^2-1}) = X \times T\bP^{2\mu^2-1}$ be the vertical tangent bundles of $P_W$ and $X \times \bP^{2\mu^2-1}$ with respect to the natural projection maps to $X$.
	Let $\scrM \subset (X \times \bP^{2\mu^2-1})_{m,d_{P_W}}$ be the open subspace consisting of those curves $u$ that are lifts of ones $v$ from $X_{m,d}$ and with the additional property that non-constant components of $u$ are lifts of non-constant components of $v$.
	We have
	\begin{equation}
		T^{\ver} P_W|_W \cong \pi_{P_W}^*W|_W.
	\end{equation}
	Hence, $T^{\ver} P_W|_W$ is convex by Lemma \ref{lemmaconvexityofbundles}. Also, $T^{\ver} (X \times \bP^{2\mu^2-1})$ is convex.

	Hence, $W_{m,d}$ and $\scrM$
	are smooth over $X_{m,d}$.
	This implies that the natural inclusion map
	\begin{equation}
		W_{m,d} \hookrightarrow \scrM \subset (X \times \bP^{2\mu^2-1})_{m,d_{P_W}}
	\end{equation}
	is a regular immersion.
	Hence, $C^{P_W/X \times \bP^{2\mu^2-1}}_{m,d_{P_W}}|_W$ is equal to the normal bundle of $W_{m,d}$ inside $(X \times \bP^{2\mu^2-1})_{m,d_{P_W}}$.
	Since $\widetilde{Q}|_W$ is convex as well by Lemma \ref{lemmaconvexityofbundles}, we have that this normal bundle is also identified with $\widetilde{Q}_{m,d_{P_W}}|_W$ giving our result.
\end{proof}

\begin{lemma} \label{lemmaboundsforfirstchernclasses}
	For each irreducible curve $C$ in $X$ we have
	\begin{equation} \label{chernclassinequality1}
		c_1(V)(C) \leq \mu r \deg(C),
	\end{equation}
	\begin{equation} \label{chernclassinequality2}
		c_1(W)(C) \leq \mu r \deg(C)
	\end{equation}
	\begin{equation} \label{chernclassinequality3}
		c_1(Q)(C) \leq 4r\mu^2 \deg(C).
	\end{equation}
\end{lemma}
\begin{proof}
	We have a short exact sequence:
	\begin{equation} \label{eqnshortexactseuqnecforVW}
		0 \to V \to \calO_X(r)^{\mu} \to W \to 0.
	\end{equation}
	Hence,
	\begin{equation} \label{eqnsumofc1s}
		c_1(V)(C) + c_1(W)(C) = \mu r\deg(C).
	\end{equation}
	Since $V$ and $W$ are strongly convex by Lemma \ref{lemmaconvexityofbundles},
	we get that $c_1(V)(C) \geq 0$ and $c_1(W)(C) \geq 0$.
	Hence, by Equation \eqref{eqnsumofc1s}, we get
	Equations \eqref{chernclassinequality1} and \eqref{chernclassinequality2} hold.

	We have a short exact sequence $0 \to W \oplus \calO_X \to \calO_{\bP^N}(2r)^{2\mu^2} \to Q \to 0$ and so
	\begin{equation}\label{eqnsumofc1s2}
		c_1(W)(C) + c_1(Q)(C) = 4r\mu^2 \deg(C).
	\end{equation}
	Since $W$ and $Q$ are strongly convex by Lemma \ref{lemmaconvexityofbundles}, we have $c_1(W)(C) \geq 0$
	and $c_1(W)(Q) \geq 0$.
	Hence, by Equation \eqref{eqnsumofc1s2}, we get
	Equation \eqref{chernclassinequality3} holds.
\end{proof}

\begin{lemma} \label{lemmaWgloballygenerated}
	Suppose $N>0$.
	Let $D'_T$, $D'_{\widetilde{V}}$, $D'_W$, $D'_{\widetilde{Q}}$, $D$, $E$, $F$, $F'$ be very general divisors corresponding to the line bundles:
	\begin{equation}
		\begin{aligned}
			 & \calO_{\bP^N \times \bP}(r,1), \ \calO_{\bP^N \times \bP}(r,1), \ \calO_{\bP^N \times \bP}(2r,0), \ \calO_{\bP^N \times \bP}(r,1),                                             \\
			 & \calO_{\bP^N \times \bP}(72\mu^4,72\mu^4), \ \calO_{\bP^N \times \bP}(36\mu^4,36\mu^4), \ \calO_{\bP^N \times \bP}(12\mu^3,12\mu^3), \ \calO_{\bP^N \times \bP}(3\mu^3,3\mu^3)
		\end{aligned}
	\end{equation}
	respectively.
	Then
	\begin{enumerate}
		\item $T^* \otimes \calO(D'_T)$, $\widetilde{V}^* \otimes \calO(D'_{\widetilde{V}}|_{X \times \bP^{2\mu^2-1}})$, $W^* \otimes \calO(D'_W|_X)$, $\widetilde{Q}^* \otimes \calO(D'_{\widetilde{Q}}|_{P_W})$, $V^* \otimes \calO(F'|_X)$, $T^* \otimes \calO(F')$ and $W^* \otimes \calO(F'|_X)$ are globally generated,
		\item $D-D'_T$ is $\BL_{X \times \bP^{2\mu^2-1}}(\bP^N \times \bP^{2\mu^2-1})$-compatible and its tensor product with the ideal sheaf defining $X \times \bP^{2\mu^2-1}$ in $\bP^N \times \bP^{2\mu^2-1}$ is globally generated,
		\item $(D - D'_T - D'_{\widetilde{V}})|_{X \times \bP^{2\mu^2-1}}$ is $\BL_{P_W}(X \times \bP^{2\mu^2-1})$-compatible and also its tensor product with the ideal sheaf defining $P_W$ in $X \times \bP^{2\mu^2-1}$ is globally generated,
		\item the following divisors are nef:
		      \begin{equation} \label{eqndivisorstobenef}
			      \begin{aligned}
				       & D - D'_T, \quad D - D'_T - D'_{\widetilde{V}}, \quad D - E - D'_{\widetilde{Q}},                                                           \\
				       & E - D'_T - D'_{\widetilde{V}} - D'_{\widetilde{Q}}, \quad F - D'_T - D'_{\widetilde{V}} - D'_W, \quad E|_{P_W} - \pi_W^*(F|_X) - \infty_W,
			      \end{aligned}
		      \end{equation}
		\item $D|_{P_W}$ is $\widetilde{Q}$-large and
		\item $F|_X$ is $V, T, W$-large.
	\end{enumerate}
\end{lemma}
\begin{proof}
	We have $T^* \otimes \calO(D'_T) = \calO_{\bP^N \otimes \bP^{2\mu^2 -1}}(0,1)^\mu$ and so is globally generated.
	By Equation \eqref{eqnexactsequence}
	combined with the fact that $\widetilde{V}$ is the pullback of $V$ via the projection map to $X$, we have an exact sequence:
	\begin{equation} \label{eqnexactsequencesurjaection}
		\calO_{X \times \bP^{2\mu^2-1}}^{\mu} \to \widetilde{V}^* \otimes \calO_{X \times \bP^{2\mu^2-1}}(r,0) \to 0.
	\end{equation}
	Hence, $\widetilde{V}^* \otimes \calO(D'_{\widetilde{V}}|_{X \times \bP^{2\mu^2-1}})$ is globally generated.
	Since $W$ is the cokernel of the map $e_V$ from \eqref{lemmaexactsequence}, we get an exact sequence:
	\begin{equation} \label{eqnwglobalygenerated}
		\calO_X^{\mu^2} \to W^* \otimes \calO_X(2r) \to 0
	\end{equation}
	and hence $W^* \otimes \calO(D'_W|_X)$ is globally generated since $D'_W|_X = \calO_X(2r)$ by Lemma \ref{lemmapullbackbundles}.

	By \eqref{eqnexateeuaenfeaiet}, we have an exact sequence:
	\begin{equation} \label{eqnfirstWexactsequence}
		0 \to Q \to \calO_X(3r)^{2\mu^3}.
	\end{equation}
	Taking duals and tensoring with $\calO_X(3r)$ gives an exact sequence:
	\begin{equation} \label{eqnfirstWexactsequencedual}
		\calO_X^{2\mu^3} \to Q^* \otimes \calO_X(3r) \to 0.
	\end{equation}
	Hence, we have an exact sequence:
	\begin{equation}  \label{eqnsecondWexactsequence}
		\calO_{P_W}^{2\mu^3} \to \pi_{P_W}^*Q^* \otimes \calO_{P_W}(3r,0) \to 0.
	\end{equation}
	By \eqref{eqnequationformalbundle} we get:
	\begin{equation} \label{eqnwidetildeWdual}
		\widetilde{Q}^* \otimes \calO_{P_W}(-2r,1) = \pi_{P_W}^*Q^*.
	\end{equation}
	Combining \eqref{eqnsecondWexactsequence} with \eqref{eqnwidetildeWdual},
	we get the exact sequence
	\begin{equation} \label{eqnaefaefa}
		\calO_{P_W}^{2\mu^3} \to \widetilde{Q}^* \otimes \calO_{P_W}(r,1) \to 0.
	\end{equation}
	Hence, $\widetilde{Q}^* \otimes \calO(D'_{\widetilde{Q}}|_{P_W})$ is globally generated.

	Since $N > 0$, we have
	\begin{equation} \label{eqnmuinequality}
		\mu \geq \max(N,r).
	\end{equation}
	Equations \eqref{eqnmuinequality} and \eqref{eqnexactsequence} and Lemma \ref{lemmapullbackbundles} tell us that $V^* \otimes \calO(F'|_X)$ is globally generated.
	We have $T^* \otimes \calO(F'|_X) = \calO_X(4\mu^3-r)^\mu$ is globally generated by Equation \eqref{eqnmuinequality} too.
	Equations \eqref{eqnmuinequality} and \eqref{eqnwglobalygenerated} tell us that $W^* \otimes \calO(F'|_X)$ is globally generated.

	Because $\frI_X \otimes \calO_X(r)$ is globally generated by Proposition \ref{propositioncastelnuovothm}
	we get that $\calO_X(r)$ is $\BL_X \bP^N$-compatible by Lemma \ref{lemmanicesufficientcondtionforblxycompatible}.
	Hence, $D-D'_T$ is $\BL_{X \times \bP^{2\mu^2-1}}(\bP^N \times \bP^{2\mu^2-1})$-compatible and also this divisor and its tensor product with the ideal sheaf defining $X \times \bP^{2\mu^2-1}$ in $\bP^N \times \bP^{2\mu^2-1}$ is globally generated.

	Let $\frI_{P_W}$ be the ideal sheaf defining $P_W$ inside $X \times \bP^{2\mu^2-1}$.
	Let $\scrF = \frI_{P_W} \otimes \calO_{X \times \bP^{2\mu^2-1}}(r,1)$.
	Since $\scrF$ is the kernel of the natural map
	\begin{equation}
		\calO_{\bP(W(-3r) \oplus \calO_X(-3r))}(1) \to \calO_{\bP(\calO_X^{2\mu^2}(-r))}(1),
	\end{equation}
	induced by $e_W \otimes \id_{\calO_X(-3r)}$, we get
	that
	\begin{equation} \label{eqnforpushforward2}
		(\pi_X)_*(\scrF) = \ker(e_W \otimes \id_{\calO_X(-3r)}) = Q^* \otimes \calO_X(3r)
	\end{equation}
	since pushforward is left exact.
	Therefore, since $Q^* \otimes \calO_X(3r)$ is globally generated by \eqref{eqnfirstWexactsequencedual},
	we get by \eqref{eqnforpushforward2} that
	$(\pi_X)_*(\scrF)$ is globally generated.
	Also,
	\begin{equation} \label{eqnsurjective}
		\pi_X^* (\pi_X)_* \scrF \to \scrF
	\end{equation}
	is surjective since it is surjective at the level of fibers and by Nakayama's lemma.
	Combining the previous two facts
	gives us
	that $\scrF$ is globally generated.
	Hence, $$(D - D'_T - D'_{\widetilde{V}})|_{X \times \bP^{2\mu^2-1}}$$ is $\BL_{P_W}(X \times \bP^{2\mu^2-1})$-compatible
	by Lemma \ref{lemmanicesufficientcondtionforblxycompatible}. Also, its tensor product with $\frI_{P_W}$ is globally generated.

	We have $\calO(\infty_W) = \calO_{P_W}(-2r,1)$ by Lemma \ref{lemmapullbackbundles}.
	We have $F|_X = \calO_X(12\mu^3 + 24 \mu^3 r)$ by Lemma \ref{lemmapullbackbundles} too.
	Since $\pi_{P_W} = \pi_X \circ \iota_{P_W}$, this implies $\pi_{P_W}^*(F|_X) = \calO_{P_W}(12\mu^3 + 24 \mu^3 r,0)$.
	Combining this with Equation \eqref{eqnmuinequality}, we get that the divisors in \eqref{eqndivisorstobenef} are all nef.

	Let $C$ be an irreducible curve in $X$ whose normalization has genus zero.
	Since $N>0$ and by \eqref{eqnmuinequality}, we have
	\begin{equation}
		8\mu^3 \geq 1+2(2(N+1)\mu^2-1).
	\end{equation}
	Hence, by Lemma \ref{lemmaboundsforfirstchernclasses}, we have
	\begin{equation} \label{eqnc1ineqforV}
		(12\mu^3-1-2(2(N+1)\mu^2-1))\deg(C) \geq 2 + c_1(V)(C),
	\end{equation}
	\begin{equation} \label{eqnfirstcaeca}
		(12\mu^3-1 - 2(2(N+1)\mu^2-1)) \deg(C) \geq 2 + c_1(W)(C)
	\end{equation}
	and
	\begin{equation} \label{eqnfirstcaeca2}
		(72\mu^3-1 - 2(2(N+1)\mu^2-1)) \deg(C) \geq 2 + c_1(Q)(C).
	\end{equation}
	Therefore, by Equation \eqref{eqnequationformalbundle}
	\begin{equation} \label{eqnfirstcaeca3}
		(72\mu^3-1 - 2(2(N+1)\mu^2-1)) \deg(C) \geq 2 + c_1(\widetilde{Q})(C).
	\end{equation}
	Also,
	\begin{equation} \label{eqnc1ineqforT}
		(12\mu^3-1-2(2(N+1)\mu^2-1))\deg(C) \geq 2 + r \deg(C) = 2 + c_1(T)(C).
	\end{equation}

	Let $\phi : \bP^N \times \bP^{2\mu^2-1} \hookrightarrow \bP^{2(N+1)\mu^2-1}$
	be the Segre embedding.
	Since $D$ is the restriction of a very general divisor in $\calO_{\bP^{2(N+1)\mu^2-1}}(72\mu^4)$,
	we have by Equation \eqref{eqnfirstcaeca3} combined with Lemma \ref{lemmavlargeexist},
	that $D|_{P_W}$ is $\widetilde{Q}$-large.

	We have $F|_X$ is $V,T,W$-large by Equations \eqref{eqnc1ineqforV},\eqref{eqnc1ineqforT} and \eqref{eqnfirstcaeca}
	combined with Lemmas \ref{lemmaconvexityofbundles} and \ref{lemmavlargeexist}.

\end{proof}

\begin{lemma} \label{lemmamainargumentstuffesf}
	Let $d \in H_2^+(X)-0$.
	Suppose that there exists $C>0$ so that
	\begin{equation} \label{eqnaefaef20}
		\left|\langle [\alpha_1],\cdots,[\alpha_m] \rangle^X_{m,d} \right| \leq C \prod_{j=1}^m |\alpha_i|_{C^0,\omega}
	\end{equation}
	for each $m \leq 2((N+1)\deg(d)+ N-3)$ and each collection of closed homogenous differential forms $\alpha_1,\cdots,\alpha_m$ on $X$.
	Then
	\begin{equation} \label{eqnaefaef2}
		\left|\langle [\alpha_1],\cdots,[\alpha_m] \rangle^X_{m,d} \right| \leq m! C e^{\omega \cdot d} \prod_{j=1}^m |\alpha_i|_{C^0,\omega}
	\end{equation}
	for each $m \in \bN$ and each collection of closed homogenous differential forms $\alpha_1,\cdots,\alpha_m$ on $X$.
\end{lemma}
\begin{proof}
	Let $m \in \bN$
	and $\alpha_1,\cdots,\alpha_m$ a collection of closed homogenous differential forms.
	Let $\delta := 2((N+1)\deg(d)+ N-3)$.
	If the degree of $\alpha_j$ is equal to one or zero for some $j \in \{1,\cdots,m\}$ then
	\begin{equation} \label{eqnaefaegt494ggag}
		\langle [\alpha_1],\cdots,[\alpha_m] \rangle^X_{m,d}
	\end{equation}
	vanishes
	and so our inequality is satisfied.
	As a result we can assume that the degree of $\alpha_j$ is $\geq 2$ for each $j = 1,\cdots, m$.
	After permuting the $\alpha_j$'s we can assume that there exists $k \in \{1,\cdots,m\}$ so that
	the degree of $\alpha_j$ is equal to $2$ for each $j=1,\cdots,k$ and is greater than $2$ for each $j \in \{k+1,\cdots,m\}$.
	Also, since the dimension of $X_{m,d}$ is bounded above by $\delta/2 + m$ by Lemma \ref{lemmadimensionofX},
	we get \eqref{eqnaefaegt494ggag} vanishes if $3(m-k) + 2k > \delta + 2m$.
	Hence, we can assume that $m - k \leq \delta$.
	We have
	\begin{equation}
		\left|\langle [\alpha_1],\cdots,[\alpha_m] \rangle_{m,d}^X \right| =
	\end{equation}
	\begin{equation}
		\left|\langle [\alpha_{k+1}],\cdots,[\alpha_m] \rangle_{m,d}^X \prod_{j=1}^k \alpha_j \cdot d \right|
	\end{equation}
	\begin{equation}
		\leq \left|\langle [\alpha_{k+1}],\cdots,[\alpha_m] \rangle_{m,d}^X  \right| \prod_{j=1}^k (|\alpha_j|_{C^0,\omega} \omega \cdot d)
	\end{equation}
	\begin{equation}
		= \left|\langle [\alpha_{k+1}],\cdots,[\alpha_m] \rangle_{m,d}^X   \right| k!\frac{(\omega \cdot d)^k}{k!} \prod_{j=1}^k |\alpha_j|_{C^0,\omega}
	\end{equation}
	\begin{equation}\label{eqnafa0elaster}
		\leq \left|\langle [\alpha_{k+1}],\cdots,[\alpha_m] \rangle_{m,d}^X   \right| m! e^{\omega \cdot d} \prod_{j=1}^k |\alpha_j|_{C^0,\omega}.
	\end{equation}
	Since $m-k \leq \delta$ we have
	\begin{equation}
		\left|\langle [\alpha_{k+1}],\cdots,[\alpha_m] \rangle_{m,d}^X \right| \leq C \prod_{j={k+1}}^m |\alpha_i|_{C^0,\omega}
	\end{equation}
	by assumption, and hence \eqref{eqnafa0elaster} is less than or equal to:
	\begin{equation}
		C \prod_{j=k+1}^m |\alpha_j|_{C^0,\omega} m! e^{\omega \cdot d} \prod_{j=1}^k |\alpha_j|_{C^0,\omega}
	\end{equation}
	\begin{equation}
		= m! C e^{\omega \cdot d} \prod_{j=1}^m |\alpha_j|_{C^0,\omega}.
	\end{equation}
\end{proof}

\begin{theorem} \label{theoremproperestimates}
	Let $r>0$ be so that $\frI_X$ and $\frI_X^2$ are $r$-regular.
	Let $m \in \bN$, $d \in H_2^+(X)$ and $\alpha_1,\cdots,\alpha_m$ closed homogenous differential forms on $X$.
	Also, let $\deg(d) \in \bN$
	be the degree of the image of $d$ in $H_2^+(\bP^N) \cong \bN$.
	Then
	\begin{equation}
		\left|\langle [\alpha_1],\cdots,[\alpha_m] \rangle_{m,d} \right| \leq
		m!\left(4952{N+r \choose N}^6(\deg(d)+3)\right)!
		\prod_{j=1}^m |\alpha_j|_{C^0,\omega}.
	\end{equation}
\end{theorem}
\begin{proof}
	We can assume that $d \neq 0$ since the $d=0$ case is straightforward.
	This also implies $N>0$ too.

	We write $d_{\bP^N \times \bP} = (d_1,d_2) \in \bN^2 \cong H_2^+(\bP^N \times \bP^{2\mu^2-1})$.
	We have that $d_1$ is the degree of the image of $d$ in $\bP^N$ and $d_2 = 2rd_1$ by Lemma \ref{lemmaconvexityofbundles}.
	Let $T$ be equal to $\calO_{\bP^N \times \bP^{2\mu^2-1}}(r,0)^{\mu}$.
	Let $D$, $E$ and $F$ be as in Lemma \ref{lemmaWgloballygenerated}.
	By Proposition \ref{propdvolumeboundsproductsfo} we have:
	\begin{equation} \label{eqnbounedsed2222}
		\begin{aligned}
			D\Vol((T^k)_{m,(d_1,d_2)}) \leq                                                 & \\
			(k\mu+(52+3r)(N+2\mu^2-1+3)^2(d_1+d_2+1) + (N+2\mu^2-1+2 \times 72\mu^4 +30)m)! &
			/ (rk\mu d_1 + k\mu)!
		\end{aligned}
	\end{equation}
	\begin{equation}
		=(k\mu+(52+3r)(N+2\mu^2+2)^2((1+2r)d_1+1) + (N+2\mu^2+144\mu^4+29)m)!
		/ (rk\mu d_1 + k\mu)!
	\end{equation}
	\begin{equation}
		\begin{aligned}
			\overset{\textnormal{\eqref{eqnmuinequality}}}{\leq}
			(k\mu+(52+3\mu)(\mu+2\mu^2+2)^2((1+2\mu)d_1+1) + \\ (\mu+2\mu^2+144\mu^4+29)m)!
			/ (rk\mu d_1 + k\mu)!
		\end{aligned}
	\end{equation}
	\begin{equation}
		\begin{aligned}
			\leq
			(k\mu+(52\mu+3\mu)(\mu^2+2\mu^2+2\mu^2)^2((\mu+2\mu)d_1+3\mu) + \\ (\mu^4+2\mu^4+144\mu^4+29\mu^4)m)!
			/ (rk\mu d_1 + k\mu)!
		\end{aligned}
	\end{equation}
	\begin{equation}
		\begin{aligned}
			\leq
			(k\mu+(55\mu)(5\mu^2)^2(3\mu)(d_1+1)) + 144\mu^4m)!
			/ (rk\mu d_1 + k\mu)!
		\end{aligned}
	\end{equation}
	\begin{equation}
		\begin{aligned}
			\leq
			((k\mu+ 4125\mu^6(d_1+1)) + 144\mu^4m)!
			/ (rk\mu d_1 + k\mu)!
		\end{aligned}
	\end{equation}
	and so
	\begin{equation} \label{eqnbounedsed222}
		\begin{aligned}
			D\Vol((T^k)_{m,(d_1,d_2)}) \leq
			\eta_k :=
			(k\mu+ 4125\mu^6(d_1+1) + 144\mu^4m)!
			/ (rk\mu d_1 + k\mu)!.
		\end{aligned}
	\end{equation}

	Let \begin{equation}
		\overline{C}_{m,d} = C^{X \times \bP^{2\mu^2-1}/\bP^N \times \bP^{2\mu^2-1}}_{m,d_{X \times \bP}} \subset \widetilde{V}_{m,d_{X \times \bP}}
	\end{equation} be as in Definition \ref{defnsystemofnormalcones}.
	By Proposition \ref{propositionboundsfornormalcones} combined with Lemma \ref{lemmaWgloballygenerated},
	we have
	\begin{equation}
		D\Vol((T^k)_{m,(d_1,d_2)}) \geq (D|_{X \times \bP^{2\mu^2-1}})\Vol((T^k|_{X \times \bP^{2\mu^2-1}})_{m,d_{X \times \bP}},\overline{C}_{m,d})
	\end{equation}
	and hence by \eqref{eqnbounedsed222} we have
	\begin{equation} \label{eqnaefaef04404040}
		(D|_{X \times \bP^{2\mu^2-1}})\Vol((T^k|_{X \times \bP^{2\mu^2-1}})_{m,d_{X \times \bP}},\overline{C}_{m,d}) \leq \eta_k.
	\end{equation}
	By Proposition \ref{propositionboundsfornormalcones} combined with Lemma \ref{lemmaWgloballygenerated},
	we have
	\begin{equation} \label{eqnaefaef044040402333prev}
		(D|_{X \times \bP^{2\mu^2-1}})\Vol((T^k)_{m,d}|_{X \times \bP^{2\mu^2-1}},\overline{C}_{m,d}) \geq (D|_{P_W})\Vol((T^k)_{m,d}|_W,\overline{C}_{m,d}|_W,C^{P_W/X \times \bP^{2\mu^2-1}}_{m,d_{P_W}}|_W)
	\end{equation}
	(Definition \ref{defnrestrictionofcones}).

	Also, by Lemma \ref{lemmanormalconeforW0hd}, we have that
	$C^{P_W/X \times \bP^{2\mu^2-1}}_{m,d_{P_W}}|_W$ is
	$\widetilde{Q}_{m,d_{P_W}}|_W$.
	Hence, by Proposition \ref{propboundsforsummands} and
	Lemmas \ref{lemmaconvexityofbundles} and \ref{lemmaWgloballygenerated}, we have
	\begin{equation} \label{eqnaefaef044040402333}
		(E|_{P_W})\Vol((T^k)_{m,d}|_W,\overline{C}_{m,d}|_W) \leq g_d!(D|_{P_W})\Vol((T^k)_{m,d}|_W,\overline{C}_{m,d}|_W, C^{P_W/X \times \bP^{2\mu^2-1}}_{m,d_{P_W}}|_W)
	\end{equation}
	where
	\begin{equation} \label{eqndefiningofgd}
		\begin{aligned}
			g_d =
			c_1(\widetilde{Q})\cdot d_{P_W} + \dim(\widetilde{Q}) \\
			\overset{\textnormal{Lemma \ref{lemmaconvexityofbundles}}}{=}
			c_1(Q)(d) + \dim(Q).
		\end{aligned}
	\end{equation}
	Combining equations \eqref{eqnaefaef044040402333prev}, \eqref{eqnaefaef044040402333} and \eqref{eqnaefaef04404040}, gives
	\begin{equation} \label{eqnaevae9e4e9g4}
		(E|_{P_W})\Vol((T^k)_{m,d}|_W,\overline{C}_{m,d}|_W) \leq g_d! \eta_k.
	\end{equation}
	By Proposition \ref{propcurvesincompatificationsofconvexvecctorbundles} combined with Lemmas \ref{lemmaWgloballygenerated} and \ref{lemmapullbacks}
	we have
	\begin{equation}
		(E|_{P_W})\Vol((T^k|_W)_{m,d_{P_W}},\overline{C}_{m,d}|_W) \geq (F|_X)\Vol((T^k|_X)_{m,d},\overline{C}_{m,d}|_X,W_{m,d})
	\end{equation}
	and hence
	\begin{equation} \label{eqnefe852}
		(F|_X)\Vol((T^k|_X)_{m,d},\overline{C}_{m,d}|_X,W_{m,d}) \leq g_d! \eta_k.
	\end{equation}

	By Lemma \ref{lemmaboundsforfirstchernclasses} combined with the fact that
	\begin{equation}
		\dim(Q) \leq \dim(\calO_X(2r)^{2\mu^2}) = 2\mu^2
	\end{equation}
	and by Equation \eqref{eqndefiningofgd}
	we get
	\begin{equation}
		g_d \leq
		4r\mu^2d_1  + 2\mu^2
	\end{equation}
	\begin{equation}
		\overset{\textnormal{\eqref{eqnmuinequality}}}{\leq} 4\mu^3(d_1+1).
	\end{equation}
	Hence, by Equation \eqref{eqnefe852}, we get
	\begin{equation} \label{eqnefe853}
		(F|_X)\Vol((T^k|_X)_{m,d},\overline{C}_{m,d}|_X,W_{m,d}) \leq (4\mu^3(d_1+1))! \eta_k
	\end{equation}
	\begin{equation}
		\overset{\textnormal{\eqref{eqnreallreallyterribleinequality}}}{\leq} (4\mu^3(d_1+1)+k\mu+ 4125\mu^6(d_1+1) + 144\mu^4m)!
		/ (rk\mu d_1 + k\mu)!
	\end{equation}
	\begin{equation}
		\leq (k\mu+ 4129\mu^6(d_1+1) + 144\mu^4m)!
		/ (rk\mu d_1 + k\mu)!.
	\end{equation}
	Hence,
	\begin{equation} \label{eqnefe8541}
		(F|_X)\Vol((T^k|_X)_{m,d},\overline{C}_{m,d}|_X,W_{m,d}) \leq (k\mu+ 4129\mu^6(d_1+1) + 144\mu^4m)!
		/ (rk\mu d_1 + k\mu)!.
	\end{equation}
	Since $\overline{C}_{m,d}|_X = C^{X/\bP^N}_{m,d}$ by Lemma \ref{lemmapullbacks}, we have
	\begin{equation} \label{eqnefe854}
		(F|_X)\Vol((T^k|_X)_{m,d}, C^{X/\bP^N}_{m,d}, W_{m,d}) \leq (k\mu+ 4129\mu^6(d_1+1) + 144\mu^4m)!
		/ (rk\mu d_1 + k\mu)!.
	\end{equation}

	Let $r_d$ be as in Equation \eqref{eqnaefaefforrd} and $\rho_{d,T}$ is as in Equation \eqref{eqnrhodT}.
	We have $c_1(T)d = r\mu d_1$ and $\dim(T) = \mu$ and so
	\begin{equation} \label{eqnrehodty}
		\rho_{d,T} = r\mu d_1 + \mu + 2 \overset{\textnormal{\eqref{eqnmuinequality}}}{\leq} \mu^2 d_1 + \mu + 2\mu \leq \mu^2(d_1+3).
	\end{equation}

	Let $\pi_X : \bP^N \times \bP^{2\mu^2-1} \to \bP^N$ be the natural projection map as in Definition \ref{defnprojectionmapandex}.
	Choose a K\"{a}hler form $\omega'$ on $\bP^N \times \bP^{2\mu^2-1}$ Poincar\'{e} dual to $\calO_{\bP^N \times \bP}(3\mu^3-1,3\mu^3)$.
	Then $\omega + \omega'|_X$ is Poincar\'{e} dual to $\frac{1}{4}F|_X = F'|_X$ where $F'$ is as in Lemma \ref{lemmaWgloballygenerated}.
	Therefore, since we have a short exact sequence of convex vector bundles
	\begin{equation}
		0 \to V \to T|_X \to W \to 0,
	\end{equation}
	we get by Lemmas \ref{lemmavolumeboundsimplyGWbounds}
	and
	\ref{lemmaWgloballygenerated}
	that
	\begin{equation}
		\left| \langle [\alpha_1],\cdots,[\alpha_m] \rangle_{m,d}^X \right| \leq
	\end{equation}
	\begin{equation}
		\begin{aligned}
			\sum_{\ell = 0}^{r_d} \sum_{k=0}^{r_d-\ell}  {k+\ell \choose \ell} (23r_d + 2 m N + 18m +\rho_{d,T}+23)!(\ell\rho_{d,T})! \\
			(F|_X)\Vol(C^{X/\bP^N}_{m,d}, T_{m,d}^\ell, W_{m,d})
			\prod_{j=1}^m |\alpha_j|_{C^0,\omega+\omega'|_X}
		\end{aligned}
	\end{equation}
	\begin{equation}
		\begin{aligned}
			\overset{\textnormal{Lemma \ref{lemmaaboutpositivedefiniteforms}}}{\leq}
			\sum_{\ell = 0}^{r_d} \sum_{k=0}^{r_d-\ell}  {k+\ell \choose \ell} (23r_d + 2 m N + 18m +\rho_{d,T}+23)!(\ell\rho_{d,T})! \\
			(F|_X)\Vol(C^{X/\bP^N}_{m,d}, T_{m,d}^\ell, W_{m,d})
			\prod_{j=1}^m |\alpha_j|_{C^0,\omega}
		\end{aligned}
	\end{equation}
	\begin{equation}
		\begin{aligned}
			\overset{\textnormal{\eqref{eqnefe854}}}{\leq}
			\sum_{\ell = 0}^{r_d} \sum_{k=0}^{r_d-\ell}  {k+\ell \choose \ell} (23r_d + 2 m N + 18m +\rho_{d,T}+23)!(\ell\rho_{d,T})! \\
			(\ell\mu+ 4129\mu^6(d_1+1) + 144\mu^4m)!
			/ (rl\mu d_1 + \ell\mu)!
			\prod_{j=1}^m |\alpha_j|_{C^0,\omega}
		\end{aligned}
	\end{equation}
	\begin{equation}
		\begin{aligned}
			\overset{\textnormal{\eqref{eqnrehodty}}}{=}
			\sum_{\ell = 0}^{r_d} \sum_{k=0}^{r_d-\ell}  {k+\ell \choose \ell} (23r_d + 2 m N + 18m+\mu^2(d_1+3)+23)!(lr\mu d_1 + \ell\mu+2\ell)! \\
			(\ell\mu+ 4129\mu^6(d_1+1) + 144\mu^4m)!
			/ (rl\mu d_1 + \ell\mu)!
			\prod_{j=1}^m |\alpha_j|_{C^0,\omega}
		\end{aligned}
	\end{equation}
	\begin{equation}
		\begin{aligned}
			\leq
			\sum_{\ell = 0}^{r_d} \sum_{k=0}^{r_d-\ell}  {k+\ell \choose \ell} (23r_d + 2 m N + 18m+\mu^2(d_1+3)+23)!
			(\ell\mu+ 4129\mu^6(d_1+1) + 144\mu^4m)! \\  \times (lrd_1+\ell\mu+2\ell)^{2\ell}
			\prod_{j=1}^m |\alpha_j|_{C^0,\omega}
		\end{aligned}
	\end{equation}
	\begin{equation}
		\begin{aligned}
			\leq
			(r_d+1)^2 2^{2r_d} (23r_d + 2 m N + 18m+\mu^2(d_1+3)+23)! \\ \times
			(r_d\mu+ 4129\mu^6(d_1+1) + 144\mu^4m)! (r_dr d_1+r_d\mu+2r_d)^{2r_d}
			\prod_{j=1}^m |\alpha_j|_{C^0,\omega}
		\end{aligned}
	\end{equation}
	\begin{equation}
		\begin{aligned}
			\leq
			(r_d+1)^2 2^{2r_d} (23r_d + 2 m N + 18m+\mu^2(d_1+3)+23)! \\ \times
			(r_d\mu+ 4129\mu^6(d_1+1) + 144\mu^4m)! (r_d\mu d_1+r_d\mu+2r_d)^{2r_d}
			\prod_{j=1}^m |\alpha_j|_{C^0,\omega}
		\end{aligned}
	\end{equation}
	\begin{equation} 		\label{eqnaefef94f4}
		\begin{aligned}
			\leq
			(r_d+1)^2 2^{2r_d} (23r_d + 2 m N + 18m+\mu^2(d_1+3)+23)! \\ \times
			(r_d\mu+ 4129\mu^6(d_1+1) + 144\mu^4m)! (r_d\mu(d_1+3))^{2r_d}
			\prod_{j=1}^m |\alpha_j|_{C^0,\omega}
		\end{aligned}
	\end{equation}

	Now by Equation \eqref{eqnaefaefforrd},
	\begin{equation}
		r_d = (N+1)d_1 + N + c_1(V)(d) + c_1(W)(d) + \dim(V) + \dim(W)
	\end{equation}
	\begin{equation} \label{eqnaef40040a49999333}
		\overset{\textnormal{Lemma \ref{lemmaboundsforfirstchernclasses}}}{\leq}
		(N+1)d_1 + N + \mu r d_1 + \mu r d_1 + \dim(V)+\dim(W).
	\end{equation}
	Now
	\begin{equation}
		\dim(V)+\dim(W) \leq \dim(\calO_X(r)^{\mu}) = \mu.
	\end{equation}
	Hence, \eqref{eqnaef40040a49999333} is less than or equal to:
	\begin{equation}
		(N+1)d_1 + N + \mu r d_1 + \mu r d_1 +  \mu
	\end{equation}
	\begin{equation}
		=((N+1)+2\mu r)d_1 + N + \mu
	\end{equation}
	\begin{equation}
		\overset{\textnormal{\eqref{eqnmuinequality}}}{\leq}
		((\mu+1)+2\mu^2)d_1 + \mu + \mu
	\end{equation}
	\begin{equation}
		\leq 4\mu^2 d_1 + 2\mu^2
	\end{equation}
	\begin{equation}
		\leq 4\mu^2(d_1+1).
	\end{equation}
	Combining this with \eqref{eqnrehodty} we get \eqref{eqnaefef94f4} is less than or equal to:
	\begin{equation}
		\begin{aligned}
			\leq
			(4\mu^2(d_1+1)+1)^2 2^{2(4\mu^2(d_1+1))} (23 \times 4\mu^2(d_1+1) + 2 m N + 18m +\mu^2(d_1+3)+23)! \\
			\times ((4\mu^2(d_1+1))\mu+ 4129\mu^6(d_1+1) + 144\mu^4m)! ((4\mu^2(d_1+1))\mu(d_1+3))^{8\mu^2(d_1+1)}
			\prod_{j=1}^m |\alpha_j|_{C^0,\omega}
		\end{aligned}
	\end{equation}
	\begin{equation}
		\begin{aligned}
			\overset{\textnormal{\eqref{eqnmuinequality}}}{\leq}
			(4\mu^2(d_1+1)+1)^2 2^{8\mu^2(d_1+1)} (92\mu^2(d_1+1) + 2m\mu + 18m + 3\mu^2(d_1+1) +23)! \\
			\times ((4\mu^2(d_1+1))\mu+ 4129\mu^6(d_1+1) + 144\mu^4m)! ((12\mu^3(d_1+1)^2))^{8\mu^2(d_1+1)}
			\prod_{j=1}^m |\alpha_j|_{C^0,\omega}
		\end{aligned}
	\end{equation}
	\begin{equation}
		\begin{aligned}
			\leq
			(4\mu^2(d_1+1)+1)^2 2^{8\mu^2(d_1+1)} (92\mu^2(d_1+1) + 2m\mu + 18m + 3\mu^2(d_1+1) +23)! \\
			\times ((4\mu^3(d_1+1))+ 4129\mu^6(d_1+1) + 144\mu^4m)! ((4\mu^2(d_1+1)))^{16\mu^2(d_1+1)}
			\prod_{j=1}^m |\alpha_j|_{C^0,\omega}
		\end{aligned}
	\end{equation}
	\begin{equation}
		\begin{aligned}
			\overset{\textnormal{\eqref{eqnreallreallyterribleinequality}}}{\leq}
			(4\mu^2(d_1+1)+1+2+2+8\mu^2(d_1+1)+92\mu^2(d_1+1) + 2m\mu + 18m + 3\mu^2(d_1+1) + 23 + \\
			4\mu^3(d_1+1)+ 4129\mu^6(d_1+1) + 144\mu^4m + 4\mu^2(d_1+1)+16\mu^2(d_1+1))!
			\prod_{j=1}^m |\alpha_j|_{C^0,\omega}
		\end{aligned}
	\end{equation}
	\begin{equation}
		\begin{aligned}
			=
			(28+2m\mu+131\mu^2(d_1+1)+ 18m + \\
			4\mu^3(d_1+1)+ 4129\mu^6(d_1+1) + 144\mu^4m)!
			\prod_{j=1}^m |\alpha_j|_{C^0,\omega}
		\end{aligned}
	\end{equation}
	\begin{equation}
		\begin{aligned}
			\leq
			(4292\mu^6(d_1+1)+ 164\mu^4m)!
			\prod_{j=1}^m |\alpha_j|_{C^0,\omega}.
		\end{aligned}
	\end{equation}

	So, putting everything together so far we get:
	\begin{equation} \label{eqn3tr943t049}
		\begin{aligned}
			 & \left| \langle [\alpha_1],\cdots,[\alpha_m] \rangle_{m,d}^X \right| \leq
			(4292\mu^6(d_1+1)+ 164\mu^4m)!
			\prod_{j=1}^m |\alpha_j|_{C^0,\omega}.
		\end{aligned}
	\end{equation}
	Therefore, if
	\begin{equation}
		m \leq 2((N+1)d_1+ N-3),
	\end{equation}
	we get
	\begin{equation} \label{eqn3tr943t0493}
		\left| \langle [\alpha_1],\cdots,[\alpha_m] \rangle_{m,d}^X \right|
		\leq
	\end{equation}
	\begin{equation}
		\begin{aligned}
			(4292\mu^6(d_1+1)+ 164\mu^4(2((N+1)d_1+ N-3)))!
			\prod_{j=1}^m |\alpha_j|_{C^0,\omega}
		\end{aligned}
	\end{equation}
	\begin{equation}
		\begin{aligned}
			\overset{\textnormal{\eqref{eqnmuinequality}}}{\leq}
			(4292\mu^6(d_1+1)+ 164\mu^4(2((\mu+1)d_1+ \mu-3)))!
			\prod_{j=1}^m |\alpha_j|_{C^0,\omega}
		\end{aligned}
	\end{equation}
	\begin{equation}
		\begin{aligned}
			\leq
			(4292\mu^6(d_1+1)+ 164\mu^4(2((\mu+\mu)d_1+ \mu)))!
			\prod_{j=1}^m |\alpha_j|_{C^0,\omega}
		\end{aligned}
	\end{equation}
	\begin{equation}
		\begin{aligned}
			\leq
			(4292\mu^6(d_1+1)+ 164\mu^4(4\mu d_1+ 4\mu))!
			\prod_{j=1}^m |\alpha_j|_{C^0,\omega}
		\end{aligned}
	\end{equation}
	\begin{equation}
		\leq
		(4948\mu^6(d_1+1))!
		\prod_{j=1}^m |\alpha_j|_{C^0,\omega}.
	\end{equation}

	Combining this with Lemma \ref{lemmamainargumentstuffesf},
	we get:
	\begin{equation} \label{eqn3tr943t049444}
		\left| \langle [\alpha_1],\cdots,[\alpha_m] \rangle_{m,d}^X \right|
		\leq m!e^{\omega \cdot d}(4948\mu^6(d_1+1))!
		\prod_{j=1}^m |\alpha_j|_{C^0,\omega}
	\end{equation}
	for each $m \in \bN$.
	Hence, for each $m \in \bN$,
	\begin{equation} \label{eqn3tr943t0494444}
		\left| \langle [\alpha_1],\cdots,[\alpha_m] \rangle_{m,d}^X \right|
		\leq m!4^{d_1}(4948\mu^6(d_1+1))!
		\prod_{j=1}^m |\alpha_j|_{C^0,\omega}
	\end{equation}
	\begin{equation}
		\overset{\textnormal{\eqref{eqnreallreallyterribleinequality}}}{\leq}
		m!(4+d_1+4948\mu^6(d_1+1))!
		\prod_{j=1}^m |\alpha_j|_{C^0,\omega}
	\end{equation}
	\begin{equation}
		\leq
		m!(4952\mu^6(d_1+1))!
		\prod_{j=1}^m |\alpha_j|_{C^0,\omega}
	\end{equation}
	\begin{equation}
		\overset{\textnormal{\eqref{eqndefnofmu}}}{=}
		m!\left(4952{N+r \choose N}^6(d_1+1)\right)!
		\prod_{j=1}^m |\alpha_j|_{C^0,\omega}.
	\end{equation}
\end{proof}

We will now prove the main theorem \ref{maintheorem}.
Let us recall its statement:
\begin{theorem*}
	Suppose $X \subset \bP^N$ is cut out by polynomials of degree at most $\delta_X \in \bN$.
	Let $d \in H_2(X;\bZ)$ be a curve class and let $\deg(d) \in \bN$ be its image in $H_2(\bP^N;\bZ) \cong \bZ$.
	Then for each tuple $\alpha_1,\cdots,\alpha_m$ of homogenous closed differential forms on $X$ we have
	\begin{equation} \label{eqnmaintheoremineqluaty1}
		\left|\langle [\alpha_1],\cdots,[\alpha_m] \rangle^X_{m,d} \right| \leq
		m!\left(4952{N+(4\delta_X)^{2^{N-1}}+1 \choose N}^6(\deg(d)+1)\right)!
		\prod_{j=1}^m |\alpha_j|_{C^0,\omega}
	\end{equation}
	where $|\alpha_j|_{C^0,\omega}$ is the $C^0$ norm of $\alpha_j$ with respect to the induced Fubini Study metric on $X$ for each $j$.
\end{theorem*}
\begin{proof}
	Recall that $r>0$ is such that $\frI_X$ and $\frI_X^2$ are $r$-regular.
	By \cite[Theorem 3.7]{BayerMumford1993}, we can assume that $r \leq (4\delta_X)^{2^{N-1}}$.
	Hence, our result follows from Theorem \ref{theoremproperestimates} above.
\end{proof}

\bibliographystyle{mcleanalpha}
\bibliography{references}

\end{document}